\documentclass{amsart}
\usepackage{amssymb}
\usepackage{amsfonts}

\newtheorem{theorem}{Theorem}
\theoremstyle{plain}

\newtheorem{conclusion}[theorem]{Conclusion}

\newtheorem{conjecture}[theorem]{Conjecture}

\newtheorem{definition}[theorem]{Definition}

\newtheorem{lemma}[theorem]{Lemma}
\newtheorem{notation}[theorem]{Notation}
\newtheorem{problem}[theorem]{Problem}
\newtheorem{proposition}[theorem]{Proposition}
\newtheorem{remark}[theorem]{Remark}

\numberwithin{equation}{section}
\input{tcilatex}

\begin{document}
\title[$T1$ theorem]{The two weight $T1$ theorem for fractional Riesz
transforms when one measure is supported on a curve}
\author[E.T. Sawyer]{Eric T. Sawyer}
\address{ Department of Mathematics \& Statistics, McMaster University, 1280
Main Street West, Hamilton, Ontario, Canada L8S 4K1 }
\email{sawyer@mcmaster.ca}
\thanks{Research supported in part by NSERC}
\author[C.-Y. Shen]{Chun-Yen Shen}
\address{ Department of Mathematics \\
National Central University \\
Chungli, 32054, Taiwan }
\email{chunyshen@gmail.com}
\thanks{C.-Y. Shen supported in part by the NSC, through grant
NSC102-2115-M-008-015-MY2}
\author[I. Uriarte-Tuero]{Ignacio Uriarte-Tuero}
\address{ Department of Mathematics \\
Michigan State University \\
East Lansing MI }
\email{ignacio@math.msu.edu}
\thanks{ I. Uriarte-Tuero has been partially supported by grants DMS-1056965
(US NSF), MTM2010-16232, MTM2009-14694-C02-01 (Spain), and a Sloan
Foundation Fellowship. }
\date{September 19, 2015}

\begin{abstract}
Let $\sigma $ and $\omega $ be locally finite positive Borel measures on $%
\mathbb{R}^{n}$. We assume that at least one of the two measures $\sigma $
and $\omega $ is supported on a regular $C^{1,\delta }$ curve in $\mathbb{R}%
^{n}$. Let $\mathbf{R}^{\alpha ,n}$\ be the $\alpha $-fractional Riesz
transform vector on $\mathbb{R}^{n}$. We prove the $T1$ theorem for $\mathbf{%
R}^{\alpha ,n}$: namely that $\mathbf{R}^{\alpha ,n}$ is bounded from $%
L^{2}\left( \sigma \right) $ to $L^{2}\left( \omega \right) $ \emph{if and
only if} the $\mathcal{A}_{2}^{\alpha }$ conditions with holes hold, the
punctured $A_{2}^{\alpha }$ conditions hold, and the cube testing condition
for $\mathbf{R}^{\alpha ,n}$\textbf{\ }and its dual both hold. The special
case of the Cauchy transform, $n=2$ and $\alpha =1$, when the curve is a
line or circle, was established by Lacey, Sawyer, Shen, Uriarte-Tuero and
Wick in \cite{LaSaShUrWi}.

This $T1$ theorem represents essentially the most general $T1$ theorem
obtainable by methods of energy reversal. More precisely, for the
pushforwards of the measures $\sigma $ and $\omega $, under a change of
variable to straighten out the curve to a line, we use reversal of energy to
prove that the quasienergy conditions in \cite{SaShUr5} are implied by the $%
\mathcal{A}_{2}^{\alpha }$ with holes, punctured $A_{2}^{\alpha }$, and
quasicube testing conditions for $\mathbf{R}^{\alpha ,n}$. Then we apply the
main theorem in \cite{SaShUr5} to deduce the $T1$ theorem above.
\end{abstract}

\maketitle
\tableofcontents

\section{Introduction}

\subsection{A brief history of the $T1$ theorem}

The celebrated $T1$ theorem of David and Journ\'{e} \cite{DaJo}
characterizes those singular integral operators $T$ on $\mathbb{R}^{n}$ that
are bounded on $L^{2}\left( \mathbb{R}^{n}\right) $, and does so in terms of
a weak boundedness property, and the membership of the two functions $T%
\mathbf{1}$ and $T^{\ast }\mathbf{1}$ in the space of bounded mean
oscillation,%
\begin{eqnarray*}
\left\Vert T\mathbf{1}\right\Vert _{BMO\left( \mathbb{R}^{n}\right) }
&\lesssim &\left\Vert \mathbf{1}\right\Vert _{L^{\infty }\left( \mathbb{R}%
^{n}\right) }=1, \\
\left\Vert T^{\ast }\mathbf{1}\right\Vert _{BMO\left( \mathbb{R}^{n}\right)
} &\lesssim &\left\Vert \mathbf{1}\right\Vert _{L^{\infty }\left( \mathbb{R}%
^{n}\right) }=1.
\end{eqnarray*}%
These latter conditions are actually the following \emph{testing conditions}
in disguise,%
\begin{eqnarray*}
\left\Vert T\mathbf{1}_{Q}\right\Vert _{L^{2}\left( \mathbb{R}^{n}\right) }
&\lesssim &\left\Vert \mathbf{1}_{Q}\right\Vert _{L^{2}\left( \mathbb{R}%
^{n}\right) }=\sqrt{\left\vert Q\right\vert }\ , \\
\left\Vert T^{\ast }\mathbf{1}_{Q}\right\Vert _{L^{2}\left( \mathbb{R}%
^{n}\right) } &\lesssim &\left\Vert \mathbf{1}_{Q}\right\Vert _{L^{2}\left( 
\mathbb{R}^{n}\right) }=\sqrt{\left\vert Q\right\vert }\ ,
\end{eqnarray*}%
tested over all indicators of cubes $Q$ in $\mathbb{R}^{n}$ for both $T$ and
its dual operator $T^{\ast }$. This theorem was the culmination of decades
of investigation into the nature of cancellation conditions required for
boundedness of singular integrals\footnote{%
See e.g. chapter VII of Stein \cite{Ste} and the references given there for
a historical background.}.

A parallel thread of investigation culminated in the theorem of Coifman and
Fefferman\footnote{%
See e.g. chapter V of \cite{Ste} and the references given there for the long
history of this investigation, in which the celebrated theorem of Hunt,
Muckenhoupt and Wheeden played a critical role.} that characterizes those
nonnegative weights $w$ on $\mathbb{R}^{n}$ for which all of the `nicest' of
the $L^{2}\left( \mathbb{R}^{n}\right) $ bounded singular integrals $T$
above are bounded on weighted spaces $L^{2}\left( \mathbb{R}^{n};w\right) $,
and does so in terms of the $A_{2}$ condition of Muckenhoupt,%
\begin{equation*}
\left( \frac{1}{\left\vert Q\right\vert }\int_{Q}w\left( x\right) dx\right)
\left( \frac{1}{\left\vert Q\right\vert }\int_{Q}\frac{1}{w\left( x\right) }%
dx\right) \lesssim 1\ ,
\end{equation*}%
taken over all cubes $Q$ in $\mathbb{R}^{n}$. This condition is also a
testing condition in disguise, namely it is a consequence of%
\begin{equation*}
\left\Vert T\left( \mathbf{s}_{Q}\frac{1}{w}\right) \right\Vert
_{L^{2}\left( \mathbb{R}^{n};w\right) }\lesssim \left\Vert \mathbf{s}%
_{Q}\right\Vert _{L^{2}\left( \mathbb{R}^{n};\frac{1}{w}\right) }\ ,
\end{equation*}%
tested over all `indicators with tails' $\mathbf{s}_{Q}\left( x\right) =%
\frac{\ell \left( Q\right) }{\ell \left( Q\right) +\left\vert
x-c_{Q}\right\vert }$ of cubes $Q$ in $\mathbb{R}^{n}$.

A natural synthesis of these two results leads to the `two weight' question
of which pairs of weights $\left( \sigma ,\omega \right) $ have the property
that nice singular integrals are bounded from $L^{2}\left( \mathbb{R}%
^{n};\sigma \right) $ to $L^{2}\left( \mathbb{R}^{n};\omega \right) $. The
simplest (nontrivial) singular integral of all is the Hilbert transform $%
Hf\left( x\right) =\int_{\mathbb{R}}\frac{f\left( y\right) }{y-x}dy$ on the
real line, and Nazarov, Treil and Volberg formulated the two weight question
for the Hilbert transform \cite{Vol}, that in turn led to the NTV conjecture:

\begin{conjecture}
\cite{Vol} The Hilbert transform is bounded from $L^{2}\left( \mathbb{R}%
^{n};\sigma \right) $ to $L^{2}\left( \mathbb{R}^{n};\omega \right) $, i.e.%
\begin{equation*}
\left\Vert H\left( f\sigma \right) \right\Vert _{L^{2}\left( \mathbb{R}%
^{n};\omega \right) }\lesssim \left\Vert f\right\Vert _{L^{2}\left( \mathbb{R%
}^{n};\sigma \right) },\ \ \ \ \ f\in L^{2}\left( \mathbb{R}^{n};\sigma
\right) ,
\end{equation*}%
if and only if the two weight $A_{2}$ condition with tails holds,%
\begin{equation*}
\left( \frac{1}{\left\vert Q\right\vert }\int_{Q}\mathbf{s}_{Q}^{2}d\omega
\left( x\right) \right) \left( \frac{1}{\left\vert Q\right\vert }\int_{Q}%
\mathbf{s}_{Q}^{2}d\sigma \left( x\right) \right) \lesssim 1\ ,
\end{equation*}%
for all cubes $Q$, and the two testing conditions hold,%
\begin{eqnarray*}
\left\Vert H\mathbf{1}_{Q}\sigma \right\Vert _{L^{2}\left( \mathbb{R}%
^{n};\omega \right) } &\lesssim &\left\Vert \mathbf{1}_{Q}\right\Vert
_{L^{2}\left( \mathbb{R}^{n};\sigma \right) }=\sqrt{\left\vert Q\right\vert
_{\sigma }}\ , \\
\left\Vert H^{\ast }\mathbf{1}_{Q}\omega \right\Vert _{L^{2}\left( \mathbb{R}%
^{n};\sigma \right) } &\lesssim &\left\Vert \mathbf{1}_{Q}\right\Vert
_{L^{2}\left( \mathbb{R}^{n};\omega \right) }=\sqrt{\left\vert Q\right\vert
_{\omega }}\ ,
\end{eqnarray*}%
for all cubes $Q$.
\end{conjecture}

In a groundbreaking series of papers including \cite{NTV1},\cite{NTV2} and 
\cite{NTV3}, Nazarov, Treil and Volberg used weighted Haar decompositions
with random grids, introduced their `pivotal' condition, and proved the
above conjecture under the side assumption that the pivotal condition held.
Subsequently, in joint work of two of us, Sawyer and Uriarte-Tuero, with
Lacey \cite{LaSaUr2}, it was shown that the pivotal condition was not
necessary in general, a necessary `energy' condition was introduced as a
substitute, and a hybrid merging of these two conditions was shown to be
sufficient for use as a side condition. Eventually, these three authors with
Shen established the NTV conjecture in a two part paper; Lacey, Sawyer, Shen
and Uriarte-Tuero \cite{LaSaShUr3} and Lacey \cite{Lac}. A key ingredient in
the proof was an `energy reversal' phenomenom enabled by the Hilbert
transform kernel equality%
\begin{equation*}
\frac{1}{y-x}-\frac{1}{y-x^{\prime }}=\frac{x-x^{\prime }}{\left( y-x\right)
\left( y-x^{\prime }\right) },
\end{equation*}%
having the remarkable property that the denominator on the right hand side
remains \emph{positive} for all $y$ outside the smallest interval containing
both $x$ and $x^{\prime }$. This proof of the NTV conjecture was given in
the special case that the weights $\sigma $ and $\omega $ had no point
masses in common, largely to avoid what were then thought to be technical
issues. However, these issues turned out to be considerably more
interesting, and this final assumption of no common point masses was removed
shortly after by Hyt\"{o}nen \cite{Hyt2}, who also simplified some aspects
of the proof.

At this juncture, attention naturally turned to the analogous two weight
inequalities for higher dimensional singular integrals, as well as $\alpha $%
-fractional singular integrals such as the Cauchy transform in the plane. In
a long paper \cite{SaShUr4}, begun on the \textit{arXiv} in 2013, the
authors introduced the appropriate notions of Poisson kernel to deal with
the $A_{2}^{\alpha }$ condition on the one hand, and the $\alpha $-energy
condition on the other hand (unlike for the Hilbert transform, these two
Poisson kernels differ in general). The main result of that paper
established the $T1$ theorem for `elliptic' vectors of singular integrals
under the side assumption that an energy condition and its dual held, thus
identifying the \emph{culprit} in higher dimensions as the energy conditions
(see also \cite{SaShUr5} where the restriction to no common point masses was
removed). A general $T1$ conjecture is this.

\begin{conjecture}
Let $\mathbf{T}^{\alpha ,n}$ denote an elliptic vector of standard $\alpha $%
-fractional singular integrals in $\mathbb{R}^{n}$. Then $\mathbf{T}^{\alpha
,n}$ is bounded from $L^{2}\left( \mathbb{R}^{n};\sigma \right) $ to $%
L^{2}\left( \mathbb{R}^{n};\omega \right) $, i.e.%
\begin{equation*}
\left\Vert \mathbf{T}^{\alpha ,n}\left( f\sigma \right) \right\Vert
_{L^{2}\left( \mathbb{R}^{n};\omega \right) }\lesssim \left\Vert
f\right\Vert _{L^{2}\left( \mathbb{R}^{n};\sigma \right) },\ \ \ \ \ f\in
L^{2}\left( \mathbb{R}^{n};\sigma \right) ,
\end{equation*}%
if and only if the two one-tailed $\mathcal{A}_{2}^{\alpha }$ conditions
with holes hold, the punctured $A_{2}^{\alpha }$ conditions hold, and the
two testing conditions hold,%
\begin{eqnarray*}
\left\Vert \mathbf{T}^{\alpha ,n}\mathbf{1}_{Q}\sigma \right\Vert
_{L^{2}\left( \mathbb{R}^{n};\omega \right) } &\lesssim &\left\Vert \mathbf{1%
}_{Q}\right\Vert _{L^{2}\left( \mathbb{R}^{n};\sigma \right) }=\sqrt{%
\left\vert Q\right\vert _{\sigma }}\ , \\
\left\Vert \mathbf{T}^{\alpha ,n,\func{dual}}\mathbf{1}_{Q}\omega
\right\Vert _{L^{2}\left( \mathbb{R}^{n};\sigma \right) } &\lesssim
&\left\Vert \mathbf{1}_{Q}\right\Vert _{L^{2}\left( \mathbb{R}^{n};\omega
\right) }=\sqrt{\left\vert Q\right\vert _{\omega }}\ ,
\end{eqnarray*}%
for all cubes $Q$ in $\mathbb{R}^{n}$ (whose sides need not be parallel to
the coordinate axes).
\end{conjecture}

A positive resolution to this conjecture could have implications for a
number of problems that are higher dimensional analogues of those connected
to the Hilbert transform (see e.g. \cite{Vol}, \cite{NiTr}, \cite{NaVo}, 
\cite{VoYu}, \cite{PeVoYu}, \cite{PeVoYu1}, \cite{IwMa}, \cite{LaSaUr}, \cite%
{AsGo} and \cite{AsZi}).

In view of the aforementioned main result in \cite{SaShUr5}, the following
conjecture is stronger.

\begin{conjecture}
Let $\mathbf{T}^{\alpha ,n}$ denote an elliptic vector of standard $\alpha $%
-fractional singular integrals in $\mathbb{R}^{n}$. If $\mathbf{T}^{\alpha
,n}$ is bounded from $L^{2}\left( \mathbb{R}^{n};\sigma \right) $ to $%
L^{2}\left( \mathbb{R}^{n};\omega \right) $, then the energy conditions hold
as defined in Definition \ref{energy condition} below.
\end{conjecture}

At the time of this writing, it is not known if the energy conditions are
necessary or not in general. The paper \cite{LaWi2} by Lacey and Wick
overlaps \cite{SaShUr4} to some extent.

While no counterexamples have yet been discovered to the energy conditions,
there are some cases in which they have been proved to hold. Of course, the
energy conditions hold for the Hilbert transform on the line \cite{LaSaUr2},
and in recent joint work with M. Lacey and B. Wick, the five of us have
established that the energy conditions hold for the Cauchy transform in the
plane in the special case where one of the measures is supported on either a
straight line or a circle, thus proving the $T1$ theorem in this case. The
key to this result was an extension of the energy reversal phenomenon for
the Hilbert transform to the setting of the Cauchy transform, and here the
one-dimensional nature of the line and circle played a critical role. In
particular, a special decomposition of a $2$-dimensional measure into `end'
and `side' pieces played a crucial role, and was in fact discovered
independently in both the initial version of this paper and in \cite{LaWi}.

In this paper, we extend the $T1$ theorem to the setting where one of the
measures is supported on a H\"{o}lder continuously differentiable curve (in
higher dimensions). This result seems to represent the best possible $T1$
theorem that can be obtained from the methods of energy reversal, and
requires a number of new ideas, especially of a geometric nature, as opposed
to the more algebraic `corona' ideas developed for the solution to the NTV
conjecture. In particular, changes of variable are made to straighten
sufficiently small pieces of the curve to a line, and the resulting operator
norms, $A_{2}^{\alpha }$ conditions, and testing condition constants are
tracked under these changes of variable. This tracking presents significant
subtleties, especially for the testing constants, which require appropriate
tangent plane approximations to the phase function of the testing kernel.
Further effort is then needed to control the testing conditions associated
with these pieces by the testing conditions we assume for the curve in the
first place. In particular, a `localized triple testing' condition is
derived that enables the reduction of testing conditions for small pieces of
transformed measure on a line to the testing conditions for the global
measures. Yet another complication arises here in the use of quasicubes in
the proof - dictated by pushforwards of ordinary cubes - and this requires
the new notion of `$L$-transverse' and its properties in order to control
the intersections of quasicubes with lines.

We now give a more precise description of what is in this paper and its
relation to the literature.

\subsection{Statement of results}

In \cite{SaShUr5} (see also \cite{SaShUr} and \cite{SaShUr4} for special
cases), under a side assumption that certain \emph{energy conditions} hold,\
the authors show in particular that the two weight inequality%
\begin{equation}
\left\Vert \mathbf{R}^{\alpha ,n}\left( f\sigma \right) \right\Vert
_{L^{2}\left( \omega \right) }\lesssim \left\Vert f\right\Vert _{L^{2}\left(
\sigma \right) },  \label{2 weight}
\end{equation}%
for the vector of Riesz transforms $\mathbf{R}^{\alpha ,n}$ in $\mathbb{R}%
^{n}$ (with $0\leq \alpha <n$) holds if and only if the $\mathcal{A}%
_{2}^{\alpha }$ conditions with holes hold, the punctured $A_{2}^{\alpha ,%
\limfunc{punct}}$ conditions hold, the quasicube testing conditions hold,
and the quasiweak boundedness property holds. Here a quasicube is a globally
biLipschitz image of a usual cube. Precise definitions of all terms used
here are given in the next section. It is not known at the time of this
writing whether or not these or any other energy conditions are necessary
for \emph{any} vector $\mathbf{T}^{\alpha ,n}$ of fractional singular
integrals in $\mathbb{R}^{n}$ with $n\geq 2$, apart from the trivial case of
positive operators. In particular there are no known counterexamples. We
also showed in \cite{SaShUr2} and \cite{SaShUr3} that the technique of
reversing energy, typically used to prove energy conditions, fails
spectacularly in higher dimension (and we thank M. Lacey for showing us this
failure for the Cauchy transform with the circle measure). See also the
counterexamples for the fractional Riesz transforms in \cite{LaWi2}.

The purpose of this paper is to show that if $\sigma $ and $\omega $ are
locally finite positive Borel measures (possibly having common point
masses), and at least \emph{one} of the two measures $\sigma $ and $\omega $
is supported on a line in $\mathbb{R}^{n}$, or on a regular $C^{1,\delta }$
curve in $\mathbb{R}^{n}$, then the energy conditions are indeed necessary
for boundedness of the fractional Riesz transform $\mathbf{R}^{\alpha ,n}$,
and hence that a $T1$ theorem holds for $\mathbf{R}^{\alpha ,n}$ (see
Theorem \ref{general T1} below). Just after the first version of this paper
appeared on the \emph{arXiv}, M. Lacey and B. Wick \cite{LaWi} independently
posted a similar result for the special case of the Cauchy transform in the
plane where one measure is supported on a line or a circle, and the five
authors have combined on the paper \cite{LaSaShUrWi} in this setting.

The vector of $\alpha $-fractional Riesz transforms is given by%
\begin{equation*}
\mathbf{R}^{\alpha ,n}=\left\{ R_{\ell }^{\alpha ,n}:1\leq \ell \leq
n\right\} ,\ \ \ \ \ 0\leq \alpha <n,
\end{equation*}%
where the component Riesz transforms $R_{\ell }^{\alpha ,n}$ are the
convolution fractional singular integrals $R_{\ell }^{\alpha ,n}f\equiv
K_{\ell }^{\alpha ,n}\ast f$ with odd kernel defined by%
\begin{equation*}
K_{\ell }^{\alpha ,n}\left( w\right) \equiv c_{\alpha ,n}\frac{w^{\ell }}{%
\left\vert w\right\vert ^{n+1-\alpha }}.
\end{equation*}%
Finally, we remark that the $T1$ theorem under this geometric condition has
application to the weighted discrete Hilbert transform $H_{\left( \Gamma
,v\right) }$ when the sequence $\Gamma $ is supported on an appropriate $%
C^{1,\delta }$ curve in the complex plane. See \cite{BeMeSe} where $%
H_{\left( \Gamma ,v\right) }$ is essentially the Cauchy transform with $n=2$
and $\alpha =1$.

We now recall a special case of our main two weight theorem from \cite%
{SaShUr5} which plays a key role here - see also \cite{SaShUr} and \cite%
{SaShUr4} for earlier versions. Let $\mathcal{P}^{n}$ denote the collection
of all cubes in $\mathbb{R}^{n}$ with sides parallel to the coordinate axes,
and denote by $\mathcal{D}^{n}\subset \mathcal{P}^{n}$ a dyadic grid in $%
\mathbb{R}^{n}$. The side conditions $\mathcal{A}_{2}^{\alpha }$, $\mathcal{A%
}_{2}^{\alpha ,\func{dual}}$, $A_{2}^{\alpha ,\limfunc{punct}}$, $%
A_{2}^{\alpha ,\limfunc{punct},\func{dual}}$, $\mathcal{E}_{\alpha }$ and $%
\mathcal{E}_{\alpha }^{\func{dual}}$ depend only only on the measure pair $%
\left( \sigma ,\omega \right) $, while the necessary conditions $\mathfrak{T}%
_{\mathbf{R}^{\alpha ,n}}$, $\mathfrak{T}_{\mathbf{R}^{\alpha ,n}}^{\func{%
dual}}$ and $\mathcal{WBP}_{\mathbf{R}^{\alpha ,n}}$ depend on the measure
pair $\left( \sigma ,\omega \right) $ as well as the singular operator $%
\mathbf{R}_{\sigma }^{\alpha ,n}$. These conditions will be explained below.
For convenience in notation, we use Fraktur font for A to denote,%
\begin{equation*}
\mathfrak{A}_{2}^{\alpha }\equiv \mathcal{A}_{2}^{\alpha }+\mathcal{A}%
_{2}^{\alpha ,\func{dual}}+A_{2}^{\alpha ,\limfunc{punct}}+A_{2}^{\alpha ,%
\limfunc{punct},\func{dual}},
\end{equation*}%
or when the measure pair is important, 
\begin{equation}
\mathfrak{A}_{2}^{\alpha }\left( \sigma ,\omega \right) \equiv \mathcal{A}%
_{2}^{\alpha }\left( \sigma ,\omega \right) +\mathcal{A}_{2}^{\alpha ,\func{%
dual}}\left( \sigma ,\omega \right) +A_{2}^{\alpha ,\limfunc{punct}}\left(
\sigma ,\omega \right) +A_{2}^{\alpha ,\limfunc{punct},\func{dual}}\left(
\sigma ,\omega \right) .  \label{big A2}
\end{equation}

\begin{notation}
In order to avoid confusion with the use of $\ast $ for pullbacks and
pushforwards of maps, we will use the superscript $\func{dual}$ in place of $%
\ast $ to denote `dual conditions' throughout this paper.
\end{notation}

\begin{theorem}
\label{T1 theorem}Suppose that $\mathbf{R}^{\alpha ,n}$ is the vector of $%
\alpha $-fractional Riesz transforms in $\mathbb{R}^{n}$, and that $\omega $
and $\sigma $ are positive locally finite Borel measures on $\mathbb{R}^{n}$
(possibly having common point masses). Set $\mathbf{R}_{\sigma }^{\alpha
,n}f=\mathbf{R}^{\alpha ,n}\left( f\sigma \right) $ for any smooth
truncation of $\mathbf{R}^{\alpha ,n}$. Let $\Omega :\mathbb{R}%
^{n}\rightarrow \mathbb{R}^{n}$ be a globally biLipschitz map.

\begin{enumerate}
\item Suppose $0\leq \alpha <n$ and that $\gamma \geq 2$ is given. Then the
vector Riesz transform $\mathbf{R}_{\sigma }^{\alpha ,n}$ is bounded from $%
L^{2}\left( \sigma \right) $ to $L^{2}\left( \omega \right) $, i.e. 
\begin{equation}
\left\Vert \mathbf{R}_{\sigma }^{\alpha ,n}f\right\Vert _{L^{2}\left( \omega
\right) }\leq \mathfrak{N}_{\mathbf{R}^{\alpha ,n}}\left\Vert f\right\Vert
_{L^{2}\left( \sigma \right) },  \label{two weight}
\end{equation}%
uniformly in smooth truncations of $\mathbf{R}^{\alpha ,n}$, and moreover%
\begin{equation}
\mathfrak{N}_{\mathbf{R}^{\alpha ,n}}\leq C_{\alpha }\left( \sqrt{\mathfrak{A%
}_{2}^{\alpha }}+\mathcal{E}_{\alpha }+\mathcal{E}_{\alpha }^{\func{dual}}+%
\mathfrak{T}_{\mathbf{R}^{\alpha ,n}}+\mathfrak{T}_{\mathbf{R}^{\alpha ,n}}^{%
\func{dual}}+\mathcal{WBP}_{\mathbf{R}^{\alpha ,n}}\right) ,
\label{norm bound}
\end{equation}%
provided that the two dual $\mathcal{A}_{2}^{\alpha }$ conditions with holes
hold, the punctured dual $A_{2}^{\alpha ,\limfunc{punct}}$ conditions hold,
and the two dual quasicube testing conditions for $\mathbf{R}^{\alpha ,n}$
hold, the quasiweak boundedness property for $\mathbf{R}^{\alpha ,n}$ holds
for a sufficiently large constant $C$ depending on the goodness parameter $%
\mathbf{r}$, and provided that the two dual quasienergy conditions $\mathcal{%
E}_{\alpha }+\mathcal{E}_{\alpha }^{\func{dual}}<\infty $ hold uniformly
over all dyadic grids $\mathcal{D}^{n}$, and where the goodness parameters $%
\mathbf{r}$ and $\varepsilon $ implicit in the definition of $\mathcal{M}_{%
\mathbf{r}-\limfunc{deep}}^{\ell }\left( K\right) $ below are fixed
sufficiently large and small respectively depending on $n$, $\alpha $ and $%
\gamma $. Here $\mathfrak{N}_{\mathbf{R}^{\alpha ,n}}=\mathfrak{N}_{\mathbf{R%
}^{\alpha ,n}}\left( \sigma ,\omega \right) $ is the least constant in (\ref%
{two weight}).

\item Conversely, suppose $0\leq \alpha <n$ and that the Riesz transform
vector $\mathbf{R}_{\sigma }^{\alpha ,n}$ is bounded from $L^{2}\left(
\sigma \right) $ to $L^{2}\left( \omega \right) $, 
\begin{equation*}
\left\Vert \mathbf{R}_{\sigma }^{\alpha ,n}f\right\Vert _{L^{2}\left( \omega
\right) }\leq \mathfrak{N}_{\mathbf{R}^{\alpha ,n}}\left\Vert f\right\Vert
_{L^{2}\left( \sigma \right) }.
\end{equation*}%
Then the testing conditions and weak boundedness property hold for $\mathbf{R%
}^{\alpha ,n}$, the fractional $\mathcal{A}_{2}^{\alpha }$ conditions with
holes hold, and the punctured dual $A_{2}^{\alpha ,\limfunc{punct}}$
conditions hold, and moreover,%
\begin{equation*}
\sqrt{\mathfrak{A}_{2}^{\alpha }}+\mathfrak{T}_{\mathbf{R}^{\alpha ,n}}+%
\mathfrak{T}_{\mathbf{R}^{\alpha ,n}}^{\func{dual}}+\mathcal{WBP}_{\mathbf{R}%
^{\alpha ,n}}\leq C\mathfrak{N}_{\mathbf{R}^{\alpha ,n}}.
\end{equation*}
\end{enumerate}
\end{theorem}

\begin{problem}
It is an open question whether or not the energy conditions are necessary
for boundedness of $\mathbf{R}_{\sigma }^{\alpha ,n}$. See \cite{SaShUr3}
for a failure of \emph{energy reversal} in higher dimensions - such an
energy reversal was used in dimension $n=1$ to prove the necessity of the
energy condition for the Hilbert transform.
\end{problem}

\begin{remark}
\label{surgery}In \cite{LaWi2}, M. Lacey and B. Wick use the NTV technique
of surgery to show that an \emph{expectation} over grids of an analogue of
the weak boundedness property for the Riesz transform vector $\mathbf{R}%
^{\alpha ,n}$ is controlled by the $\mathcal{A}_{2}^{\alpha }$ and cube
testing conditions, together with a small multiple of the operator norm.
They then claim a $T1$ theorem with a side condition of uniformly full
dimensional measures, using independent grids corresponding to each measure,
resulting in an elimination of the weak boundedness property as a condition.
In any event, the weak boundedness property is always necessary for the norm
inequality, and as such can be viewed as a weak close cousin of the testing
conditions.
\end{remark}

The main result of this paper is the $T1$ theorem for a measure supported on
a regular $C^{1,\delta }$ curve. We point out that the cubes occurring in
the testing conditions in the following theorem are of course restricted to
those that intersect the curve, otherwise the integrals vanish, and include
not only those in $\mathcal{P}^{n}$ with sides parallel to the axes, but
also those in $\mathcal{Q}^{n}$ consisting of all rotations of the cubes in $%
\mathcal{P}^{n}$:%
\begin{eqnarray}
\mathfrak{T}_{\mathbf{R}^{\alpha ,n}}^{2} &\equiv &\sup_{Q\in \mathcal{Q}%
^{n}}\frac{1}{\left\vert Q\right\vert _{\sigma }}\int_{Q}\left\vert \mathbf{R%
}^{\alpha ,n}\left( \mathbf{1}_{Q}\sigma \right) \right\vert ^{2}\omega
<\infty ,  \label{usual testing} \\
\left( \mathfrak{T}_{\mathbf{R}^{\alpha ,n}}^{\func{dual}}\right) ^{2}
&\equiv &\sup_{Q\in \mathcal{Q}^{n}}\frac{1}{\left\vert Q\right\vert
_{\omega }}\int_{Q}\left\vert \left( \mathbf{R}^{\alpha ,n}\right) ^{\func{%
dual}}\left( \mathbf{1}_{Q}\omega \right) \right\vert ^{2}\sigma <\infty . 
\notag
\end{eqnarray}%
In the special case considered in \cite{LaSaShUrWi} of the Cauchy transform
in the plane with $\omega $ supported on the unit circle $\mathbb{T}$ or the
real line $\mathbb{R}$, the testing is taken over the smaller collection of
all Carleson squares.

We consider \emph{regular} $C^{1,\delta }$ curves in $\mathbb{R}^{n}$
defined as follows.

\begin{definition}
Suppose $\delta >0$, $I=\left[ a,b\right] $ is a closed interval on the real
line with $-\infty <a<b<\infty $, and that $\Phi :I\rightarrow \mathbb{R}%
^{n} $ is a $C^{1,\delta }$ curve parameterized by arc length. The curve is
one-to-one with the possible exception that $\Phi \left( a\right) =\Phi
\left( b\right) $. We refer to any curve as above as a \emph{regular} $%
C^{1,\delta }$ curve.
\end{definition}

\begin{theorem}
\label{general T1}Let $0\leq \alpha <n$ and suppose $\Phi $ is a regular $%
C^{1,\delta }$ curve. Suppose further that\\*[0.25in]
(\textbf{1}) $\sigma $ and $\omega $ are positive locally finite Borel
measures on $\mathbb{R}^{n}$ (possibly having common point masses), and $%
\omega $ is supported in $\mathcal{L}\equiv \limfunc{range}\Phi $, and\\*[%
0.25in]
(\textbf{2}) $\mathbf{R}^{\alpha ,n}$ is the vector of $\alpha $-fractional
Riesz transforms in $\mathbb{R}^{n}$, and $\mathbf{R}_{\sigma }^{\alpha ,n}f=%
\mathbf{R}^{\alpha ,n}\left( f\sigma \right) $ for any smooth truncation of $%
\mathbf{R}^{\alpha ,n}$.\\*[0.25in]
\textit{Then }$\mathbf{R}_{\sigma }^{\alpha ,n}$ is bounded from $%
L^{2}\left( \sigma \right) $ to $L^{2}\left( \omega \right) $, i.e. 
\begin{equation*}
\left\Vert \mathbf{R}_{\sigma }^{\alpha ,n}f\right\Vert _{L^{2}\left( \omega
\right) }\leq \mathfrak{N}_{\mathbf{R}^{\alpha ,n}}\left\Vert f\right\Vert
_{L^{2}\left( \sigma \right) },
\end{equation*}%
uniformly in smooth truncations of $\mathbf{R}^{\alpha ,n}$, \emph{if and
only if} the two dual $\mathcal{A}_{2}^{\alpha }$ conditions with holes
hold, the punctured dual $A_{2}^{\alpha ,\limfunc{punct}}$ conditions hold,
and the two dual cube testing conditions (\ref{usual testing}) for $\mathbf{R%
}^{\alpha ,n}$ hold. Moreover we have the equivalence%
\begin{equation*}
\mathfrak{N}_{\mathbf{R}^{\alpha ,n}}\approx \sqrt{\mathfrak{A}_{2}^{\alpha }%
}+\mathfrak{T}_{\mathbf{R}^{\alpha ,n}}+\mathfrak{T}_{\mathbf{R}^{\alpha
,n}}^{\func{dual}}.
\end{equation*}
\end{theorem}

\subsection{Techniques}

The remainder of the introduction is devoted to giving an overview of the
techniques and arguments needed to obtain Theorem \ref{general T1} from
Theorem \ref{T1 theorem}. For this we need $\Omega $-quasicubes and
conformal $\alpha $-fractional Riesz transforms $\mathbf{R}_{\Psi }^{\alpha
,n}$ where $\Omega $ is a globally biLipschitz map and $\Psi $ is a $%
C^{1,\delta }$ diffeomorphism of $\mathbb{R}^{n}$. We now describe these
issues in more detail. Let $\Omega :\mathbb{R}^{n}\rightarrow \mathbb{R}^{n}$
be a globally biLipschitz map as defined in Definition \ref{maps} below, and
refer to the images $\Omega Q$ of cubes in $\mathcal{Q}^{n}$ under the map $%
\Omega $ as $\Omega $-quasicubes or simply quasicubes. These $\Omega $%
-quasicubes will often be used in place of cubes in the testing conditions,
energy conditions and weak boundedness property, and we will use $\Omega 
\mathcal{Q}^{n}$ as a superscript to indicate this. For example, the
quasitesting analogue of the usual testing conditions (\ref{usual testing})
is:%
\begin{eqnarray}
\left( \mathfrak{T}_{\mathbf{R}^{\alpha ,n}}^{\Omega \mathcal{Q}^{n}}\right)
^{2} &\equiv &\sup_{Q\in \Omega \mathcal{Q}^{n}}\frac{1}{\left\vert
Q\right\vert _{\sigma }}\int_{Q}\left\vert \mathbf{R}^{\alpha ,n}\left( 
\mathbf{1}_{Q}\sigma \right) \right\vert ^{2}\omega <\infty ,
\label{quasi testing} \\
\left( \mathfrak{T}_{\mathbf{R}^{\alpha ,n}}^{\Omega \mathcal{Q}^{n},\func{%
dual}}\right) ^{2} &\equiv &\sup_{Q\in \Omega \mathcal{Q}^{n}}\frac{1}{%
\left\vert Q\right\vert _{\omega }}\int_{Q}\left\vert \left( \mathbf{R}%
^{\alpha ,n}\right) ^{\func{dual}}\left( \mathbf{1}_{Q}\omega \right)
\right\vert ^{2}\sigma <\infty .  \notag
\end{eqnarray}%
We alert the reader to the fact that different collections of quasicubes
will be considered in the course of proving our theorem. The definitions of
these terms, and the remaining terms used below, will be given precisely in
the next section. When the superscript $\Omega \mathcal{Q}^{n}$ is omitted,
it is understood that the quasicubes are the usual cubes $\mathcal{Q}^{n}$.

The next result shows that the quasienergy conditions are in fact necessary
for boundedness of the Riesz transform vector $\mathbf{R}^{\alpha ,n}$ when
one of the measures is supported on a \emph{line}. In that case, the
quasienergy conditions are even implied by the Muckenhoupt $\mathcal{A}%
_{2}^{\alpha }$ conditions with holes and the quasitesting conditions.
Moreover, the backward tripled quasitesting condition and the quasiweak
boundedness property are implied by the Muckenhoupt with holes and
quasitesting conditions as well, but provided the quasicubes come from a $%
C^{1}$ diffeomorphism and are rotated in an appropriate way.

Finally, in order to obtain the $T1$ theorem when one measure is supported
on a curve, we will need to generalize the fractional Riesz transforms $%
\mathbf{R}^{\alpha ,n}$ that we can consider in this theorem. Consider $\Psi
:\mathbb{R}^{n}\rightarrow \mathbb{R}^{n}$ given by%
\begin{equation*}
\Psi \left( x\right) =\left( x^{1},x^{2}-\psi ^{2}\left( x^{1}\right)
,x^{3}-\psi ^{3}\left( x^{1}\right) ,...,x^{n}-\psi ^{n}\left( x^{1}\right)
\right) =x-\left( 0,\psi \left( x_{1}\right) \right) ,
\end{equation*}%
where $x=\left( x^{1},x^{\prime }\right) $ and 
\begin{equation*}
\psi \left( t\right) =\left( \psi ^{2}\left( t\right) ,\psi ^{3}\left(
t\right) ,...,\psi ^{n}\left( t\right) \right) \in \mathbb{R}^{n-1},
\end{equation*}%
is a $C^{1,\delta }$ function $\psi :\mathbb{R}\rightarrow \mathbb{R}^{n-1}$%
. Let $\mathbf{K}^{\alpha ,n}\left( x,y\right) $ denote the vector Riesz
kernel and define%
\begin{equation*}
\mathbf{K}_{\Psi }^{\alpha ,n}\left( x,y\right) =\frac{\left\vert
y-x\right\vert ^{n+1-\alpha }}{\left\vert \Psi \left( y\right) -\Psi \left(
x\right) \right\vert ^{n+1-\alpha }}\mathbf{K}^{\alpha ,n}\left( x,y\right)
=c_{\alpha ,n}\frac{y-x}{\left\vert \Psi \left( y\right) -\Psi \left(
x\right) \right\vert ^{n+1-\alpha }}.
\end{equation*}

\begin{definition}
\label{def conformal}We refer to the operator $\mathbf{R}_{\Psi }^{\alpha
,n} $ with kernel $\mathbf{K}_{\Psi }^{\alpha ,n}$ as a \emph{conformal} $%
\alpha $-fractional Riesz transform. We also define the factor%
\begin{equation*}
\Gamma _{\Psi }\left( x,y\right) \equiv \frac{\left\vert y-x\right\vert
^{n+1-\alpha }}{\left\vert \Psi \left( y\right) -\Psi \left( x\right)
\right\vert ^{n+1-\alpha }}
\end{equation*}%
to be the conformal factor associated with $\Psi $ and $\mathbf{K}^{\alpha
,n}$.
\end{definition}

\begin{notation}
We emphasize that the $C^{1,\delta }$ diffeomorphism $\Psi $ that appears in
the definition of the conformal $\alpha $-fractional Riesz transform $%
\mathbf{R}_{\Psi }^{\alpha ,n}$ need not have any relation to the
biLipschitz map $\Omega :\mathbb{R}^{n}\rightarrow \mathbb{R}^{n}$ that is
used to define the quasicubes under consideration. On the other hand, we
will have reason to consider $\Psi $-quasicubes as well in connection with
changes of variable.
\end{notation}

We use the tangent line approximations to $\mathbf{R}_{\Psi }^{\alpha ,n}$
having kernels $\mathbf{K}_{\Psi }^{\alpha ,n}\left( x,y\right) \rho _{\eta
,R}^{\alpha }$ where $\rho _{\eta ,R}^{\alpha }$ is defined in the next
section. It is shown in \cite{SaShUr5} (see also \cite{LaSaShUr3} for the
one-dimensional case without holes) that one can replace these tangent line
truncations with any reasonable notion of truncation, including the usual
cutoff truncations.

We now introduce a condition on $\Omega $-quasicubes that plays a role in
deriving the necessity of the tripled testing and weak boundedness
conditions. See Lemma \ref{connected}\ below for the relevant consequences
of this condition. We begin with a collection of `good' rotations $R$ that
take the standard basis $\left\{ \mathbf{e}_{i}\right\} _{i=1}^{n}$ to a
basis $\left\{ R\mathbf{e}_{j}\right\} _{j=1}^{n}$ in which no unit vector $R%
\mathbf{e}_{j}$ is too close to any unit vector $\mathbf{e}_{i}$.

Let $\mathfrak{R}^{n}$ denote the group of rotations in $\mathbb{R}^{n}$.
For $\mathbf{e}\in \mathbb{S}^{n-1}$ and $0<\eta <1$ let%
\begin{equation*}
\digamma _{\mathbf{e,}\eta }\equiv \left\{ R\in \mathfrak{R}^{n}:\left\vert
\left\langle R\mathbf{e},\mathbf{e}_{k}\right\rangle \right\vert \leq \eta 
\text{ for }1\leq k\leq n\right\} .
\end{equation*}%
Note that the condition $\digamma _{\mathbf{e,}\eta }\neq \emptyset $ is
independent of the unit vector $\mathbf{e}$, and depends only on $\eta $, by
transitivity of rotations. Fix $\eta =\eta _{n}\in \left( 0,1\right) $ so
that $\digamma _{\mathbf{e,}\eta }\neq \emptyset $ for all $\mathbf{e}\in 
\mathbb{S}^{n-1}$ (this requires $\eta _{n}\geq \frac{1}{\sqrt{n}}$).

\begin{definition}
\label{L-trans}Let $L$ be a line in $\mathbb{R}^{n}$. A $C^{1}$
diffeomorphism $\Omega :\mathbb{R}^{n}\rightarrow \mathbb{R}^{n}$ is $L$-%
\emph{transverse} if%
\begin{equation*}
\left\Vert D\Omega ^{-1}-R\right\Vert _{\infty }<\frac{1-\eta }{4}
\end{equation*}%
for some $R\in \digamma _{\mathbf{e}_{L},\eta }$ where $\mathbf{e}_{L}$ is a
unit vector in the direction of $L$.
\end{definition}

\begin{theorem}
\label{main}Fix a collection of $\Omega $-quasicubes. Let $\sigma $ and $%
\omega $ be locally finite positive Borel measures on $\mathbb{R}^{n}$
(possibly having common point masses). Suppose that $\mathbf{R}_{\Psi
}^{\alpha ,n}$ is a conformal $\alpha $-fractional Riesz transform with $%
0\leq \alpha <n$, and where $\Psi $ is a $C^{1,\delta }$ diffeomorphism with 
$\Psi \left( x\right) =x-\left( 0,\psi \left( x_{1}\right) \right) $ where%
\begin{equation}
\left\Vert D\psi \right\Vert _{\infty }<\frac{1}{8n}\left( 1-\frac{\alpha }{n%
}\right) .  \label{small Lip}
\end{equation}%
Impose the tangent line truncations for $\mathbf{R}_{\Psi }^{\alpha ,n}$ in
the $\Omega $-quasitesting conditions. If the measure $\omega $ is supported
on a line $L$, then%
\begin{equation*}
\mathcal{E}_{\alpha }^{\Omega \mathcal{Q}^{n}}\lesssim \sqrt{\mathcal{A}%
_{2}^{\alpha }}+\mathfrak{T}_{\mathbf{R}_{\Psi }^{\alpha ,n}}^{\Omega 
\mathcal{Q}^{n}}\text{ and }\mathcal{E}_{\alpha }^{\Omega \mathcal{Q}^{n},%
\func{dual}}\lesssim \sqrt{\mathcal{A}_{2}^{\alpha ,\func{dual}}}+\mathfrak{T%
}_{\mathbf{R}_{\Psi }^{\alpha ,n}}^{\Omega \mathcal{Q}^{n},\func{dual}}\ .
\end{equation*}%
If in addition $\Omega $ is a $C^{1}$ diffeomorphism and $L$-transverse, then%
\begin{equation*}
\mathcal{WBP}_{\mathbf{R}_{\Psi }^{\alpha ,n}}^{\Omega \mathcal{P}%
^{n}}\lesssim \mathfrak{T}_{\mathbf{R}_{\Psi }^{\alpha ,n}}^{\Omega \mathcal{%
P}^{n},\limfunc{triple},\func{dual}}\lesssim \mathfrak{T}_{\mathbf{R}_{\Psi
}^{\alpha ,n}}^{\Omega \mathcal{P}^{n},\func{dual}}+\sqrt{\mathcal{A}%
_{2}^{\alpha }}+\sqrt{\mathcal{A}_{2}^{\alpha ,\func{dual}}}\ .
\end{equation*}
\end{theorem}

\begin{remark}
We restrict Theorem \ref{main} to conformal \textbf{Riesz} transforms $%
\mathbf{R}_{\Psi }^{\alpha ,n}$ in order to exploit the special property
that for $j\geq 2$, the scalar transforms $R_{j}^{\alpha ,n}$ and $\left(
R_{\Psi }^{\alpha ,n}\right) _{j}$ behave like a Poisson operator when
acting on a measure supported on the $x_{1}$-axis. This property is not
shared by higher order Riesz transforms, such as the Beurling transform in
the plane, and this accounts for our failure to treat such singular
integrals at this time. The restriction to \textbf{conformal} Riesz
transforms $\mathbf{R}_{\Psi }^{\alpha ,n}$ is dictated by the reversal of
energy that is possible for these special transforms when the phase of the
singular integral is $y-x$ and one of the measures is supported on a line.
\end{remark}

Since the conformal factor $\Gamma _{\Psi }\left( x,y\right) $ in Definition %
\ref{def conformal} satisfies the estimates 
\begin{eqnarray}
\frac{1}{C} &\leq &\Gamma \left( x,y\right) \leq C,
\label{sizeandsmoothness Gamma} \\
\left\vert \nabla \Gamma \left( x,y\right) \right\vert &\leq &C\left\vert
x-y\right\vert ^{-1},  \notag \\
\left\vert \nabla \Gamma \left( x,y\right) -\nabla \Gamma \left( x^{\prime
},y\right) \right\vert &\leq &C\left( \frac{\left\vert x-x^{\prime
}\right\vert }{\left\vert x-y\right\vert }\right) ^{\delta }\left\vert
x-y\right\vert ^{-1},\ \ \ \ \ \frac{\left\vert x-x^{\prime }\right\vert }{%
\left\vert x-y\right\vert }\leq \frac{1}{2},  \notag \\
\left\vert \nabla \Gamma \left( x,y\right) -\nabla \Gamma \left( x,y^{\prime
}\right) \right\vert &\leq &C\left( \frac{\left\vert y-y^{\prime
}\right\vert }{\left\vert x-y\right\vert }\right) ^{\delta }\left\vert
x-y\right\vert ^{-1},\ \ \ \ \ \frac{\left\vert y-y^{\prime }\right\vert }{%
\left\vert x-y\right\vert }\leq \frac{1}{2},  \notag
\end{eqnarray}%
it is easy to see from the product rule that the conformal fractional Riesz
transforms are standard fractional singular integrals in the sense used in 
\cite{SaShUr5}. Since they are also strongly elliptic as in \cite{SaShUr5},
it follows from the main theorem in \cite{SaShUr5} that

\begin{conclusion}
\label{con thm}Theorem \ref{T1 theorem} above holds with $\mathbf{R}_{\Psi
}^{\alpha ,n}$ in place of $\mathbf{R}^{\alpha ,n}$ provided $\Psi $ is a $%
C^{1,\delta }$ diffeomorphism.
\end{conclusion}

If we combine Theorem \ref{main} with this extension of Theorem \ref{T1
theorem}, we immediately obtain the following $T1$ theorem as a corollary.

\begin{remark}
The following theorem generalizes the $T1$ theorem for the Hilbert transform
(\cite{Lac}, \cite{LaSaShUr3} and \cite{Hyt2}) both in that the supports of
measures are more general, and in that the kernels treated are more general.
See also related work in the references given at the end of the paper.
\end{remark}

\begin{theorem}
\label{final}Fix a line $L$ and a collection of $\Omega $-quasicubes and
suppose that $\Omega $ is a $C^{1}$ diffeomorphism and $L$-transverse. Let $%
\sigma $ and $\omega $ be locally finite positive Borel measures on $\mathbb{%
R}^{n}$ (possibly having common point masses). Suppose that $\mathbf{R}%
_{\Psi }^{\alpha ,n}$ is a conformal fractional Riesz transform with $0\leq
\alpha <n$, where $\Psi $ is a $C^{1,\delta }$ diffeomorphism given by $\Psi
\left( x\right) =x-\left( 0,\psi \left( x_{1}\right) \right) $ where $\psi $
satisfies (\ref{small Lip}), i.e.%
\begin{equation*}
\left\Vert D\psi \right\Vert _{\infty }<\frac{1}{8n}\left( 1-\frac{\alpha }{n%
}\right) .
\end{equation*}%
Set $\left( \mathbf{R}_{\Psi }^{\alpha ,n}\right) _{\sigma }f=\mathbf{R}%
_{\Psi }^{\alpha ,n}\left( f\sigma \right) $ for any smooth truncation of $%
\mathbf{R}_{\Psi }^{\alpha ,n}$. If at least one of the measures $\sigma $
and $\omega $ is supported on the line $L$, then the operator norm $%
\mathfrak{N}_{\mathbf{R}_{\Psi }^{\alpha ,n}}$ of $\left( \mathbf{R}_{\Psi
}^{\alpha ,n}\right) _{\sigma }$ as an operator from\thinspace $L^{2}\left(
\sigma \right) $ to $L^{2}\left( \omega \right) $, uniformly in smooth
truncations, satisfies%
\begin{equation*}
\mathfrak{N}_{\mathbf{R}_{\Psi }^{\alpha ,n}}\approx C_{\alpha }\left( \sqrt{%
\mathfrak{A}_{2}^{\alpha }}+\mathfrak{T}_{\mathbf{R}_{\Psi }^{\alpha
,n}}^{\Omega \mathcal{Q}^{n}}+\mathfrak{T}_{\mathbf{R}_{\Psi }^{\alpha
,n}}^{\Omega \mathcal{Q}^{n},\func{dual}}\right) .
\end{equation*}
\end{theorem}

Our extension of Theorem \ref{final} to the case when one measure is
compactly supported on a $C^{1,\delta }$ curve $\mathcal{L}$ presented as a
graph requires additional work. More precisely, we suppose that $\mathcal{L}$
is presented as the graph of a $C^{1,\delta }$ function $\psi :\mathbb{R}%
\rightarrow \mathbb{R}^{n-1}$ given by%
\begin{equation*}
\psi \left( t\right) =\left( \psi ^{2}\left( t\right) ,\psi ^{3}\left(
t\right) ,...,\psi ^{n}\left( t\right) \right) \in \mathbb{R}^{n-1},\ \ \ \
\ t\in \mathbb{R}.
\end{equation*}%
Define $\Psi :\mathbb{R}^{n}\rightarrow \mathbb{R}^{n}$ by%
\begin{equation*}
\Psi \left( x\right) =\left( x^{1},x^{2}-\psi ^{2}\left( x^{1}\right)
,x^{3}-\psi ^{3}\left( x^{1}\right) ,...,x^{n}-\psi ^{n}\left( x^{1}\right)
\right) =x-\left( 0,\psi \left( x^{1}\right) \right) ,
\end{equation*}%
where $x=\left( x^{1},x^{\prime }\right) $. Then $\Psi $ is globally
invertible with inverse map%
\begin{equation*}
\Psi ^{-1}\left( \xi \right) =\left( \xi ^{1},\xi ^{2}+\psi ^{2}\left( \xi
_{1}\right) ,\xi ^{3}+\psi ^{3}\left( \xi _{1}\right) ,...,\xi ^{n}+\psi
^{n}\left( \xi _{1}\right) \right) =\xi +\left( 0,\psi \left( \xi
_{1}\right) \right) .
\end{equation*}%
Both $\Psi $ and its inverse $\Psi ^{-1}$ are $C^{1,\delta }$ maps, and $%
\Psi \mid _{\mathcal{L}}$ is a $C^{1,\delta }$ diffeomorphism from the curve 
$\mathcal{L}$ to the $x_{1}$-axis. Set $\Psi _{\ast }\mathcal{Q}^{n}=\left(
\Psi ^{-1}\right) ^{\ast }\mathcal{Q}^{n}=\left\{ \Psi Q:Q\in \mathcal{Q}%
^{n}\right\} $. The images $\Psi Q$ of cubes $Q$ under the map $\Psi $ are $%
\Psi $\emph{-quasicubes}.

The next theorem is a preliminary version of the main Theorem \ref{general
T1} that requires only the change of variable estimates in Propositions \ref%
{Muck equiv} and \ref{change of variable} below. The `defects' in this
preliminary version are that the quasitesting conditions are related to the
map $\Psi $ defining the curve $\mathcal{L}$, and that the smallness
condition (\ref{small Lip}) is imposed on the derivative of $\psi $.

\begin{theorem}
\label{curve}Let $n\geq 2$ and $0\leq \alpha <n$. Suppose that $\mathcal{L}$
is a $C^{1,\delta }$ curve in $\mathbb{R}^{n}$ presented as the graph of a $%
C^{1,\delta }$ function $\psi :\mathbb{R}\rightarrow \mathbb{R}^{n-1}$ as
above, and assume that (\ref{small Lip}) holds, i.e.%
\begin{equation*}
\left\Vert D\psi \right\Vert _{\infty }<\frac{1}{8n}\left( 1-\frac{\alpha }{n%
}\right) .
\end{equation*}%
Let $\omega $ and $\sigma $ be positive Borel measures (possibly having
common point masses), and assume that $\omega $ is compactly supported in $%
\mathcal{L}$. Let $\Psi $ be associated to $\psi $ as above. Finally, set $%
\mathcal{R}^{n}=R\mathcal{P}^{n}$ where $R$ is a rotation that is $L$%
-transverse when $L$ is the $x_{1}$-axis. Then the $\alpha $-fractional
Riesz transform $\mathbf{R}^{\alpha ,n}$ is bounded from $L^{2}\left( \sigma
\right) $ to $L^{2}\left( \omega \right) $ if and only if the Muckenhoupt
conditions hold, $\mathcal{A}_{2}^{\alpha }+\mathcal{A}_{2}^{\alpha ,\func{%
dual}}+A_{2}^{\alpha ,\limfunc{punct}}+A_{2}^{\alpha ,\limfunc{punct},\func{%
dual}}<\infty $, and the quasitesting conditions hold, $\mathfrak{T}_{%
\mathbf{R}^{\alpha ,n}}^{\Psi \mathcal{R}^{n}}+\mathfrak{T}_{\mathbf{R}%
^{\alpha ,n}}^{\Psi \mathcal{R}^{n},\func{dual}}<\infty $, where $\mathfrak{T%
}_{\mathbf{R}^{\alpha ,n}}^{\Psi \mathcal{R}^{n}}$ and $\mathfrak{T}_{%
\mathbf{R}^{\alpha ,n}}^{\Psi \mathcal{R}^{n},\func{dual}}$ are the best
constants in%
\begin{eqnarray*}
&&\int_{\Psi Q}\left\vert \mathbf{R}^{\alpha ,n}\left( \mathbf{1}_{\Psi
Q}\sigma \right) \right\vert ^{2}d\omega \leq \left( \mathfrak{T}_{\mathbf{R}%
^{\alpha ,n}}^{\Psi \mathcal{R}^{n}}\right) ^{2}\left\vert \Psi Q\right\vert
_{\sigma }, \\
&&\int_{\Psi Q}\left\vert \mathbf{R}^{\alpha ,n,\func{dual}}\left( \mathbf{1}%
_{\Psi Q}\omega \right) \right\vert ^{2}d\sigma \leq \left( \mathfrak{T}_{%
\mathbf{R}^{\alpha ,n}}^{\Psi \mathcal{R}^{n},\func{dual}}\right)
^{2}\left\vert \Psi Q\right\vert _{\omega }, \\
&&\ \ \ \ \ \text{for all cubes }Q\in \mathcal{R}^{n}=R\mathcal{P}^{n}.
\end{eqnarray*}%
Moreover we have the equivalence%
\begin{equation*}
\mathfrak{N}_{\mathbf{R}_{\sigma }^{\alpha ,n}}\left( \sigma ,\omega \right)
\approx \sqrt{\mathfrak{A}_{2}^{\alpha }\left( \sigma ,\omega \right) }+%
\mathfrak{T}_{\mathbf{R}^{\alpha ,n}}^{\Psi \mathcal{R}^{n}}\left( \sigma
,\omega \right) +\mathfrak{T}_{\mathbf{R}^{\alpha ,n}}^{\Psi \mathcal{R}^{n},%
\func{dual}}\left( \sigma ,\omega \right) .
\end{equation*}
\end{theorem}

The bound on $\left\Vert D\psi \right\Vert _{\infty }$ can be relaxed, but
we will not pursue this here. To obtain Theorem \ref{general T1}, we instead
remove the Lipschitz assumption by cutting the support $\mathcal{L}$ of $%
\omega $ into sufficiently small pieces $\mathcal{L}_{i}$ where the
oscillation of the tangents to $\mathcal{L}_{i}$ is small. Here the
necessity of the tripled testing condition $\mathfrak{T}_{\mathbf{R}_{\Psi
}^{\alpha ,n}}^{\Omega \mathcal{Q}^{n},\limfunc{triple},\func{dual}}$ in
Theorem \ref{main}\ plays a key role in permitting our testing conditions to
be taken with respect to the entire measure $\omega $, rather than with
respect to the corresponding pieces $\mathbf{1}_{\mathcal{L}_{i}}\omega $.
Then the restriction to \emph{quasi}testing conditions is removed using the
fact that Theorem \ref{T1 theorem} holds for conformal Riesz transforms with
general quasicubes, see Conclusion \ref{con thm}.

Finally, we mention another direction in which Theorem \ref{final} can be
generalized, namely to the setting where $\sigma $ and $\omega $ are locally
finite positive Borel measures supported on a $\left( k_{1}+1\right) $%
-dimensional subspace $S$ and a $\left( k_{2}+1\right) $-dimensional
subspace $W$ respectively of $\mathbb{R}^{n}$, $n=k_{1}+k_{2}+1$, with $W$
and $S$ intersecting at right angles in a line $L$. The precise result and
its proof are given in the appendix at the end of the paper.

\section{Definitions}

The $\alpha $-fractional \emph{Riesz} vector $\mathbf{R}^{\alpha ,n}=\left\{
R_{\ell }^{\alpha ,n}:1\leq \ell \leq n\right\} $ has as components the
Riesz transforms $R_{\ell }^{n,\alpha }$ with odd kernel $K_{\ell }^{\alpha
,n}\left( w\right) =\frac{\Omega _{\ell }\left( w\right) }{\left\vert
w\right\vert ^{n-\alpha }}$. The \emph{tangent line truncation} of\emph{\ }%
the Riesz transform $R_{\ell }^{\alpha ,n}$ has kernel $\Omega _{\ell
}\left( w\right) \rho _{\eta ,R}^{\alpha }\left( \left\vert w\right\vert
\right) $ where $\rho _{\eta ,R}^{\alpha }$ is continuously differentiable
on an interval $\left( 0,S\right) $ with $0<\eta <R<S$, and where $\rho
_{\eta ,R}^{\alpha }\left( r\right) =r^{\alpha -n}$ if $\eta \leq r\leq R$,
and has constant derivative on both $\left( 0,\eta \right) $ and $\left(
R,S\right) $ where $\rho _{\eta ,R}^{\alpha }\left( S\right) =0$. As shown
in \cite{SaShUr5} (see \cite{LaSaShUr3} for the one dimensional case without
holes), boundedness of $R_{\ell }^{n,\alpha }$ with one set of appropriate
truncations together with the offset $A_{2}^{\alpha }$ condition (see
below), is equivalent to boundedness of $R_{\ell }^{n,\alpha }$ with all
truncations. In particular this includes the smooth truncations with kernels 
$\varphi _{\eta ,R}\left( \left\vert w\right\vert \right) K_{\ell }^{\alpha
,n}\left( w\right) =\varphi _{\eta ,R}\left( \left\vert w\right\vert \right) 
\frac{\Omega _{\ell }\left( w\right) }{\left\vert w\right\vert ^{n-\alpha }}$
where $\varphi _{\eta ,R}$ is infinitely differentiable and compactly
supported on the interval $\left( 0,\infty \right) $ with $0<\eta <R<\infty $%
, and where $\varphi _{\eta ,R}\left( r\right) =1$ if $\eta \leq r\leq R$.

\subsection{Quasicubes}

Our general notion of quasicube will be derived from the following
definition.

\begin{definition}
\label{maps}We say that a map $\Omega :\mathbb{R}^{n}\rightarrow \mathbb{R}%
^{n}$ is a \emph{globally biLipschitz map} if%
\begin{equation*}
\left\Vert \Omega \right\Vert _{Lip}\equiv \sup_{\substack{ x,y\in \mathbb{R}%
^{n}  \\ x\neq y}}\frac{\left\Vert \Omega \left( x\right) -\Omega \left(
y\right) \right\Vert }{\left\Vert x-y\right\Vert }<\infty ,
\end{equation*}%
and $\left\Vert \Omega ^{-1}\right\Vert _{Lip}<\infty $. We say that a map $%
\Psi :\mathbb{R}^{n}\rightarrow \mathbb{R}^{n}$ is a $C^{1,\delta }$\emph{\
diffeomorphism} if%
\begin{equation*}
\left\Vert \Psi \right\Vert _{C^{1,\delta }}\equiv \sup_{x\in \mathbb{R}%
^{n}}\left\Vert \nabla \Psi \left( x\right) \right\Vert +\sup_{\substack{ %
x,y\in \mathbb{R}^{n}  \\ x\neq y}}\frac{\left\Vert \nabla \Psi \left(
x\right) -\nabla \Psi \left( y\right) \right\Vert }{\left\Vert
x-y\right\Vert ^{\delta }}<\infty ,
\end{equation*}%
and $\left\Vert \Psi ^{-1}\right\Vert _{C^{1,\delta }}<\infty $. When $%
\delta =0$, we write $C^{1}=C^{1,0}$ and $\left\Vert \Psi \right\Vert
_{C^{1}}\equiv \sup_{x\in \mathbb{R}^{n}}\left\Vert \nabla \Psi \left(
x\right) \right\Vert $.
\end{definition}

Note that if $\Omega $ is a globally biLipschitz map, then there are
constants $c,C>0$ such that%
\begin{equation*}
c\leq J_{\Omega }\left( x\right) \equiv \left\vert \det D\Omega \left(
x\right) \right\vert \leq C,\ \ \ \ \ x\in \mathbb{R}^{n}.
\end{equation*}%
Special cases of globally biLipschitz maps are given by $C^{1,\delta }$
diffeomorphisms $\Psi :\mathbb{R}^{n}\rightarrow \mathbb{R}^{n}$, and these
include those used in the definition of conformal Riesz transforms above,
and defined by%
\begin{equation*}
\Psi \left( x\right) =x-\left( 0,\psi \left( x^{1}\right) \right) ,
\end{equation*}%
where $x=\left( x^{1},x^{\prime }\right) \in \mathbb{R}^{n}$ and $\psi :%
\mathbb{R}\rightarrow \mathbb{R}^{n-1}$ is a $C^{1,\delta }$ function. We
denote by $\mathcal{Q}^{n}$ the collection of \emph{all} cubes in $\mathbb{R}%
^{n}$, and by $\mathcal{P}^{n}$ the subcollection of cubes in $\mathbb{R}%
^{n} $ with sides parallel to the coordinate axes, and by $\mathcal{D}^{n}$
(contained in $\mathcal{P}^{n}$) a dyadic grid in $\mathbb{R}^{n}$.

\begin{definition}
Suppose that $\Omega :\mathbb{R}^{n}\rightarrow \mathbb{R}^{n}$ is a
globally biLipschitz map.

\begin{enumerate}
\item If $E$ is a measurable subset of $\mathbb{R}^{n}$, we define $\Omega
E\equiv \left\{ \Omega \left( x\right) :x\in E\right\} $ to be the image of $%
E$ under the homeomorphism $\Omega $.

\begin{enumerate}
\item In the special case that $E=Q\in \mathcal{Q}^{n}$ is a cube in $%
\mathbb{R}^{n}$, we will refer to $\Omega Q$ as a quasicube (or $\Omega $%
-quasicube if $\Omega $ is not clear from the context).

\item We define the center of the quasicube $\Omega Q$ to be $\Omega x_{Q}$
where $x_{Q}$ is the center of $Q$.

\item We define the side length $\ell \left( \Omega Q\right) $ of the
quasicube $\Omega Q$ to be the sidelength $\ell \left( Q\right) $ of the
cube $Q$.

\item For $r>0$ we define the `dilation' $r\Omega Q$ of a quasicube $\Omega
Q $ to be $\Omega rQ$ where $rQ$ is the usual `dilation' of a cube in $%
\mathbb{R}^{n}$ that is concentric with $Q$ and having side length $r\ell
\left( \Omega Q\right) $.
\end{enumerate}

\item If $\mathcal{K}$ is a collection of cubes in $\mathbb{R}^{n}$, we
define $\Omega \mathcal{K}\equiv \left\{ \Omega Q:Q\in \mathcal{K}\right\} $
to be the collection of quasicubes $\Omega Q$ as $Q$ ranges over $\mathcal{K}
$.

\item If $\mathcal{F}$ is a grid of cubes in $\mathbb{R}^{n}$, we define the
inherited grid structure on $\Omega \mathcal{F}$ by declaring that $\Omega Q$
is a child of $\Omega Q^{\prime }$ in $\Omega \mathcal{F}$ if $Q$ is a child
of $Q^{\prime }$ in the grid $\mathcal{F}$.
\end{enumerate}
\end{definition}

Note that if $\Omega Q$ is a quasicube, then $\left\vert \Omega Q\right\vert
^{\frac{1}{n}}\approx \left\vert Q\right\vert ^{\frac{1}{n}}=\ell \left(
Q\right) =\ell \left( \Omega Q\right) $ shows that the measure of $\Omega Q$
is approximately its sidelength to the power $n$. Moreover, there is a
positive constant $R_{\limfunc{big}}$ such that we have the comparability
containments%
\begin{equation}
Q+\Omega x_{Q}\subset R_{\limfunc{big}}\Omega Q\text{ and }\Omega Q\subset
R_{\limfunc{big}}\left( Q+\Omega x_{Q}\right) \ .  \label{comp contain}
\end{equation}

\subsection{The $\mathcal{A}_{2}^{\protect\alpha }$ conditions}

Recall that $\Omega :\mathbb{R}^{n}\rightarrow \mathbb{R}^{n}$ is a globally
biLipschitz map. Now let $\mu $ be a locally finite positive Borel measure
on $\mathbb{R}^{n}$, and suppose $Q$ is a $\Omega $-quasicube in $\mathbb{R}%
^{n}$. Recall that $\left\vert Q\right\vert ^{\frac{1}{n}}\approx \ell
\left( Q\right) $ for a quasicube $Q$. The two $\alpha $-fractional Poisson
integrals of $\mu $ on a quasicube $Q$ are given by:%
\begin{eqnarray*}
\mathrm{P}^{\alpha }\left( Q,\mu \right) &\equiv &\int_{\mathbb{R}^{n}}\frac{%
\left\vert Q\right\vert ^{\frac{1}{n}}}{\left( \left\vert Q\right\vert ^{%
\frac{1}{n}}+\left\vert x-x_{Q}\right\vert \right) ^{n+1-\alpha }}d\mu
\left( x\right) , \\
\mathcal{P}^{\alpha }\left( Q,\mu \right) &\equiv &\int_{\mathbb{R}%
^{n}}\left( \frac{\left\vert Q\right\vert ^{\frac{1}{n}}}{\left( \left\vert
Q\right\vert ^{\frac{1}{n}}+\left\vert x-x_{Q}\right\vert \right) ^{2}}%
\right) ^{n-\alpha }d\mu \left( x\right) ,
\end{eqnarray*}%
where we emphasize that $\left\vert x-x_{Q}\right\vert $ denotes Euclidean
distance between $x$ and $x_{Q}$ and $\left\vert Q\right\vert $ denotes the
Lebesgue measure of the quasicube $Q$. We refer to $\mathrm{P}^{\alpha }$ as
the \emph{standard} Poisson integral and to $\mathcal{P}^{\alpha }$ as the 
\emph{reproducing} Poisson integral. Let $\sigma $ and $\omega $ be locally
finite positive Borel measures on $\mathbb{R}^{n}$, possibly having common
point masses, and suppose $0\leq \alpha <n$.

We say that the pair $\left( K,K^{\prime }\right) $ in $\mathcal{Q}%
^{n}\times \mathcal{Q}^{n}$ are \emph{neighbours} if $K$ and $K^{\prime }$
live in a common dyadic grid and both $K\subset 3K^{\prime }\setminus
K^{\prime }$ and $K^{\prime }\subset 3K\setminus K$, and we denote by $%
\mathcal{N}^{n}$ the set of pairs $\left( K,K^{\prime }\right) $ in $%
\mathcal{Q}^{n}\times \mathcal{Q}^{n}$ that are neighbours. Let $\Omega 
\mathcal{N}^{n}=\Omega \mathcal{Q}^{n}\times \Omega \mathcal{Q}^{n}$ be the
corresponding collection of neighbour pairs of quasicubes. Then we define
the classical \emph{offset }$A_{2}^{\alpha }$\emph{\ constants} by 
\begin{equation*}
A_{2}^{\alpha }\left( \sigma ,\omega \right) \equiv \sup_{\left( Q,Q^{\prime
}\right) \in \Omega \mathcal{N}^{n}}\frac{\left\vert Q\right\vert _{\sigma }%
}{\left\vert Q\right\vert ^{1-\frac{\alpha }{n}}}\frac{\left\vert Q^{\prime
}\right\vert _{\omega }}{\left\vert Q\right\vert ^{1-\frac{\alpha }{n}}}.
\end{equation*}%
Since the cubes in $\mathcal{P}^{n}$ are products of half open, half closed
intervals $\left[ a,b\right) $, the neighbouring quasicubes $\left(
Q,Q^{\prime }\right) \in \Omega \mathcal{N}^{n}$ are disjoint, and the
common point masses of $\sigma $ nor $\omega $ do simultaneously appear in
each factor.

We now define the \emph{one-tailed} $\mathcal{A}_{2}^{\alpha }$ constant
using $\mathcal{P}^{\alpha }$. The energy constants $\mathcal{E}_{\alpha }$
introduced in the next subsection will use the standard Poisson integral $%
\mathrm{P}^{\alpha }$.

\begin{definition}
The one-sided constants $\mathcal{A}_{2}^{\alpha }$ and $\mathcal{A}%
_{2}^{\alpha ,\func{dual}}$ for the weight pair $\left( \sigma ,\omega
\right) $ are given by%
\begin{eqnarray*}
\mathcal{A}_{2}^{\alpha }\left( \sigma ,\omega \right) &\equiv &\sup_{Q\in
\Omega \mathcal{Q}^{n}}\mathcal{P}^{\alpha }\left( Q,\mathbf{1}%
_{Q^{c}}\sigma \right) \frac{\left\vert Q\right\vert _{\omega }}{\left\vert
Q\right\vert ^{1-\frac{\alpha }{n}}}<\infty , \\
\mathcal{A}_{2}^{\alpha ,\func{dual}}\left( \sigma ,\omega \right) &\equiv
&\sup_{Q\in \Omega \mathcal{Q}^{n}}\frac{\left\vert Q\right\vert _{\sigma }}{%
\left\vert Q\right\vert ^{1-\frac{\alpha }{n}}}\mathcal{P}^{\alpha }\left( Q,%
\mathbf{1}_{Q^{c}}\omega \right) <\infty .
\end{eqnarray*}
\end{definition}

Note that these definitions are the analogues of the corresponding
conditions with `holes' introduced by Hyt\"{o}nen \cite{Hyt} in dimension $%
n=1$ - the supports of the measures $\mathbf{1}_{Q^{c}}\sigma $ and $\mathbf{%
1}_{Q}\omega $ in the definition of $\mathcal{A}_{2}^{\alpha }$ are
disjoint, and so the common point masses of $\sigma $ and $\omega $ do not
appear simultaneously in each factor.

\subsubsection{Punctured Muckenhoupt conditions}

Given an at most countable set $\mathfrak{P}=\left\{ p_{k}\right\}
_{k=1}^{\infty }$ in $\mathbb{R}^{n}$, a quasicube $Q\in \Omega \mathcal{Q}%
^{n}$, and a positive locally finite Borel measure $\mu $, define 
\begin{equation*}
\mu \left( Q,\mathfrak{P}\right) \equiv \left\vert Q\right\vert _{\mu }-\sup
\left\{ \mu \left( p_{k}\right) :p_{k}\in Q\cap \mathfrak{P}\right\} ,
\end{equation*}%
where we note that the $\sup $ above is achieved at some point $p_{k}$ since 
$\mu $ is locally finite. The quantity $\mu \left( Q,\mathfrak{P}\right) $
is simply the $\widetilde{\mu }$ measure of $Q$ where $\widetilde{\mu }$ is
the measure $\mu $ with its largest point mass in $Q$ removed. Given a
locally finite measure pair $\left( \sigma ,\omega \right) $, let $\mathfrak{%
P}_{\left( \sigma ,\omega \right) }=\left\{ p_{k}\right\} _{k=1}^{\infty }$
be the at most countable set of common point masses of $\sigma $ and $\omega 
$. Then as shown in \cite{SaShUr5} (as pointed out by Hyt\"{o}nen \cite{Hyt2}%
, the one-dimensional case follows from the proof of Proposition 2.1 in \cite%
{LaSaUr2}), the weighted norm inequality (\ref{2 weight}) implies finiteness
of the following \emph{punctured} Muckenhoupt conditions:%
\begin{eqnarray*}
A_{2}^{\alpha ,\limfunc{punct}}\left( \sigma ,\omega \right) &\equiv
&\sup_{Q\in \Omega \mathcal{Q}^{n}}\frac{\sigma \left( Q,\mathfrak{P}%
_{\left( \sigma ,\omega \right) }\right) }{\left\vert Q\right\vert ^{1-\frac{%
\alpha }{n}}}\frac{\left\vert Q\right\vert _{\omega }}{\left\vert
Q\right\vert ^{1-\frac{\alpha }{n}}}, \\
A_{2}^{\alpha ,\limfunc{punct},\func{dual}}\left( \sigma ,\omega \right)
&\equiv &\sup_{Q\in \Omega \mathcal{Q}^{n}}\frac{\left\vert Q\right\vert
_{\sigma }}{\left\vert Q\right\vert ^{1-\frac{\alpha }{n}}}\frac{\omega
\left( Q,\mathfrak{P}_{\left( \sigma ,\omega \right) }\right) }{\left\vert
Q\right\vert ^{1-\frac{\alpha }{n}}}.
\end{eqnarray*}

Finally, we point out that the intersection of these conditions, namely $%
\mathcal{A}_{2}^{\alpha }+\mathcal{A}_{2}^{\alpha ,\func{dual}}A_{2}^{\alpha
,\limfunc{punct}}+A_{2}^{\alpha ,\limfunc{punct},\func{dual}}<\infty $, is
independent of the biLipschitz map $\Omega $ as follows from taking $\Psi
=\Omega ^{-1}$ in Proposition \ref{Muck equiv} below.

\subsection{Quasicube testing and quasiweak boundedness property}

The following `dual' quasicube testing conditions are necessary for the
boundedness of $\mathbf{R}^{\alpha ,n}$ from $L^{2}\left( \sigma \right) $
to $L^{2}\left( \omega \right) $:%
\begin{eqnarray*}
\mathfrak{T}_{\mathbf{R}^{\alpha ,n}}^{2} &\equiv &\sup_{Q\in \Omega 
\mathcal{Q}^{n}}\frac{1}{\left\vert Q\right\vert _{\sigma }}%
\int_{Q}\left\vert \mathbf{R}^{\alpha ,n}\left( \mathbf{1}_{Q}\sigma \right)
\right\vert ^{2}\omega <\infty , \\
\left( \mathfrak{T}_{\mathbf{R}^{\alpha ,n}}^{\func{dual}}\right) ^{2}
&\equiv &\sup_{Q\in \Omega \mathcal{Q}^{n}}\frac{1}{\left\vert Q\right\vert
_{\omega }}\int_{Q}\left\vert \left( \mathbf{R}^{\alpha ,n}\right) ^{\func{%
dual}}\left( \mathbf{1}_{Q}\omega \right) \right\vert ^{2}\sigma <\infty .
\end{eqnarray*}%
Note that these conditions are required to hold uniformly over tangent line
truncations of $\mathbf{R}^{\alpha ,n}$, and where again we point out that
in the presence of the $\mathcal{A}_{2}^{\alpha }$ conditions, we can
equivalently replace the tangent line truncations with any other admissible
truncations.

The quasiweak boundedness property for $\mathbf{R}^{\alpha ,n}$ is another
necessary condition for (\ref{2 weight}) given by%
\begin{eqnarray*}
&&\left\vert \int_{Q}\mathbf{R}^{\alpha ,n}\left( 1_{Q^{\prime }}\sigma
\right) d\omega \right\vert \leq \mathcal{WBP}_{\mathbf{R}^{\alpha ,n}}\sqrt{%
\left\vert Q\right\vert _{\omega }\left\vert Q^{\prime }\right\vert _{\sigma
}}, \\
&&\ \ \ \ \ \text{for all dyadic quasicubes }Q,Q^{\prime }\in \Omega 
\mathcal{D}\text{ with }\frac{1}{C}\leq \frac{\left\vert Q\right\vert ^{%
\frac{1}{n}}}{\left\vert Q^{\prime }\right\vert ^{\frac{1}{n}}}\leq C, \\
&&\ \ \ \ \ \text{and either }Q\subset 3Q^{\prime }\setminus Q^{\prime }%
\text{ or }Q^{\prime }\subset 3Q\setminus Q, \\
&&\ \ \ \ \ \text{and all dyadic quasigrids }\Omega \mathcal{D}\text{.}
\end{eqnarray*}

\subsection{Quasienergy conditions}

Suppose $\Omega :\mathbb{R}^{n}\rightarrow \mathbb{R}^{n}$ is a $C^{1,\delta
}$ diffeomorphism. We begin by briefly recalling some of the notation used
in \cite{SaShUr5}. Given a dyadic $\Omega $-quasicube $K\in \mathcal{D}$ and
a positive measure $\mu $ we define the $\Omega $-quasiHaar projection $%
\mathsf{P}_{K}^{\mu }\equiv \sum_{_{J\in \mathcal{D}:\ J\subset
K}}\bigtriangleup _{J}^{\mu }$ where the projections $\bigtriangleup
_{J}^{\mu }$ are the usual orthogonal projections onto the space of mean
value zero functions that are constant on the children of $J$ - see e.g. 
\cite{SaShUr5}. Now we recall the definition of a \emph{good} dyadic
quasicube - see \cite{NTV3} and \cite{LaSaUr2} and \cite{SaShUr} for more
detail - and the definition of a quasicube that is \emph{deeply embedded} in
another quasicube. We say that a dyadic quasicube $J$ is $\left( \mathbf{r}%
,\varepsilon \right) $-\emph{deeply embedded} in a dyadic quasicube $K$, or
simply $\mathbf{r}$\emph{-deeply embedded} in $K$, which we write as $%
J\Subset _{\mathbf{r}}K$, when $J\subset K$ and both 
\begin{eqnarray}
\ell \left( J\right) &\leq &2^{-\mathbf{r}}\ell \left( K\right) ,
\label{def deep embed} \\
\limfunc{quasidist}\left( J,\partial K\right) &\geq &\frac{1}{2}\ell \left(
J\right) ^{\varepsilon }\ell \left( K\right) ^{1-\varepsilon },  \notag
\end{eqnarray}%
where we define the quasidistance $\limfunc{quasidist}\left( E,F\right) $
between two sets $E$ and $F$ to be the Euclidean distance $\limfunc{dist}%
\left( \Omega ^{-1}E,\Omega ^{-1}F\right) $ between the preimages $\Omega
^{-1}E$ and $\Omega ^{-1}F$ of $E$ and $F$ under the map $\Omega $, and
where we recall that $\ell \left( J\right) \approx \left\vert J\right\vert ^{%
\frac{1}{n}}$.

\begin{definition}
Let $\mathbf{r}\in \mathbb{N}$ and $0<\varepsilon <1$. A dyadic quasicube $J$
is $\left( \mathbf{r},\varepsilon \right) $\emph{-good}, or simply \emph{good%
}, if for \emph{every} dyadic superquasicube $I$, it is the case that 
\textbf{either} $J$ has side length at least $2^{-\mathbf{r}}$ times that of 
$I$, \textbf{or} $J\Subset _{\mathbf{r}}I$ is $\left( \mathbf{r},\varepsilon
\right) $-deeply embedded in $I$.
\end{definition}

The parameters $\mathbf{r},\varepsilon $ will be fixed sufficiently large
and small respectively later on, and we denote the set of such good dyadic
quasicubes by $\Omega \mathcal{D}_{\limfunc{good}}$. We thus have 
\begin{equation*}
\left\Vert \mathsf{P}_{I}^{\mu }\mathbf{x}\right\Vert _{L^{2}\left( \mu
\right) }^{2}=\int_{I}\left\vert \mathbf{x}-\mathbb{E}_{I}^{\mu }\mathbf{x}%
\right\vert ^{2}d\mu \left( x\right) =\int_{I}\left\vert \mathbf{x}-\left( 
\frac{1}{\left\vert I\right\vert _{\mu }}\int_{I}\mathbf{x}dx\right)
\right\vert ^{2}d\mu \left( x\right) ,\ \ \ \ \ \mathbf{x}=\left(
x_{1},...,x_{n}\right) ,
\end{equation*}%
where $\mathsf{P}_{I}^{\mu }\mathbf{x}$ is the orthogonal projection of the
identity function $\mathbf{x}:\mathbb{R}^{n}\rightarrow \mathbb{R}^{n}$ onto
the vector-valued subspace of $\oplus _{k=1}^{n}L^{2}\left( \mu \right) $
consisting of functions supported in $I$ with $\mu $-mean value zero, and
where $\mathbb{E}_{I}^{\mu }\mathbf{x}$ is the expectation ($\mu $-average)
of $\mathbf{x}$ on the cube $I$. At this point we emphasize that in the
setting of quasicubes we continue to use the linear function $\mathbf{x}$
and not the pushforward of $\mathbf{x}$ by $\Omega $. The reason of course
is that the quasienergy defined below is used to capture the first order
information in the Taylor expansion of a singular kernel.

We use the collection $\mathcal{M}_{\mathbf{r}-\limfunc{deep}}\left(
K\right) $ of \emph{maximal} $\mathbf{r}$-deeply embedded dyadic
subquasicubes of a dyadic quasicube $K$. We let $J^{\ast }=\gamma J$ where $%
\gamma \geq 2$. The goodness parameter $\mathbf{r}$ is chosen sufficiently
large, depending on $\varepsilon $ and $\gamma $, that the bounded overlap
property 
\begin{equation}
\sum_{J\in \mathcal{M}_{\mathbf{r}-\limfunc{deep}}\left( K\right) }\mathbf{1}%
_{J^{\ast }}\leq \beta \mathbf{1}_{K}\ ,  \label{bounded overlap}
\end{equation}%
holds for some positive constant $\beta $ depending only on $n,\gamma ,%
\mathbf{r}$ and $\varepsilon $ (see \cite{SaShUr4}). We will also need the
following refinements of $\mathcal{M}_{\mathbf{r}-\limfunc{deep}}\left(
K\right) $ for each $\ell \geq 0$:%
\begin{equation*}
\mathcal{M}_{\mathbf{r}-\limfunc{deep}}^{\ell }\left( K\right) \equiv
\left\{ J\in \mathcal{M}_{\mathbf{r}-\limfunc{deep}}\left( \pi ^{\ell
}K\right) :J\subset L\text{ for some }L\in \mathcal{M}_{\mathbf{r}-\limfunc{%
deep}}\left( K\right) \right\} ,
\end{equation*}%
where $\pi ^{\ell }K$ denotes the $\ell ^{th}$ parent above $K$ in the
dyadic grid. Since $J\in \mathcal{M}_{\mathbf{r}-\limfunc{deep}}^{\ell
}\left( K\right) $ implies $\gamma J\subset K$, we also have from (\ref%
{bounded overlap}) that%
\begin{equation}
\sum_{J\in \mathcal{M}_{\mathbf{r}-\limfunc{deep}}^{\left( \ell \right)
}\left( K\right) }\mathbf{1}_{J^{\ast }}\leq \beta \mathbf{1}_{K}\ ,\ \ \ \
\ \text{for each }\ell \geq 0.  \label{bounded overlap'}
\end{equation}%
Of course $\mathcal{M}_{\mathbf{r}-\limfunc{deep}}^{0}\left( K\right) =%
\mathcal{M}_{\mathbf{r}-\limfunc{deep}}\left( K\right) $, but $\mathcal{M}_{%
\mathbf{r}-\limfunc{deep}}^{\ell }\left( K\right) $ is in general a finer
subdecomposition of $K$ the larger $\ell $ is, and may in fact be empty.

There is one final generalization we need. We say that a quasicube $J\in
\Omega \mathcal{P}^{n}$ is $\left( \mathbf{r},\varepsilon \right) $-\emph{%
deeply embedded} in a quasicube $K\in \Omega \mathcal{P}^{n}$, or simply $%
\mathbf{r}$\emph{-deeply embedded} in $K$, which we write as $J\Subset _{%
\mathbf{r}}K$, when $J\subset K$ and both 
\begin{eqnarray*}
\ell \left( J\right) &\leq &2^{-\mathbf{r}}\ell \left( K\right) , \\
\limfunc{quasidist}\left( J,\partial K\right) &\geq &\frac{1}{2}\ell \left(
J\right) ^{\varepsilon }\ell \left( K\right) ^{1-\varepsilon }.
\end{eqnarray*}%
This is the same definition as we gave earlier for \emph{dyadic} quasicubes,
but is now extended to arbitrary quasicubes $J,K\in \Omega \mathcal{P}^{n}$.
Now given $K\in \Omega \mathcal{D}^{n}$ and $F\in \Omega \mathcal{P}^{n}$
with $\ell \left( F\right) \geq \ell \left( K\right) $, define 
\begin{equation*}
\mathcal{M}_{\mathbf{r}-\limfunc{deep}}^{F}\left( K\right) \equiv \left\{ 
\text{maximal }J\in \Omega \mathcal{D}^{n}:J\Subset _{\mathbf{r}}K\text{ and 
}J\Subset _{\mathbf{r}}F\right\} ,
\end{equation*}%
and%
\begin{equation*}
\left( \mathcal{E}_{\alpha }^{\limfunc{xrefined}}\right) ^{2}\equiv
\sup_{I}\sup_{F\in \Omega \mathcal{P}^{n}:\ \ell \left( F\right) \geq 2\ell
\left( I\right) }\frac{1}{\left\vert I\right\vert _{\sigma }}\sum_{J\in 
\mathcal{M}_{\mathbf{r}-\limfunc{deep}}^{F}\left( I\right) }\left( \frac{%
\mathrm{P}^{\alpha }\left( J,\mathbf{1}_{I\setminus \gamma J}\sigma \right) 
}{\left\vert J\right\vert ^{\frac{1}{n}}}\right) ^{2}\left\Vert \mathsf{P}%
_{J}^{\limfunc{subgood},\omega }\mathbf{x}\right\Vert _{L^{2}\left( \omega
\right) }^{2}\ .
\end{equation*}%
The important difference here is that the quasicube $F\in \Omega \mathcal{P}%
^{n}$ is permitted to lie outside the quasigrid $\Omega \mathcal{D}^{n}$
containing $K$. Similarly we have a dual version of $\mathcal{E}_{\alpha }^{%
\limfunc{xrefined}}$.

\begin{definition}
\label{energy condition}Suppose $\sigma $ and $\omega $ are positive Borel
measures on $\mathbb{R}^{n}$. Then the quasienergy condition constant $%
\mathcal{E}_{\alpha }^{\Omega \mathcal{Q}^{n}}$ is given by%
\begin{equation*}
\left( \mathcal{E}_{\alpha }^{\Omega \mathcal{Q}^{n}}\right) ^{2}\equiv \sup 
_{\substack{ F\in \Omega \mathcal{P}^{n}  \\ \ell \left( F\right) \geq \ell
\left( I\right) }}\sup_{I=\dot{\cup}I_{r}}\frac{1}{\left\vert I\right\vert
_{\sigma }}\sum_{r=1}^{\infty }\sum_{J\in \mathcal{M}_{\mathbf{r}-\limfunc{%
deep}}^{F}\left( I_{r}\right) }\left( \frac{\mathrm{P}^{\alpha }\left( J,%
\mathbf{1}_{I\setminus \gamma J}\sigma \right) }{\left\vert J\right\vert ^{%
\frac{1}{n}}}\right) ^{2}\left\Vert \mathsf{P}_{J}^{\omega }\mathbf{x}%
\right\Vert _{L^{2}\left( \omega \right) }^{2}\ ,
\end{equation*}%
where $\sup_{I=\dot{\cup}I_{r}}$ above is taken over

\begin{enumerate}
\item all dyadic quasigrids $\Omega \mathcal{D}$,

\item all $\Omega \mathcal{D}$-dyadic quasicubes $I$,

\item and all subpartitions $\left\{ I_{r}\right\} _{r=1}^{N\text{ or }%
\infty }$ of the quasicube $I$ into $\Omega \mathcal{D}$-dyadic
subquasicubes $I_{r}$.
\end{enumerate}
\end{definition}

This definition of the quasienergy constant $\mathcal{E}_{\alpha }^{\Omega 
\mathcal{Q}^{n}}$ is larger than that used in \cite{SaShUr5}. There is a
similar definition for the dual (backward) quasienergy condition that simply
interchanges $\sigma $ and $\omega $ everywhere. These definitions of\ the
quasienergy condition depend on the choice of goodness parameters $\mathbf{r}
$ and $\varepsilon $.

Finally, we record the following elementary special case of the Energy Lemma
(see e.g. \cite{SaShUr5} or \cite{SaShUr} or \cite{LaWi}) that we will need
here. Recall that our quasicubes come from a fixed globally biLipschitz map $%
\Omega $ in $\mathbb{R}^{n}$. Our singular integrals below will be conformal
fractional Riesz transforms associated with an unrelated $C^{1,\delta }$
diffeomorphism $\Psi $ of $\mathbb{R}^{n}$ that is presented as a graph.

\begin{lemma}[\textbf{Quasienergy Lemma}]
\label{ener}Suppose that $\Omega $ is a globally biLipschitz map, and that $%
\Psi $ is a $C^{1,\delta }$ diffeomorphism of $\mathbb{R}^{n}$. Let $J\ $be
a quasicube in $\Omega \mathcal{D}^{\omega }$. Let $\psi _{J}$ be an $%
L^{2}\left( \omega \right) $ function supported in $J$ and with $\omega $%
-integral zero. Let $\nu $ be a positive measure supported in $\mathbb{R}%
^{n}\setminus \gamma J$ with $\gamma \geq 2$. Then for $\mathbf{R}_{\Psi
}^{\alpha ,n}$ a conformal $\alpha $-fractional Riesz transform, we have%
\begin{equation*}
\left\vert \left\langle \mathbf{R}_{\Psi }^{\alpha ,n}\left( \nu \right)
,\psi _{J}\right\rangle _{\omega }\right\vert \lesssim \left\Vert \psi
_{J}\right\Vert _{L^{2}\left( \omega \right) }\left( \frac{\mathrm{P}%
^{\alpha }\left( J,\nu \right) }{\left\vert J\right\vert ^{\frac{1}{n}}}%
\right) \left\Vert \mathsf{P}_{J}^{\omega }\mathbf{x}\right\Vert
_{L^{2}\left( \omega \right) }.
\end{equation*}
\end{lemma}

\section{One measure supported in a line}

In this section we prove Theorem \ref{main}, i.e. we prove that the $\Omega $%
-quasienergy conditions, the backward tripled $\Omega $-quasitesting
conditions (for appropriately rotated quasicubes), and the $\Omega $%
-quasiweak boundedness property (for appropriately rotated quasicubes) are
implied by the Muckenhoupt $\mathcal{A}_{2}^{\alpha }$ conditions and the $%
\Omega $-quasitesting conditions $\mathfrak{T}_{\mathbf{R}_{\Psi }^{\alpha
,n}}^{\Omega \mathcal{Q}^{n}}$ and $\mathfrak{T}_{\mathbf{R}_{\Psi }^{\alpha
,n}}^{\Omega \mathcal{Q}^{n},\func{dual}}$ associated to the tangent line
truncations of a conformal $\alpha $-fractional Riesz transform $\mathbf{R}%
_{\Psi }^{\alpha ,n}$, when \emph{one} of the measures $\omega $ is
supported in a certain line, and the other measure $\sigma $ is arbitrary.
The one-dimensional character of just one of the measures is enough to
circumvent the failure of strong reversal of energy as described in \cite%
{SaShUr2} and \cite{SaShUr3}.

\begin{notation}
We emphasize again that the $C^{1,\delta }$ diffeomorphism $\Psi $ that
appears in the definition of the conformal $\alpha $-fractional Riesz
transform $\mathbf{R}_{\Psi }^{\alpha ,n}$ need not have any relation to the
globally biLipschitz map $\Omega $ that is used to define the quasicubes
under consideration.
\end{notation}

Recall that the conformal Riesz transforms $\mathbf{R}_{\Psi }^{\alpha ,n}$
considered here have vector kernel $\mathbf{K}_{\Psi }^{\alpha ,n}$ defined
by%
\begin{equation*}
\mathbf{K}_{\Psi }^{\alpha ,n}\left( y,x\right) \equiv \frac{y-x}{\left\vert
\Psi \left( y\right) -\Psi \left( x\right) \right\vert ^{n+1-\alpha }},
\end{equation*}%
where we suppose that $\Psi $ is given as the graph of $\psi :\mathbb{R}%
\rightarrow \mathbb{R}^{n-1}$: 
\begin{equation}
\Psi \left( x\right) =\left( x^{1},x^{\prime }+\psi \left( x^{1}\right)
\right) =\left( x^{1},x^{2}+\psi ^{2}\left( x^{1}\right) ,...,x^{n}+\psi
^{n}\left( x^{1}\right) \right) ,  \label{def Psi}
\end{equation}%
where $\psi \in C^{1,\delta }$.

Fix a collection of $\Omega $-quasicubes where $\Omega :\mathbb{R}%
^{n}\rightarrow \mathbb{R}^{n}$ is a globally biLipschitz map unrelated to $%
\Psi $. Fix a dyadic quasigrid $\Omega \mathcal{D}$, and suppose that $%
\omega $ is supported in the $x_{1}$-axis, which we denote by $L$. We will
show that both quasienergy conditions hold relative to $\Omega \mathcal{D}$.
Furthermore, when $\Omega $ is a $C^{1}$ diffeomorphism and $L$-transverse,
we will show that the backward tripled quasitesting condition, and hence
also the quasiweak boundedness property, is controlled by $\mathcal{A}%
_{2}^{\alpha ,\func{dual}}$ and dual quasitesting.

Let $3<\gamma =\gamma \left( n,\alpha \right) $ where $\gamma $ will be
taken sufficiently large depending on $n$ and $\alpha $ for the arguments
below to be valid - see in particular (\ref{smallness}), (\ref{provided''}),
(\ref{provided}) and (\ref{provided'}) below - and where we also need $%
\left\Vert D\psi \right\Vert _{\infty }$ sufficiently small depending on $n$
and $\alpha $ as in (\ref{small Lip}) above, i.e.%
\begin{equation}
\left\Vert D\psi \right\Vert _{\infty }<\frac{1}{8n^{2}}\left( n-\alpha
\right) .  \label{assumptions}
\end{equation}

\subsection{Backward quasienergy condition}

The dual (backward) quasienergy condition $\mathcal{E}_{\alpha }^{\Omega 
\mathcal{Q}^{n},\func{dual}}\lesssim \mathfrak{T}_{\mathbf{R}_{\Psi
}^{\alpha ,n}}^{\Omega \mathcal{Q}^{n},\func{dual}}+\sqrt{\mathcal{A}%
_{2}^{\alpha ,\func{dual}}}$ is the more straightforward of the two to
verify, and so we turn to it first. We will show%
\begin{equation*}
\sup_{\ell \geq 0}\sum_{r=1}^{\infty }\sum_{J\in \mathcal{M}_{\limfunc{deep}%
}^{\ell }\left( I_{r}\right) }\left( \frac{\mathrm{P}^{\alpha }\left( J,%
\mathbf{1}_{I\setminus J^{\ast }}\omega \right) }{\left\vert J\right\vert ^{%
\frac{1}{n}}}\right) ^{2}\left\Vert \mathsf{P}_{J}^{\sigma }\mathbf{x}%
\right\Vert _{L^{2}\left( \sigma \right) }^{2}\leq \left( \mathfrak{T}_{%
\mathbf{R}_{\Psi }^{\alpha ,n}}^{\Omega \mathcal{Q}^{n},\func{dual}}\right)
^{2}\left\vert I\right\vert _{\omega }\ ,
\end{equation*}%
for all partitions of a dyadic quasicube $I=\dbigcup\limits_{r=1}^{\infty
}I_{r}$ into dyadic subquasicubes $I_{r}$. We fix $\ell \geq 0$ and suppress
both $\ell $ and $\mathbf{r}$ in the notation $\mathcal{M}_{\limfunc{deep}%
}\left( I_{r}\right) =\mathcal{M}_{\mathbf{r}-\limfunc{deep}}^{\ell }\left(
I_{r}\right) $. Recall that $J^{\ast }=\gamma J$, and that the bounded
overlap property (\ref{bounded overlap'}) holds. We may of course assume
that $I$ intersects the $x_{1}$-axis $L$. Now we set $\mathcal{M}_{\limfunc{%
deep}}\equiv \dbigcup\limits_{r=1}^{\infty }\mathcal{M}_{\limfunc{deep}%
}\left( I_{r}\right) $ and write%
\begin{equation*}
\sum_{r=1}^{\infty }\sum_{J\in \mathcal{M}_{\limfunc{deep}}\left(
I_{r}\right) }\left( \frac{\mathrm{P}^{\alpha }\left( J,\mathbf{1}%
_{I\setminus \gamma J}\omega \right) }{\left\vert J\right\vert ^{\frac{1}{n}}%
}\right) ^{2}\left\Vert \mathsf{P}_{J}^{\sigma }\mathbf{x}\right\Vert
_{L^{2}\left( \sigma \right) }^{2}=\sum_{J\in \mathcal{M}_{\limfunc{deep}%
}}\left( \frac{\mathrm{P}^{\alpha }\left( J,\mathbf{1}_{I\setminus \gamma
J}\omega \right) }{\left\vert J\right\vert ^{\frac{1}{n}}}\right)
^{2}\left\Vert \mathsf{P}_{J}^{\sigma }\mathbf{x}\right\Vert _{L^{2}\left(
\sigma \right) }^{2}\ .
\end{equation*}%
We will consider the cases $3J\cap L=\emptyset $ and $3J\cap L\neq \emptyset 
$ separately.

Suppose $3J\cap L=\emptyset $. There is $c>0$ and a finite sequence $\left\{
\xi _{k}\right\} _{k=1}^{N}$ in $\mathbb{S}^{n-1}$ (actually of the form $%
\xi _{k}=\left( 0,\xi _{k}^{2},...,\xi _{k}^{n}\right) $) with the following
property. For each $J\in \mathcal{M}_{\limfunc{deep}}$ with $3J\cap
L=\emptyset $, there is $1\leq k=k\left( J\right) \leq N$ such that for $%
y\in J$ and $x\in I\cap L$, the linear combination $\xi _{k}\cdot \mathbf{K}%
_{\Psi }^{\alpha ,n}\left( y,x\right) $ is positive and satisfies 
\begin{equation*}
\xi _{k}\cdot \mathbf{K}_{\Psi }^{\alpha ,n}\left( y,x\right) =\frac{\xi
_{k}\cdot \left( y-x\right) }{\left\vert \Psi \left( y\right) -\Psi \left(
x\right) \right\vert ^{n+1-\alpha }}\gtrsim \frac{\ell \left( J\right) }{%
\left\vert y-x\right\vert ^{n+1-\alpha }}.
\end{equation*}%
For example, in the plane $n=2$, if $J$ lies above the $x_{1}$-axis $L$,
then for $y\in J$ and $x\in L$ we have $y_{2}\gtrsim \left( 3-1\right) \ell
\left( J\right) >\ell \left( J\right) $ and $x_{2}=0$, hence the estimate 
\begin{equation*}
\left( 0,1\right) \cdot \mathbf{K}_{\Psi }^{\alpha ,n}\left( y,x\right) =%
\frac{y_{2}-x_{2}}{\left\vert \Psi \left( y\right) -\Psi \left( x\right)
\right\vert ^{n+1-\alpha }}\gtrsim \frac{\ell \left( J\right) }{\left\vert
y-x\right\vert ^{n+1-\alpha }}.
\end{equation*}%
For $J$ below\thinspace $L$ we take the unit vector $\left( 0,-1\right) $ in
place of $\left( 0,1\right) $. Thus for $y\in J\in \mathcal{M}_{\limfunc{deep%
}}$ and $k=k\left( J\right) $ we have the following `weak reversal' of
quasienergy for the conformal Riesz transform $\mathbf{R}_{\Psi }^{\alpha
,n} $ with kernel $\mathbf{K}_{\Psi }^{\alpha ,n}\left( y,x\right) $, 
\begin{eqnarray}
\left\vert \mathbf{R}_{\Psi }^{\alpha ,n}\left( \mathbf{1}_{I\cap L}\omega
\right) \left( y\right) \right\vert &=&\left\vert \int_{I\cap L}\mathbf{K}%
_{\Psi }^{\alpha ,n}\left( y,x\right) d\omega \left( x\right) \right\vert
\label{weak control} \\
&\geq &\left\vert \int_{I\cap L}\xi _{k}\cdot \mathbf{K}_{\Psi }^{\alpha
,n}\left( y,x\right) d\omega \left( x\right) \right\vert  \notag \\
&\gtrsim &\int_{I\cap L}\frac{\ell \left( J\right) }{\left\vert
y-x\right\vert ^{n+1-\alpha }}d\omega \left( x\right) \approx \mathrm{P}%
^{\alpha }\left( J,\mathbf{1}_{I}\omega \right) .  \notag
\end{eqnarray}%
Thus from (\ref{weak control}) and the pairwise disjointedness of $J\in 
\mathcal{M}_{\limfunc{deep}}$, we have%
\begin{eqnarray*}
&&\sum_{\substack{ J\in \mathcal{M}_{\limfunc{deep}}  \\ 3J\cap L=\emptyset 
}}\left( \frac{\mathrm{P}^{\alpha }\left( J,\mathbf{1}_{I}\omega \right) }{%
\left\vert J\right\vert ^{\frac{1}{n}}}\right) ^{2}\left\Vert \mathsf{P}%
_{J}^{\sigma }\mathbf{x}\right\Vert _{L^{2}\left( \sigma \right)
}^{2}\lesssim \sum_{\substack{ J\in \mathcal{M}_{\limfunc{deep}}  \\ 3J\cap
L=\emptyset }}\mathrm{P}^{\alpha }\left( J,\mathbf{1}_{I}\omega \right)
^{2}\left\vert J\right\vert _{\sigma } \\
&\lesssim &\sum_{J\in \mathcal{M}_{\limfunc{deep}}}\int_{J}\left\vert 
\mathbf{R}_{\Psi }^{\alpha ,n}\left( \mathbf{1}_{I\cap L}\omega \right)
\left( y\right) \right\vert ^{2}d\sigma \left( y\right) \\
&\leq &\int_{I}\left\vert \mathbf{R}_{\Psi }^{\alpha ,n}\left( \mathbf{1}%
_{I}\omega \right) \left( y\right) \right\vert ^{2}d\sigma \left( y\right)
\leq \left( \mathfrak{T}_{\mathbf{R}_{\Psi }^{\alpha ,n}}^{\Omega \mathcal{Q}%
^{n},\func{dual}}\right) ^{2}\left\vert I\right\vert _{\omega }\ .
\end{eqnarray*}

Now we turn to estimating the sum over those quasicubes $J\in \mathcal{M}_{%
\limfunc{deep}}$ for which $3J\cap L\neq \emptyset $. In this case we use
the one-dimensional nature of the support of $\omega $ to obtain a strong
reversal of one of the partial quasienergies. Recall the Hilbert transform
inequality for intervals $J$ and $I$ with $2J\subset I$ and $\limfunc{supp}%
\mu \subset \mathbb{R}\setminus I$: 
\begin{eqnarray}
\sup_{y,z\in J}\frac{H\mu \left( y\right) -H\mu \left( z\right) }{y-z}
&=&\int_{\mathbb{R}\setminus I}\left\{ \frac{\frac{1}{x-y}-\frac{1}{x-z}}{y-z%
}\right\} d\mu \left( x\right)  \label{Hilbert} \\
&=&\int_{\mathbb{R}\setminus I}\frac{1}{\left( x-y\right) \left( x-z\right) }%
d\mu \left( x\right) \approx \frac{\mathrm{P}\left( J,\mu \right) }{%
\left\vert J\right\vert }.  \notag
\end{eqnarray}%
We wish to obtain a similar control in the situation at hand, but the matter
is now complicated by the extra dimensions. Fix $y=\left( y^{1},y^{\prime
}\right) ,z=\left( z^{1},z^{\prime }\right) \in J$ and $x=\left(
x^{1},0\right) \in L\setminus \gamma J$.

We consider first the case%
\begin{equation}
\left\vert y^{\prime }-z^{\prime }\right\vert \leq \left\vert
y^{1}-z^{1}\right\vert ,  \label{case 1}
\end{equation}%
We pause to recall the main assumption in (\ref{assumptions}) regarding the
size of the graphing function:%
\begin{equation}
\left\Vert D\psi \right\Vert _{\infty }<\frac{1}{8n^{2}}\left( n-\alpha
\right) .  \label{also note}
\end{equation}%
Now the first component $\left( \mathbf{R}_{\Psi }^{\alpha ,n}\right) _{1}$
is `positive' in the direction of the $x^{1}$-axis $L$, and so for $\left(
y^{1},y^{\prime }\right) ,\left( z^{1},z^{\prime }\right) \in J$, we write%
\begin{eqnarray*}
&&\frac{\left( \mathbf{R}_{\Psi }^{\alpha ,n}\right) _{1}\left( \mathbf{1}%
_{I\setminus \gamma J}\omega \right) \left( y^{1},y^{\prime }\right) -\left( 
\mathbf{R}_{\Psi }^{\alpha ,n}\right) _{1}\left( \mathbf{1}_{I\setminus
\gamma J}\omega \right) \left( z^{1},z^{\prime }\right) }{y^{1}-z^{1}} \\
&=&\int_{I\setminus \gamma J}\left\{ \frac{\left( \mathbf{K}_{\Psi }^{\alpha
,n}\right) _{1}\left( \left( y^{1},y^{\prime }\right) ,x\right) -\left( 
\mathbf{K}_{\Psi }^{\alpha ,n}\right) _{1}\left( \left( z^{1},z^{\prime
}\right) ,x\right) }{y^{1}-z^{1}}\right\} d\omega \left( x\right) \\
&=&\int_{I\setminus \gamma J}\left\{ \frac{\frac{y^{1}-x^{1}}{\left\vert
\Psi \left( y\right) -\Psi \left( x\right) \right\vert ^{n+1-\alpha }}-\frac{%
z^{1}-x^{1}}{\left\vert \Psi \left( z\right) -\Psi \left( x\right)
\right\vert ^{n+1-\alpha }}}{y^{1}-z^{1}}\right\} d\omega \left( x\right) \ .
\end{eqnarray*}

For $0\leq t\leq 1$ define 
\begin{eqnarray*}
w_{t} &\equiv &ty+\left( 1-t\right) z=z+t\left( y-z\right) , \\
\text{so that }w_{t}-x &=&t\left( y-x\right) +\left( 1-t\right) \left(
z-x\right) ,
\end{eqnarray*}%
and 
\begin{equation*}
\Phi \left( t\right) \equiv \frac{w_{t}^{1}-x^{1}}{\left\vert \Psi \left(
w_{t}\right) -\Psi \left( x\right) \right\vert ^{n+1-\alpha }},
\end{equation*}%
so that%
\begin{equation*}
\frac{y^{1}-x^{1}}{\left\vert \Psi \left( y\right) -\Psi \left( x\right)
\right\vert ^{n+1-\alpha }}-\frac{z^{1}-x^{1}}{\left\vert \Psi \left(
z\right) -\Psi \left( x\right) \right\vert ^{n+1-\alpha }}=\Phi \left(
1\right) -\Phi \left( 0\right) =\int_{0}^{1}\Phi ^{\prime }\left( t\right)
dt\ .
\end{equation*}%
Now we will use (\ref{def Psi}) and $\nabla \left\vert \xi \right\vert
^{\tau }=\tau \left\vert \xi \right\vert ^{\tau -2}\xi $ to compute that%
\begin{eqnarray*}
\frac{d}{dt}\Phi \left( t\right) &=&\frac{y^{1}-z^{1}}{\left\vert \Psi
\left( w_{t}\right) -\Psi \left( x\right) \right\vert ^{n+1-\alpha }} \\
&&+\left( w_{t}^{1}-x^{1}\right) \frac{-\left( n+1-\alpha \right) }{%
\left\vert \Psi \left( w_{t}\right) -\Psi \left( x\right) \right\vert
^{n+3-\alpha }} \\
&&\times \left( 
\begin{array}{cc}
w_{t}^{1}-x^{1}, & w_{t}^{\prime }-x^{\prime }+\psi \left( w_{t}^{1}\right)
-\psi \left( x^{1}\right)%
\end{array}%
\right) \left[ 
\begin{array}{cc}
1 & \mathbf{0} \\ 
D\psi \left( w_{t}^{1}\right) & \mathbf{I}_{n-1}%
\end{array}%
\right] \left( 
\begin{array}{c}
y^{1}-z^{1} \\ 
y^{\prime }-z^{\prime }%
\end{array}%
\right) ,
\end{eqnarray*}%
where $\mathbf{I}_{n-1}$ denotes the $\left( n-1\right) \times \left(
n-1\right) $ identity matrix. Thus we have%
\begin{eqnarray*}
\frac{d}{dt}\Phi \left( t\right) &=&\frac{y^{1}-z^{1}}{\left\vert \Psi
\left( w_{t}\right) -\Psi \left( x\right) \right\vert ^{n+1-\alpha }}-\left(
n+1-\alpha \right) \left( w_{t}^{1}-x^{1}\right) \frac{\left(
w_{t}^{1}-x^{1}\right) \left( y^{1}-z^{1}\right) }{\left\vert \Psi \left(
w_{t}\right) -\Psi \left( x\right) \right\vert ^{n+3-\alpha }} \\
&&-\left( n+1-\alpha \right) \left( w_{t}^{1}-x^{1}\right) \frac{\left(
w_{t}^{\prime }-x^{\prime }+\psi \left( w_{t}^{1}\right) -\psi \left(
x^{1}\right) \right) \cdot \left( y^{\prime }-z^{\prime }\right) }{%
\left\vert \Psi \left( w_{t}\right) -\Psi \left( x\right) \right\vert
^{n+3-\alpha }} \\
&&-\left( n+1-\alpha \right) \left( w_{t}^{1}-x^{1}\right) \frac{\left(
w_{t}^{\prime }-x^{\prime }+\psi \left( w_{t}^{1}\right) -\psi \left(
x^{1}\right) \right) \cdot D\psi \left( w_{t}^{1}\right) \left(
y^{1}-z^{1}\right) }{\left\vert \Psi \left( w_{t}\right) -\Psi \left(
x\right) \right\vert ^{n+3-\alpha }} \\
&=&\left( y^{1}-z^{1}\right) \left\{ \frac{\left\vert \Psi \left(
w_{t}\right) -\Psi \left( x\right) \right\vert ^{2}}{\left\vert \Psi \left(
w_{t}\right) -\Psi \left( x\right) \right\vert ^{n+3-\alpha }}-\left(
n+1-\alpha \right) \frac{\left\vert w_{t}^{1}-x^{1}\right\vert ^{2}}{%
\left\vert \Psi \left( w_{t}\right) -\Psi \left( x\right) \right\vert
^{n+3-\alpha }}\right\} \\
&&+\left( y^{1}-z^{1}\right) \left\{ -\left( n+1-\alpha \right) \left(
w_{t}^{1}-x^{1}\right) \frac{\left( w_{t}^{\prime }-x^{\prime }+\psi \left(
w_{t}^{1}\right) -\psi \left( x^{1}\right) \right) \cdot \left( \frac{%
y^{\prime }-z^{\prime }}{y^{1}-z^{1}}\right) }{\left\vert \Psi \left(
w_{t}\right) -\Psi \left( x\right) \right\vert ^{n+3-\alpha }}\right\} \\
&&+\left( y^{1}-z^{1}\right) \left\{ -\left( n+1-\alpha \right) \left(
w_{t}^{1}-x^{1}\right) \frac{\left( w_{t}^{\prime }-x^{\prime }+\psi \left(
w_{t}^{1}\right) -\psi \left( x^{1}\right) \right) \cdot D\psi \left(
w_{t}^{1}\right) }{\left\vert \Psi \left( w_{t}\right) -\Psi \left( x\right)
\right\vert ^{n+3-\alpha }}\right\} \\
&\equiv &\left( y^{1}-z^{1}\right) \left\{ A\left( t\right) +B\left(
t\right) +C\left( t\right) \right\} .
\end{eqnarray*}

From (\ref{also note}) we have $\left\Vert D\psi \right\Vert _{\infty }<%
\frac{1}{8n^{2}}\left( n-\alpha \right) $. Now $\left\vert
w_{t}^{1}-x^{1}\right\vert \approx \left\vert y-x\right\vert $ and $%
\left\vert w_{t}^{\prime }-x^{\prime }\right\vert =\left\vert w_{t}^{\prime
}\right\vert \lesssim \frac{\left\vert y-x\right\vert }{\gamma }$ because $%
x\in L\setminus \gamma J$ and $y,z\in J$ and $3J\cap L\neq \emptyset $, and
so we obtain from (\ref{assumptions}), with $\gamma =\gamma \left( n,\alpha
\right) $ sufficiently large, that both%
\begin{eqnarray}
\left\vert w_{t}^{\prime }-x^{\prime }\right\vert &\leq &\frac{1}{4}\sqrt{%
n-\alpha }\left\vert w_{t}^{1}-x^{1}\right\vert ,  \label{smallness} \\
\left\vert \psi \left( w_{t}^{1}\right) -\psi \left( x^{1}\right)
\right\vert &\leq &\left\Vert D\psi \right\Vert _{\infty }\left\vert
w_{t}^{1}-x^{1}\right\vert \leq \frac{1}{4}\sqrt{n-\alpha }\left\vert
w_{t}^{1}-x^{1}\right\vert .  \notag
\end{eqnarray}%
Hence we have%
\begin{eqnarray*}
-A\left( t\right) &=&-\frac{\left\vert \Psi \left( w_{t}\right) -\Psi \left(
x\right) \right\vert ^{2}}{\left\vert \Psi \left( w_{t}\right) -\Psi \left(
x\right) \right\vert ^{n+3-\alpha }}+\left( n+1-\alpha \right) \frac{%
\left\vert w_{t}^{1}-x^{1}\right\vert ^{2}}{\left\vert \Psi \left(
w_{t}\right) -\Psi \left( x\right) \right\vert ^{n+3-\alpha }} \\
&=&\frac{-\left\vert \Psi \left( w_{t}\right) -\Psi \left( x\right)
\right\vert ^{2}+\left( n+1-\alpha \right) \left\vert
w_{t}^{1}-x^{1}\right\vert ^{2}}{\left\vert \Psi \left( w_{t}\right) -\Psi
\left( x\right) \right\vert ^{n+3-\alpha }} \\
&=&\frac{-\left\vert w_{t}^{\prime }-x^{\prime }+\psi \left(
w_{t}^{1}\right) -\psi \left( x^{1}\right) \right\vert ^{2}+\left( n-\alpha
\right) \left( w_{t}^{1}-x^{1}\right) ^{2}}{\left\vert \Psi \left(
w_{t}\right) -\Psi \left( x\right) \right\vert ^{n+3-\alpha }} \\
&\geq &\frac{3}{4}\left( n-\alpha \right) \frac{\left(
w_{t}^{1}-x^{1}\right) ^{2}}{\left\vert \Psi \left( w_{t}\right) -\Psi
\left( x\right) \right\vert ^{n+3-\alpha }}\ ,
\end{eqnarray*}%
where the inequality in the final line holds because of (\ref{smallness}).
Note that we are able to control the sign of $A\left( t\right) $ above by
using the hypothesis that $\left\Vert D\psi \right\Vert _{\infty }$ is small
to keep $\left\vert \psi \left( w_{t}^{1}\right) -\psi \left( x^{1}\right)
\right\vert $ sufficiently small, and then using the hypothesis that $\gamma 
$ is large to keep $\left\vert w_{t}^{\prime }-x^{\prime }\right\vert $
sufficiently small, so that altogether $\left( n-\alpha \right) \left(
w_{t}^{1}-x^{1}\right) ^{2}$ is the dominant term in the numerator.

Now from our assumption (\ref{case 1}) and (\ref{also note}), i.e. $%
\left\Vert D\psi \right\Vert _{\infty }<\frac{1}{8n^{2}}\left( n-\alpha
\right) $, we have%
\begin{eqnarray*}
\left\vert B\left( t\right) \right\vert &=&\left\vert \left( n+1-\alpha
\right) \left( w_{t}^{1}-x^{1}\right) \frac{\left( w_{t}^{\prime }-x^{\prime
}+\psi \left( w_{t}^{1}\right) -\psi \left( x^{1}\right) \right) \cdot
\left( \frac{y^{\prime }-z^{\prime }}{y^{1}-z^{1}}\right) }{\left\vert \Psi
\left( w_{t}\right) -\Psi \left( x\right) \right\vert ^{n+3-\alpha }}%
\right\vert \\
&\leq &\left( n+1-\alpha \right) \left\vert w_{t}^{1}-x^{1}\right\vert \frac{%
\left( \left\vert w_{t}^{\prime }-x^{\prime }\right\vert +\left\Vert D\psi
\right\Vert _{\infty }\left\vert w_{t}^{1}-x^{1}\right\vert \right) \frac{%
\left\vert y^{\prime }-z^{\prime }\right\vert }{\left\vert
y^{1}-z^{1}\right\vert }}{\left\vert \Psi \left( w_{t}\right) -\Psi \left(
x\right) \right\vert ^{n+3-\alpha }} \\
&\leq &\left( n+1-\alpha \right) \frac{\left\vert w_{t}^{1}-x^{1}\right\vert
\left( \left\vert w_{t}^{\prime }-x^{\prime }\right\vert +\left\Vert D\psi
\right\Vert _{\infty }\left\vert w_{t}^{1}-x^{1}\right\vert \right) }{%
\left\vert \Psi \left( w_{t}\right) -\Psi \left( x\right) \right\vert
^{n+3-\alpha }} \\
&\leq &\left( n+1-\alpha \right) \left( \frac{a_{n,\alpha }}{\gamma }+\frac{1%
}{8n^{2}}\left( n-\alpha \right) \right) \frac{\left\vert
w_{t}^{1}-x^{1}\right\vert ^{2}}{\left\vert \Psi \left( w_{t}\right) -\Psi
\left( x\right) \right\vert ^{n+3-\alpha }} \\
&\leq &\frac{1}{4}\left( n-\alpha \right) \frac{\left(
w_{t}^{1}-x^{1}\right) ^{2}}{\left\vert \Psi \left( w_{t}\right) -\Psi
\left( x\right) \right\vert ^{n+3-\alpha }},
\end{eqnarray*}%
with a constant $a_{n,\alpha }$ independent of $\gamma $, provided 
\begin{equation}
\frac{a_{n,\alpha }}{\gamma }+\frac{1}{8n^{2}}\left( n-\alpha \right) \leq 
\frac{1}{4}\frac{n-\alpha }{n+1-\alpha },  \label{provided''}
\end{equation}%
which holds for $\gamma =\gamma \left( n,\alpha \right) $ sufficiently
large. We also have from the same calculation that%
\begin{eqnarray*}
\left\vert C\left( t\right) \right\vert &=&\left\vert \left( n+1-\alpha
\right) \left( w_{t}^{1}-x^{1}\right) \frac{\left( w_{t}^{\prime }-x^{\prime
}+\psi \left( w_{t}^{1}\right) -\psi \left( x^{1}\right) \right) \cdot D\psi
\left( w_{t}^{1}\right) }{\left\vert \Psi \left( w_{t}\right) -\Psi \left(
x\right) \right\vert ^{n+3-\alpha }}\right\vert \\
&\leq &\left( n+1-\alpha \right) \left\vert w_{t}^{1}-x^{1}\right\vert \frac{%
\left( \left\vert w_{t}^{\prime }-x^{\prime }\right\vert +\left\Vert D\psi
\right\Vert _{\infty }\left\vert w_{t}^{1}-x^{1}\right\vert \right)
\left\Vert D\psi \right\Vert _{\infty }}{\left\vert \Psi \left( w_{t}\right)
-\Psi \left( x\right) \right\vert ^{n+3-\alpha }} \\
&\leq &\frac{1}{4}\left( n-\alpha \right) \frac{\left(
w_{t}^{1}-x^{1}\right) ^{2}}{\left\vert \Psi \left( w_{t}\right) -\Psi
\left( x\right) \right\vert ^{n+3-\alpha }}\left\Vert D\psi \right\Vert
_{\infty } \\
&\leq &\frac{1}{4}\left( n-\alpha \right) \frac{\left(
w_{t}^{1}-x^{1}\right) ^{2}}{\left\vert \Psi \left( w_{t}\right) -\Psi
\left( x\right) \right\vert ^{n+3-\alpha }}\frac{1}{8n^{2}}\left( n-\alpha
\right) \\
&<&\frac{1}{4}\left( n-\alpha \right) \frac{\left( w_{t}^{1}-x^{1}\right)
^{2}}{\left\vert \Psi \left( w_{t}\right) -\Psi \left( x\right) \right\vert
^{n+3-\alpha }}.
\end{eqnarray*}%
Thus altogether in case (\ref{case 1}) we have 
\begin{eqnarray*}
&&\left\vert \left( \mathbf{R}_{\Psi }^{\alpha ,n}\right) _{1}\left( \mathbf{%
1}_{I\setminus \gamma J}\omega \right) \left( y^{1},y^{\prime }\right)
-\left( \mathbf{R}_{\Psi }^{\alpha ,n}\right) _{1}\left( \mathbf{1}%
_{I\setminus \gamma J}\omega \right) \left( z^{1},z^{\prime }\right)
\right\vert \\
&=&\left\vert y^{1}-z^{1}\right\vert \left\vert \int_{I\setminus \gamma J}%
\frac{\int_{0}^{1}\frac{d}{dt}\Phi \left( t\right) dt}{y^{1}-z^{1}}d\omega
\left( x\right) \right\vert \\
&=&\left\vert y^{1}-z^{1}\right\vert \left\vert \int_{I\setminus \gamma
J}\int_{0}^{1}\left\{ A\left( t\right) +B\left( t\right) +C\left( t\right)
\right\} dtd\omega \left( x\right) \right\vert \\
&\gtrsim &\left\vert y^{1}-z^{1}\right\vert \left\vert \int_{I\setminus
\gamma J}\int_{0}^{1}\left\{ \left( n-\alpha \right) \frac{\left(
w_{t}^{1}-x^{1}\right) ^{2}}{\left\vert \Psi \left( w_{t}\right) -\Psi
\left( x\right) \right\vert ^{n+3-\alpha }}\right\} dtd\omega \left(
x\right) \right\vert \\
&\approx &\left\vert y^{1}-z^{1}\right\vert \left( n-\alpha \right)
\int_{I\setminus \gamma J}\frac{\left( c_{J}^{1}-x^{1}\right) ^{2}}{%
\left\vert c_{J}-x\right\vert ^{n+3-\alpha }}d\omega \left( x\right) \\
&\approx &\left\vert y^{1}-z^{1}\right\vert \frac{\mathrm{P}^{\alpha }\left(
J,\mathbf{1}_{I\setminus \gamma J}\omega \right) }{\left\vert J\right\vert ^{%
\frac{1}{n}}}\ ,
\end{eqnarray*}%
where the constants implicit in $\approx $ depend only on $n$ and $\alpha $.

On the other hand, in the case that%
\begin{equation}
\left\vert y^{\prime }-z^{\prime }\right\vert >\left\vert
y^{1}-z^{1}\right\vert ,  \label{case 2}
\end{equation}%
we write%
\begin{eqnarray*}
\left( \mathbf{R}^{\alpha ,n}\right) ^{\prime } &=&\left( R_{2}^{\alpha
,n},...,R_{n}^{\alpha ,n}\right) , \\
\mathbf{\Phi }\left( t\right) &=&\frac{w_{t}^{\prime }-x^{\prime }}{%
\left\vert \Psi \left( w_{t}\right) -\Psi \left( x\right) \right\vert
^{n+1-\alpha }},
\end{eqnarray*}%
with $w_{t}=ty+\left( 1-t\right) z$ as before. Then as above we obtain 
\begin{equation*}
\frac{y^{\prime }-x^{\prime }}{\left\vert \Psi \left( y\right) -\Psi \left(
x\right) \right\vert ^{n+1-\alpha }}-\frac{z^{\prime }-x^{\prime }}{%
\left\vert \Psi \left( z\right) -\Psi \left( x\right) \right\vert
^{n+1-\alpha }}=\mathbf{\Phi }\left( 1\right) -\mathbf{\Phi }\left( 0\right)
=\int_{0}^{1}\frac{d}{dt}\mathbf{\Phi }\left( t\right) dt,
\end{equation*}%
where if we write $\widehat{y^{k}}\equiv \left(
0,y^{2},...,y^{k-1},0,y^{k+1},...,y^{n}\right) $, we have, similarly to the
computation of $\frac{d}{dt}\Phi \left( t\right) $ above,%
\begin{eqnarray*}
&&\frac{d}{dt}\mathbf{\Phi }\left( t\right) \equiv \left\{ \frac{d}{dt}\Phi
_{k}\left( t\right) \right\} _{k=2}^{n} \\
&=&\left\{ \left( y^{k}-z^{k}\right) \left[ \frac{\left\vert \Psi \left(
w_{t}\right) -\Psi \left( x\right) \right\vert ^{2}}{\left\vert \Psi \left(
w_{t}\right) -\Psi \left( x\right) \right\vert ^{n+3-\alpha }}-\left(
n+1-\alpha \right) \frac{\left( w_{t}^{k}-x^{k}\right) \left[
w_{t}^{k}-x^{k}+\psi ^{k}\left( w_{t}^{1}\right) -\psi ^{k}\left(
x^{1}\right) \right] }{\left\vert \Psi \left( w_{t}\right) -\Psi \left(
x\right) \right\vert ^{n+3-\alpha }}\right] \right\} _{k=2}^{n} \\
&&-\left\{ \left( n+1-\alpha \right) \left( w_{t}^{k}-x^{k}\right) \frac{%
\left( w_{t}^{1}-x^{1}\right) \left( y^{1}-z^{1}\right) +\left( \widehat{%
w_{t}^{k}}-\widehat{x^{k}}+\widehat{\psi ^{k}}\left( w_{t}^{1}\right) -%
\widehat{\psi ^{k}}\left( x^{1}\right) \right) \cdot \left( \widehat{y^{k}}-%
\widehat{z^{k}}\right) }{\left\vert \Psi \left( w_{t}\right) -\Psi \left(
x\right) \right\vert ^{n+3-\alpha }}\right\} _{k=2}^{n} \\
&&-\left\{ \left( y^{1}-z^{1}\right) \left[ -\left( n+1-\alpha \right)
\left( w_{t}^{k}-x^{k}\right) \frac{\left( w_{t}^{\prime }-x^{\prime }+\psi
\left( w_{t}^{1}\right) -\psi \left( x^{1}\right) \right) \cdot D\psi \left(
w_{t}^{1}\right) }{\left\vert \Psi \left( w_{t}\right) -\Psi \left( x\right)
\right\vert ^{n+3-\alpha }}\right] \right\} _{k=2}^{n} \\
&\equiv &\left\{ \left( y^{k}-z^{k}\right) A_{k}\left( t\right) \right\}
_{k=2}^{n}+\left\{ V_{k}\left( t\right) \right\} _{k=2}^{n}+\left\{ \left(
y^{1}-z^{1}\right) C_{k}\left( t\right) \right\} _{k=2}^{n}\equiv \mathbf{U}%
\left( t\right) +\mathbf{V}\left( t\right) +\mathbf{W}\left( t\right) \ .
\end{eqnarray*}%
Now for $2\leq k\leq n$ we have $x^{k}=0$ and so%
\begin{eqnarray}
&&  \label{Ak} \\
A_{k}\left( t\right) &=&\frac{\left\vert \Psi \left( w_{t}\right) -\Psi
\left( x\right) \right\vert ^{2}}{\left\vert \Psi \left( w_{t}\right) -\Psi
\left( x\right) \right\vert ^{n+3-\alpha }}-\left( n+1-\alpha \right) \frac{%
w_{t}^{k}\left[ w_{t}^{k}+\psi ^{k}\left( w_{t}^{1}\right) -\psi ^{k}\left(
x^{1}\right) \right] }{\left\vert \Psi \left( w_{t}\right) -\Psi \left(
x\right) \right\vert ^{n+3-\alpha }}  \notag \\
&=&\frac{\left\vert \Psi \left( w_{t}\right) -\Psi \left( x\right)
\right\vert ^{2}-\left( n+1-\alpha \right) w_{t}^{k}\left[ w_{t}^{k}+\psi
^{k}\left( w_{t}^{1}\right) -\psi ^{k}\left( x^{1}\right) \right] }{%
\left\vert \Psi \left( w_{t}\right) -\Psi \left( x\right) \right\vert
^{n+3-\alpha }}  \notag \\
&=&\frac{\left\vert w_{t}^{1}-x^{1}\right\vert ^{2}+\sum_{j\neq 1}\left\vert
w_{t}^{j}+\psi ^{j}\left( w_{t}^{1}\right) -\psi ^{j}\left( x^{1}\right)
\right\vert ^{2}-\left( n+1-\alpha \right) w_{t}^{k}\left( w_{t}^{k}+\psi
^{k}\left( w_{t}^{1}\right) -\psi ^{k}\left( x^{1}\right) \right) }{%
\left\vert \Psi \left( w_{t}\right) -\Psi \left( x\right) \right\vert
^{n+3-\alpha }}.  \notag
\end{eqnarray}%
Then using $\left\vert w_{t}^{k}\right\vert \lesssim \frac{1}{\gamma }%
\left\vert w_{t}^{1}-x^{1}\right\vert $ and $\left\vert \psi ^{k}\left(
w_{t}^{1}\right) -\psi ^{k}\left( x^{1}\right) \right\vert \leq \left\Vert
D\psi \right\Vert _{\infty }\left\vert w_{t}^{1}-x^{1}\right\vert $, we claim%
\begin{equation*}
A_{k}\left( t\right) \geq \frac{1}{2}\frac{\left\vert
w_{t}^{1}-x^{1}\right\vert ^{2}}{\left\vert \Psi \left( w_{t}\right) -\Psi
\left( x\right) \right\vert ^{n+3-\alpha }},
\end{equation*}%
where $\left\Vert D\psi \right\Vert _{\infty }$ satisfies (\ref{assumptions}%
) and $\gamma =\gamma \left( n,\alpha \right) $ is sufficiently large.
Indeed, use $\left\vert w_{t}^{k}\right\vert \leq \frac{b_{n,\alpha }}{%
\gamma }\left\vert w_{t}^{1}-x^{1}\right\vert $, where the constant $%
b_{n,\alpha }$ is independent of $\gamma $, to obtain%
\begin{eqnarray*}
&&\left\vert \Psi \left( w_{t}\right) -\Psi \left( x\right) \right\vert
^{n+3-\alpha }A_{k}\left( t\right) \\
&\geq &\left\vert w_{t}^{1}-x^{1}\right\vert ^{2}-\left( n+1-\alpha \right)
\left\vert w_{t}^{k}\right\vert \left\vert w_{t}^{k}+\psi ^{k}\left(
w_{t}^{1}\right) -\psi ^{k}\left( x^{1}\right) \right\vert \\
&\geq &\left\vert w_{t}^{1}-x^{1}\right\vert ^{2}-\left( n+1-\alpha \right) 
\left[ \left\vert w_{t}^{k}\right\vert ^{2}+\left\vert w_{t}^{k}\right\vert
\left\Vert D\psi \right\Vert _{\infty }\left\vert w_{t}^{1}-x^{1}\right\vert %
\right] \\
&\geq &\left\vert w_{t}^{1}-x^{1}\right\vert ^{2}-\left( n+1-\alpha \right) 
\left[ \left( \frac{b_{n,\alpha }}{\gamma }\right) ^{2}+\frac{b_{n,\alpha }}{%
\gamma }\left\Vert D\psi \right\Vert _{\infty }\right] \left\vert
w_{t}^{1}-x^{1}\right\vert ^{2} \\
&\geq &\frac{1}{2}\left\vert w_{t}^{1}-x^{1}\right\vert ^{2},
\end{eqnarray*}
for $\gamma =\gamma \left( n,\alpha \right) $ sufficiently large since $%
\left\Vert D\psi \right\Vert _{\infty }<\frac{n-\alpha }{8n^{2}}$ by (\ref%
{assumptions}).

Thus we have%
\begin{equation*}
\int_{I\setminus \gamma J}A_{k}\left( t\right) d\omega \left( x\right) \geq 
\frac{1}{2}c_{n,\alpha }\int_{I\setminus \gamma J}\frac{1}{\left\vert
c_{J}-x\right\vert ^{n+1-\alpha }}d\omega \left( x\right) \geq c_{n,\alpha
}^{\prime }\frac{\mathrm{P}^{\alpha }\left( J,\mathbf{1}_{I\setminus \gamma
J}\omega \right) }{\left\vert J\right\vert ^{\frac{1}{n}}},
\end{equation*}%
where $c_{n,\alpha }^{\prime }$ is \emph{independent} of the choice of $%
\gamma =\gamma \left( n,\alpha \right) $, and hence%
\begin{eqnarray*}
\left\vert \int_{I\setminus \gamma J}\int_{0}^{1}\mathbf{U}\left( t\right)
dtd\omega \left( x\right) \right\vert ^{2} &=&\left\vert \int_{I\setminus
\gamma J}\int_{0}^{1}\left\{ \left( y^{k}-z^{k}\right) A_{k}\left( t\right)
\right\} _{k=2}^{n}dtd\omega \left( x\right) \right\vert ^{2} \\
&=&\sum_{k=2}^{n}\left( y^{k}-z^{k}\right) ^{2}\left\vert \int_{I\setminus
\gamma J}\int_{0}^{1}A_{k}\left( t\right) dtd\omega \left( x\right)
\right\vert ^{2} \\
&\geq &\left( c_{n,\alpha }^{\prime }\right) ^{2}\sum_{k=2}^{n}\left(
y^{k}-z^{k}\right) ^{2}\left( \frac{\mathrm{P}^{\alpha }\left( J,\mathbf{1}%
_{I\setminus \gamma J}\omega \right) }{\left\vert J\right\vert ^{\frac{1}{n}}%
}\right) ^{2} \\
&=&\left( c_{n,\alpha }^{\prime }\right) ^{2}\left\vert y^{\prime
}-z^{\prime }\right\vert ^{2}\left( \frac{\mathrm{P}^{\alpha }\left( J,%
\mathbf{1}_{I\setminus \gamma J}\omega \right) }{\left\vert J\right\vert ^{%
\frac{1}{n}}}\right) ^{2}.
\end{eqnarray*}%
For $2\leq k\leq n$ we also have using $x^{k}=0$ and (\ref{case 2}) that%
\begin{eqnarray*}
&&\frac{1}{n+1-\alpha }\left\vert V_{k}\left( t\right) \right\vert \\
&=&\left\vert \left( w_{t}^{k}-x^{k}\right) \frac{\left(
w_{t}^{1}-x^{1}\right) \left( y^{1}-z^{1}\right) +\left( \widehat{w_{t}^{k}}-%
\widehat{x^{k}}+\widehat{\psi ^{k}}\left( w_{t}^{1}\right) -\widehat{\psi
^{k}}\left( x^{1}\right) \right) \cdot \left( \widehat{y^{k}}-\widehat{z^{k}}%
\right) }{\left\vert \Psi \left( w_{t}\right) -\Psi \left( x\right)
\right\vert ^{n+3-\alpha }}\right\vert \\
&\leq &\left\vert w_{t}^{k}\right\vert \frac{\left\vert
w_{t}^{1}-x^{1}\right\vert \left\vert y^{1}-z^{1}\right\vert +\sum_{j\neq
1,k}\left\vert w_{t}^{j}+\psi ^{j}\left( w_{t}^{1}\right) -\psi ^{j}\left(
x^{1}\right) \right\vert \left\vert y^{j}-z^{j}\right\vert }{\left\vert \Psi
\left( w_{t}\right) -\Psi \left( x\right) \right\vert ^{n+3-\alpha }} \\
&\leq &\left\{ \frac{\left\vert w_{t}^{k}\right\vert \left\vert
y^{1}-z^{1}\right\vert }{\left\vert \Psi \left( w_{t}\right) -\Psi \left(
x\right) \right\vert ^{n+2-\alpha }}+\sum_{j\neq 1,k}\frac{\left\vert
w_{t}^{k}\right\vert \left( \left\vert w_{t}^{j}\right\vert +\left\Vert
D\psi \right\Vert _{\infty }\left\vert w_{t}^{1}-x^{1}\right\vert \right)
\left\vert y^{j}-z^{j}\right\vert }{\left\vert \Psi \left( w_{t}\right)
-\Psi \left( x\right) \right\vert ^{n+3-\alpha }}\right\} \\
&\lesssim &\left\{ \frac{b_{n,\alpha }}{\gamma }\frac{\left\vert
y^{1}-z^{1}\right\vert }{\left\vert \Psi \left( w_{t}\right) -\Psi \left(
x\right) \right\vert ^{n+1-\alpha }}+\left( \sqrt{n}\ell \left( J\right)
\right) \left( \sqrt{n}\ell \left( J\right) +\left\Vert D\psi \right\Vert
_{\infty }\left\vert w_{t}^{1}-x^{1}\right\vert \right) \frac{\left\vert
y^{\prime }-z^{\prime }\right\vert }{\left\vert \Psi \left( w_{t}\right)
-\Psi \left( x\right) \right\vert ^{n+3-\alpha }}\right\} \\
&\lesssim &\left\{ \frac{b_{n,\alpha }}{\gamma }\frac{\left\vert y^{\prime
}-z^{\prime }\right\vert }{\left\vert c_{J}-x\right\vert ^{n+1-\alpha }}%
+\left( \frac{b_{n,\alpha }}{\gamma }\right) \left( \frac{b_{n,\alpha }}{%
\gamma }+\left\Vert D\psi \right\Vert _{\infty }\right) \frac{\left\vert
y^{\prime }-z^{\prime }\right\vert }{\left\vert c_{J}-x\right\vert
^{n+1-\alpha }}\right\} ,
\end{eqnarray*}%
as well as%
\begin{eqnarray*}
&&\frac{1}{n+1-\alpha }\left\vert W_{k}\left( t\right) \right\vert =\frac{1}{%
n+1-\alpha }\left\vert y^{1}-z^{1}\right\vert \left\vert C_{k}\left(
t\right) \right\vert \\
&=&\left\vert y^{1}-z^{1}\right\vert \left\vert w_{t}^{k}-x^{k}\right\vert
\left\vert \frac{\left( w_{t}^{\prime }-x^{\prime }+\psi \left(
w_{t}^{1}\right) -\psi \left( x^{1}\right) \right) \cdot D\psi \left(
w_{t}^{1}\right) }{\left\vert \Psi \left( w_{t}\right) -\Psi \left( x\right)
\right\vert ^{n+3-\alpha }}\right\vert \\
&\leq &\left\vert y^{1}-z^{1}\right\vert \left\vert w_{t}^{k}\right\vert 
\frac{\left( \left\vert w_{t}^{\prime }\right\vert +\left\vert \psi \left(
w_{t}^{1}\right) -\psi \left( x^{1}\right) \right\vert \right) \left\vert
D\psi \left( w_{t}^{1}\right) \right\vert }{\left\vert \Psi \left(
w_{t}\right) -\Psi \left( x\right) \right\vert ^{n+3-\alpha }} \\
&\leq &\left\vert y^{1}-z^{1}\right\vert \left( \frac{b_{n,\alpha }^{2}}{%
\gamma ^{2}}+\frac{b_{n,\alpha }\left\Vert D\psi \right\Vert _{\infty }}{%
\gamma }\right) \left\Vert D\psi \right\Vert _{\infty }\frac{\left\vert
w_{t}^{1}-x^{1}\right\vert ^{2}}{\left\vert \Psi \left( w_{t}\right) -\Psi
\left( x\right) \right\vert ^{n+3-\alpha }} \\
&\lesssim &\left\vert y^{1}-z^{1}\right\vert \left( \frac{1}{\gamma ^{2}}+%
\frac{\left\Vert D\psi \right\Vert _{\infty }}{\gamma }\right) \left\Vert
D\psi \right\Vert _{\infty }\frac{\left\vert w_{t}^{1}-x^{1}\right\vert }{%
\left\vert c_{J}-x\right\vert ^{n+1-\alpha }}\lesssim \frac{1}{\gamma }\frac{%
\left\vert y^{\prime }-z^{\prime }\right\vert }{\left\vert
c_{J}-x\right\vert ^{n+1-\alpha }}.
\end{eqnarray*}

Thus%
\begin{eqnarray*}
&&\left\vert \int_{I\setminus \gamma J}\int_{0}^{1}\mathbf{V}\left( t\right)
dtd\omega \left( x\right) \right\vert \\
&\leq &d_{n,\alpha }\left( n+1-\alpha \right) \left\{ \frac{1}{\gamma }%
+\left( \frac{1}{\gamma }\right) \left( \frac{1}{\gamma }+\left\Vert D\psi
\right\Vert _{\infty }\right) \right\} \int_{I\setminus \gamma J}\frac{%
\left\vert y^{\prime }-z^{\prime }\right\vert }{\left\vert
c_{J}-x\right\vert ^{n+1-\alpha }}d\omega \left( x\right) \\
&\leq &d_{n,\alpha }^{\prime }\left( n+1-\alpha \right) \left\{ \frac{1}{%
\gamma }+\left( \frac{1}{\gamma }\right) \left( \frac{1}{\gamma }+\frac{%
n-\alpha }{8n^{2}}\right) \right\} \left\vert y^{\prime }-z^{\prime
}\right\vert \frac{\mathrm{P}^{\alpha }\left( J,\mathbf{1}_{I\setminus
\gamma J}\omega \right) }{\left\vert J\right\vert ^{\frac{1}{n}}},
\end{eqnarray*}%
and%
\begin{equation*}
\left\vert \int_{I\setminus \gamma J}\int_{0}^{1}\mathbf{W}\left( t\right)
dtd\omega \left( x\right) \right\vert \leq \frac{d_{n,\alpha }^{\prime
\prime }}{\gamma }\frac{\left\vert y^{\prime }-z^{\prime }\right\vert }{%
\left\vert c_{J}-x\right\vert ^{n+1-\alpha }},
\end{equation*}%
where the constants $d_{n,\alpha }^{\prime }$ and $d_{n,\alpha }^{\prime
\prime }$ are independent of $\gamma $, and so we conclude that%
\begin{equation*}
\left\vert \int_{I\setminus \gamma J}\int_{0}^{1}\mathbf{V}\left( t\right)
dtd\omega \left( x\right) \right\vert +\left\vert \int_{I\setminus \gamma
J}\int_{0}^{1}\mathbf{W}\left( t\right) dtd\omega \left( x\right)
\right\vert \leq \frac{1}{2}\left\vert \int_{I\setminus \gamma J}\int_{0}^{1}%
\mathbf{U}\left( t\right) dtd\omega \left( x\right) \right\vert \ ,
\end{equation*}%
provided $\gamma =\gamma \left( n,\alpha \right) $ is sufficiently large.
Then if (\ref{case 2}) holds with $\gamma $ sufficiently large, we have 
\begin{eqnarray}
&&  \label{other reversal} \\
&&\left\vert \left( \mathbf{R}_{\Psi }^{\alpha ,n}\right) ^{\prime }\mathbf{1%
}_{I\setminus \gamma J}\omega \left( y^{1},y^{\prime }\right) -\left( 
\mathbf{R}_{\Psi }^{\alpha ,n}\right) ^{\prime }\mathbf{1}_{I\setminus
\gamma J}\omega \left( z^{1},z^{\prime }\right) \right\vert  \notag \\
&=&\left\vert \int_{I\setminus \gamma J}\left\{ \frac{y^{k}-x^{k}}{%
\left\vert \Psi \left( y\right) -\Psi \left( x\right) \right\vert
^{n+1-\alpha }}-\frac{z^{k}-x^{k}}{\left\vert \Psi \left( z\right) -\Psi
\left( x\right) \right\vert ^{n+1-\alpha }}\right\} _{k=2}^{n}d\omega \left(
x\right) \right\vert  \notag \\
&=&\left\vert \int_{I\setminus \gamma J}\int_{0}^{1}\mathbf{\Phi }^{\prime
}\left( t\right) dtd\omega \left( x\right) \right\vert  \notag \\
&\geq &\left\vert \int_{I\setminus \gamma J}\int_{0}^{1}\mathbf{U}\left(
t\right) dtd\omega \left( x\right) \right\vert -\left\vert \int_{I\setminus
\gamma J}\int_{0}^{1}\mathbf{V}\left( t\right) dtd\omega \left( x\right)
\right\vert -\left\vert \int_{I\setminus \gamma J}\int_{0}^{1}\mathbf{W}%
\left( t\right) dtd\omega \left( x\right) \right\vert  \notag \\
&\geq &\frac{1}{2}\left\vert \int_{I\setminus \gamma J}\int_{0}^{1}\mathbf{U}%
\left( t\right) dtd\omega \left( x\right) \right\vert  \notag \\
&\gtrsim &\left\vert y^{\prime }-z^{\prime }\right\vert \frac{\mathrm{P}%
^{\alpha }\left( J,\mathbf{1}_{I\setminus \gamma J}\omega \right) }{%
\left\vert J\right\vert ^{\frac{1}{n}}}\gtrsim \left\vert
y^{1}-z^{1}\right\vert \frac{\mathrm{P}^{\alpha }\left( J,\mathbf{1}%
_{I\setminus \gamma J}\omega \right) }{\left\vert J\right\vert ^{\frac{1}{n}}%
}\ .  \notag
\end{eqnarray}

Combining the inequalities from each case (\ref{case 1}) and (\ref{case 2})
above, and assuming $\gamma $ sufficiently large, we conclude that for all $%
y,z\in J$ we have the following `strong reversal' of the $1$-partial\
quasienergy,%
\begin{equation*}
\left\vert y^{1}-z^{1}\right\vert ^{2}\left( \frac{\mathrm{P}^{\alpha
}\left( J,\mathbf{1}_{I\setminus \gamma J}\omega \right) }{\left\vert
J\right\vert ^{\frac{1}{n}}}\right) ^{2}\lesssim \left\vert \mathbf{R}_{\Psi
}^{\alpha ,n}\mathbf{1}_{I\setminus \gamma J}\omega \left( y^{1},y^{\prime
}\right) -\mathbf{R}_{\Psi }^{\alpha ,n}\mathbf{1}_{I\setminus \gamma
J}\omega \left( z^{1},z^{\prime }\right) \right\vert ^{2}.
\end{equation*}%
Thus we have 
\begin{eqnarray*}
&&\sum_{\substack{ J\in \mathcal{M}_{\limfunc{deep}}  \\ 3J\cap L\neq
\emptyset }}\left( \frac{\mathrm{P}^{\alpha }\left( J,\mathbf{1}_{I\setminus
\gamma J}\omega \right) }{\left\vert J\right\vert ^{\frac{1}{n}}}\right)
^{2}\int_{J}\left\vert y^{1}-\mathbb{E}_{J}^{\sigma }y^{1}\right\vert
^{2}d\sigma \left( y\right) \\
&=&\frac{1}{2}\sum_{\substack{ J\in \mathcal{M}_{\limfunc{deep}}  \\ 3J\cap
L\neq \emptyset }}\left( \frac{\mathrm{P}^{\alpha }\left( J,\mathbf{1}%
_{I\setminus \gamma J}\omega \right) }{\left\vert J\right\vert ^{\frac{1}{n}}%
}\right) ^{2}\frac{1}{\left\vert J\right\vert _{\sigma }}\int_{J}\int_{J}%
\left( y^{1}-z^{1}\right) ^{2}d\sigma \left( y\right) d\sigma \left( z\right)
\\
&\lesssim &\sum_{\substack{ J\in \mathcal{M}_{\limfunc{deep}}  \\ 3J\cap
L\neq \emptyset }}\frac{1}{\left\vert J\right\vert _{\sigma }}%
\int_{J}\int_{J}\left\vert \mathbf{R}_{\Psi }^{\alpha ,n}\left( \mathbf{1}%
_{I\setminus \gamma J}\omega \right) \left( y^{1},y^{\prime }\right) -%
\mathbf{R}_{\Psi }^{\alpha ,n}\left( \mathbf{1}_{I\setminus \gamma J}\omega
\right) \left( z^{1},z^{\prime }\right) \right\vert ^{2}d\sigma \left(
y\right) d\sigma \left( z\right) \\
&\lesssim &\sum_{\substack{ J\in \mathcal{M}_{\limfunc{deep}}  \\ 3J\cap
L\neq \emptyset }}\int_{J}\left\vert \mathbf{R}_{\Psi }^{\alpha ,n}\left( 
\mathbf{1}_{I\setminus \gamma J}\omega \right) \left( y^{1},y^{\prime
}\right) \right\vert ^{2}d\sigma \left( y\right) \\
&\lesssim &\sum_{\substack{ J\in \mathcal{M}_{\limfunc{deep}}  \\ 3J\cap
L\neq \emptyset }}\int_{J}\left\vert \mathbf{R}_{\Psi }^{\alpha ,n}\left( 
\mathbf{1}_{I}\omega \right) \left( y^{1},y^{\prime }\right) \right\vert
^{2}d\sigma \left( y\right) +\sum_{\substack{ J\in \mathcal{M}_{\limfunc{deep%
}}  \\ 3J\cap L\neq \emptyset }}\int_{J}\left\vert \mathbf{R}_{\Psi
}^{\alpha ,n}\left( \mathbf{1}_{\gamma J}\omega \right) \left(
y^{1},y^{\prime }\right) \right\vert ^{2}d\sigma \left( y\right) ,
\end{eqnarray*}%
and now we obtain in the usual way that this is bounded by%
\begin{eqnarray*}
&&\int_{I}\left\vert \mathbf{R}_{\Psi }^{\alpha ,n}\left( \mathbf{1}%
_{I}\omega \right) \left( y^{1},y^{\prime }\right) \right\vert ^{2}d\sigma
\left( y\right) +\sum_{J\in \mathcal{M}}\left( \mathfrak{T}_{\mathbf{R}%
_{\Psi }^{\alpha ,n}}^{\func{dual}}\right) ^{2}\left\vert \gamma
J\right\vert _{\omega } \\
&\leq &\left( \mathfrak{T}_{\mathbf{R}_{\Psi }^{\alpha ,n}}^{\func{dual}%
}\right) ^{2}\left\vert I\right\vert _{\omega }+\beta \left( \mathfrak{T}_{%
\mathbf{R}_{\Psi }^{\alpha ,n}}^{\func{dual}}\right) ^{2}\left\vert
I\right\vert _{\omega }\lesssim \left( \mathfrak{T}_{\mathbf{R}_{\Psi
}^{\alpha ,n}}^{\func{dual}}\right) ^{2}\left\vert I\right\vert _{\omega }\ .
\end{eqnarray*}

Now we turn to the other partial quasienergies and begin with the estimate
that for $2\leq j\leq n$, we have the following `weak reversal' of energy,%
\begin{eqnarray}
&&\left\vert \left( \mathbf{R}_{\Psi }^{\alpha ,n}\right) _{j}\left( \mathbf{%
1}_{I\setminus \gamma J}\omega \right) \left( y\right) \right\vert
=\left\vert \int_{I\setminus \gamma J}\frac{y^{j}-0}{\left\vert \Psi \left(
y\right) -\Psi \left( x\right) \right\vert ^{n+1-\alpha }}d\omega \left(
x_{1},0...,0\right) \right\vert  \label{weak 1} \\
&=&\left\vert \frac{y^{j}}{\left\vert J\right\vert ^{\frac{1}{n}}}%
\int_{I\setminus \gamma J}\frac{\left\vert J\right\vert ^{\frac{1}{n}}}{%
\left\vert \Psi \left( y\right) -\Psi \left( x\right) \right\vert
^{n+1-\alpha }}d\omega \left( x_{1},0...,0\right) \right\vert \approx
\left\vert y^{j}\right\vert \frac{\mathrm{P}^{\alpha }\left( J,\mathbf{1}%
_{I\setminus \gamma J}\omega \right) }{\left\vert J\right\vert ^{\frac{1}{n}}%
}.  \notag
\end{eqnarray}%
Thus for $2\leq j\leq n$, we use $\int_{J}\left\vert y^{j}-\mathbb{E}%
_{J}^{\sigma }y^{j}\right\vert ^{2}d\sigma \left( y\right) \leq
\int_{J}\left\vert y^{j}\right\vert ^{2}d\sigma \left( y\right) $ to obtain
in the usual way%
\begin{eqnarray}
&&  \label{weak 2} \\
&&\sum_{\substack{ J\in \mathcal{M}_{\limfunc{deep}}  \\ 3J\cap L\neq
\emptyset }}\left( \frac{\mathrm{P}^{\alpha }\left( J,\mathbf{1}_{I\setminus
\gamma J}\omega \right) }{\left\vert J\right\vert ^{\frac{1}{n}}}\right)
^{2}\int_{J}\left\vert y^{j}-\mathbb{E}_{J}^{\sigma }y^{j}\right\vert
^{2}d\sigma \left( y\right)  \notag \\
&\leq &\sum_{\substack{ J\in \mathcal{M}_{\limfunc{deep}}  \\ 3J\cap L\neq
\emptyset }}\left( \frac{\mathrm{P}^{\alpha }\left( J,\mathbf{1}_{I\setminus
\gamma J}\omega \right) }{\left\vert J\right\vert ^{\frac{1}{n}}}\right)
^{2}\int_{J}\left\vert y^{j}\right\vert ^{2}d\sigma \left( y\right) =\sum 
_{\substack{ J\in \mathcal{M}_{\limfunc{deep}}  \\ 3J\cap L\neq \emptyset }}%
\int_{J}\left( \frac{\mathrm{P}^{\alpha }\left( J,\mathbf{1}_{I\setminus
\gamma J}\omega \right) }{\left\vert J\right\vert ^{\frac{1}{n}}}\right)
^{2}\left\vert y^{j}\right\vert ^{2}d\sigma \left( y\right)  \notag \\
&\lesssim &\sum_{\substack{ J\in \mathcal{M}_{\limfunc{deep}}  \\ 3J\cap
L\neq \emptyset }}\int_{J}\left\vert \left( \mathbf{R}_{\Psi }^{\alpha
,n}\right) _{j}\left( \mathbf{1}_{I\setminus \gamma J}\omega \right) \left(
y\right) \right\vert ^{2}d\sigma \left( y\right) \lesssim \left( \mathfrak{T}%
_{\mathbf{R}_{\Psi }^{\alpha ,n}}^{\Omega \mathcal{Q}^{n},\func{dual}%
}\right) ^{2}\left\vert I\right\vert _{\omega }+\sum_{J\in \mathcal{M}%
}\left( \mathfrak{T}_{\mathbf{R}_{\Psi }^{\alpha ,n}}^{\Omega \mathcal{Q}%
^{n},\func{dual}}\right) ^{2}\left\vert \gamma J\right\vert _{\omega } 
\notag \\
&\leq &\left( \mathfrak{T}_{\mathbf{R}_{\Psi }^{\alpha ,n}}^{\Omega \mathcal{%
Q}^{n},\func{dual}}\right) ^{2}\left\vert I\right\vert _{\omega }+\beta
\left( \mathfrak{T}_{\mathbf{R}_{\Psi }^{\alpha ,n}}^{\Omega \mathcal{Q}^{n},%
\func{dual}}\right) ^{2}\left\vert I\right\vert _{\omega }\lesssim \left( 
\mathfrak{T}_{\mathbf{R}_{\Psi }^{\alpha ,n}}^{\Omega \mathcal{Q}^{n},\func{%
dual}}\right) ^{2}\left\vert I\right\vert _{\omega }\ .  \notag
\end{eqnarray}%
Summing these estimates for $j=1$ and $2\leq j\leq n$ completes the proof of
the backward quasienergy condition $\mathcal{E}_{\alpha }^{\Omega \mathcal{Q}%
^{n},\func{dual}}\lesssim \mathfrak{T}_{\mathbf{R}_{\Psi }^{\alpha
,n}}^{\Omega \mathcal{Q}^{n},\func{dual}}$.

\subsection{Forward quasienergy condition\label{Subsec forward}}

Now we turn to proving the (forward) quasienergy condition $\mathcal{E}%
_{\alpha }^{\Omega \mathcal{Q}^{n}}\lesssim \mathfrak{T}_{\mathbf{R}_{\Psi
}^{\alpha ,n}}^{\Omega \mathcal{Q}^{n}}+\sqrt{\mathcal{A}_{2}^{\alpha }}$,
where $\mathcal{A}_{2}^{\alpha }$ is the Muckenhoupt condition with holes.
We must show%
\begin{equation*}
\sup_{\ell \geq 0}\sum_{r=1}^{\infty }\sum_{J\in \mathcal{M}_{\limfunc{deep}%
}^{\ell }\left( I_{r}\right) }\left( \frac{\mathrm{P}^{\alpha }\left( J,%
\mathbf{1}_{I\setminus J^{\ast }}\sigma \right) }{\left\vert J\right\vert ^{%
\frac{1}{n}}}\right) ^{2}\left\Vert \mathsf{P}_{J}^{\omega }\mathbf{x}%
\right\Vert _{L^{2}\left( \omega \right) }^{2}\leq \left( \left( \mathfrak{T}%
_{\mathbf{R}_{\Psi }^{\alpha ,n}}^{\Omega \mathcal{Q}^{n}}\right) ^{2}+%
\mathcal{A}_{2}^{\alpha }\right) \left\vert I\right\vert _{\sigma }\ ,
\end{equation*}%
for all partitions of a dyadic quasicube $I=\overset{\cdot }{%
\dbigcup\limits_{r\geq 1}}I_{r}$ into dyadic subquasicubes $I_{r}$. We again
fix $\ell \geq 0$ and suppress both $\ell $ and $\mathbf{r}$ in the notation 
$\mathcal{M}_{\limfunc{deep}}\left( I_{r}\right) =\mathcal{M}_{\mathbf{r}-%
\limfunc{deep}}^{\ell }\left( I_{r}\right) $. We may assume that all the
quasicubes $J$ intersect $\limfunc{supp}\omega $, hence that all the
quasicubes $I_{r}$ and $J$ intersect $L$, which contains $\limfunc{supp}%
\omega $. We must show%
\begin{equation*}
\sum_{r=1}^{\infty }\sum_{J\in \mathcal{M}_{\limfunc{deep}}\left(
I_{r}\right) }\left( \frac{\mathrm{P}^{\alpha }\left( J,\mathbf{1}%
_{I\setminus J^{\ast }}\sigma \right) }{\left\vert J\right\vert ^{\frac{1}{n}%
}}\right) ^{2}\left\Vert \mathsf{P}_{J}^{\omega }\mathbf{x}\right\Vert
_{L^{2}\left( \omega \right) }^{2}\leq \left( \left( \mathfrak{T}_{\mathbf{R}%
_{\Psi }^{\alpha ,n}}^{\Omega \mathcal{Q}^{n}}\right) ^{2}+\mathcal{A}%
_{2}^{\alpha }\right) \left\vert I\right\vert _{\sigma }\ .
\end{equation*}%
Let $\mathcal{M}_{\limfunc{deep}}=\dbigcup\limits_{r=1}^{\infty }\mathcal{M}%
_{\limfunc{deep}}\left( I_{r}\right) $ as above, and with $J^{\ast }=\gamma
J $ for each $J\in \mathcal{M}_{\limfunc{deep}}$, make the decomposition%
\begin{equation*}
I\setminus J^{\ast }=\text{\textsc{E}}\left( J^{\ast }\right) \dot{\cup}%
\text{\textsc{S}}\left( J^{\ast }\right)
\end{equation*}%
of $I\setminus J^{\ast }$ into \emph{end} \textsc{E}$\left( J^{\ast }\right) 
$ and \emph{side} \textsc{S}$\left( J^{\ast }\right) $ disjoint pieces
defined by%
\begin{eqnarray*}
\text{\textsc{E}}\left( J^{\ast }\right) &\equiv &\left( I\setminus J^{\ast
}\right) \cap \left\{ \left( y^{1},y^{\prime }\right) :\left\vert y^{\prime
}-c_{J}^{\prime }\right\vert \leq \frac{10}{\gamma }\left\vert
y^{1}-c_{J}^{1}\right\vert \right\} ; \\
\text{\textsc{S}}\left( J^{\ast }\right) &\equiv &\left( I\setminus J^{\ast
}\right) \setminus \text{\textsc{E}}\left( J^{\ast }\right) \ .
\end{eqnarray*}%
Then it suffices to show both 
\begin{eqnarray*}
A &\equiv &\sum_{J\in \mathcal{M}_{\limfunc{deep}}}\left( \frac{\mathrm{P}%
^{\alpha }\left( J,\mathbf{1}_{\text{\textsc{E}}\left( J^{\ast }\right)
}\sigma \right) }{\left\vert J\right\vert ^{\frac{1}{n}}}\right)
^{2}\left\Vert \mathsf{P}_{J}^{\omega }\mathbf{x}\right\Vert _{L^{2}\left(
\omega \right) }^{2}\leq \left( \left( \mathfrak{T}_{\mathbf{R}_{\Psi
}^{\alpha ,n}}^{\Omega \mathcal{Q}^{n}}\right) ^{2}+\mathcal{A}_{2}^{\alpha
}\right) \left\vert I\right\vert _{\sigma }\ , \\
B &\equiv &\sum_{J\in \mathcal{M}_{\limfunc{deep}}}\left( \frac{\mathrm{P}%
^{\alpha }\left( J,\mathbf{1}_{\text{\textsc{S}}\left( J^{\ast }\right)
}\sigma \right) }{\left\vert J\right\vert ^{\frac{1}{n}}}\right)
^{2}\left\Vert \mathsf{P}_{J}^{\omega }\mathbf{x}\right\Vert _{L^{2}\left(
\omega \right) }^{2}\leq \left( \left( \mathfrak{T}_{\mathbf{R}_{\Psi
}^{\alpha ,n}}^{\Omega \mathcal{Q}^{n}}\right) ^{2}+\mathcal{A}_{2}^{\alpha
}\right) \left\vert I\right\vert _{\sigma }\ .
\end{eqnarray*}%
Term $A$ will be estimated in analogy with the Hilbert transform estimate (%
\ref{Hilbert}), while term $B$ will \ be estimated by summing Poisson tails.
Both estimates rely heavily on the one-dimensional nature of the support of $%
\omega $, for example $\left\Vert \mathsf{P}_{J}^{\omega }\mathbf{x}%
\right\Vert _{L^{2}\left( \omega \right) }^{2}=\left\Vert \mathsf{P}%
_{J}^{\omega }x^{1}\right\Vert _{L^{2}\left( \omega \right) }^{2}$. Thus in
this quasienergy condition, there is only one nonvanishing partial
quasienergy, namely the $1$-partial quasienergy measured along the $x_{1}$%
-axis.

For $\left( x^{1},0^{\prime }\right) ,\left( z^{1},0^{\prime }\right) \in J$
in term $A$ we first claim the following `strong reversal' of quasienergy,%
\begin{eqnarray}
&&  \label{diff quotient} \\
&&\left\vert \frac{\left( \mathbf{R}_{\Psi }^{\alpha ,n}\right) _{1}\left( 
\mathbf{1}_{\text{\textsc{E}}\left( J^{\ast }\right) }\sigma \right) \left(
x^{1},0^{\prime }\right) -\left( \mathbf{R}_{\Psi }^{\alpha ,n}\right)
_{1}\left( \mathbf{1}_{\text{\textsc{E}}\left( J^{\ast }\right) }\sigma
\right) \left( z^{1},0^{\prime }\right) }{x^{1}-z^{1}}\right\vert  \notag \\
&=&\left\vert \int_{\text{\textsc{E}}\left( J^{\ast }\right) }\left\{ \frac{%
\left( \mathbf{K}_{\Psi }^{\alpha ,n}\right) _{1}\left( \left(
x^{1},0^{\prime }\right) ,y\right) -\left( \mathbf{K}_{\Psi }^{\alpha
,n}\right) _{1}\left( \left( z^{1},0^{\prime }\right) ,y\right) }{x^{1}-z^{1}%
}\right\} d\sigma \left( y\right) \right\vert  \notag \\
&=&\left\vert \int_{\text{\textsc{E}}\left( J^{\ast }\right) }\left\{ \frac{%
\frac{x^{1}-y^{1}}{\left( \left\vert x^{1}-y^{1}\right\vert ^{2}+\left\vert
\psi \left( x_{1}\right) -\psi \left( y_{1}\right) -y^{\prime }\right\vert
^{2}\right) ^{\frac{n+1-\alpha }{2}}}-\frac{z^{1}-y^{1}}{\left( \left\vert
z^{1}-y^{1}\right\vert ^{2}+\left\vert \psi \left( z_{1}\right) -\psi \left(
y_{1}\right) -y^{\prime }\right\vert ^{2}\right) ^{\frac{n+1-\alpha }{2}}}}{%
x^{1}-z^{1}}\right\} d\sigma \left( y\right) \right\vert  \notag \\
&\approx &\frac{\mathrm{P}^{\alpha }\left( J,\mathbf{1}_{\text{\textsc{E}}%
\left( J^{\ast }\right) }\sigma \right) }{\left\vert J\right\vert ^{\frac{1}{%
n}}}\ .  \notag
\end{eqnarray}%
Indeed, if we set $a\left( u\right) =\left\vert \psi \left( u+y_{1}\right)
-\psi \left( y_{1}\right) -y^{\prime }\right\vert $ and $s=x^{1}-y^{1}$ and $%
t=z^{1}-y^{1}$, then the term in braces in (\ref{diff quotient}) is%
\begin{eqnarray*}
&&\frac{\frac{x^{1}-y^{1}}{\left( \left\vert x^{1}-y^{1}\right\vert
^{2}+\left\vert \psi \left( x^{1}\right) -\psi \left( y^{1}\right)
-y^{\prime }\right\vert ^{2}\right) ^{\frac{n+1-\alpha }{2}}}-\frac{%
z^{1}-y^{1}}{\left( \left\vert z^{1}-y^{1}\right\vert ^{2}+\left\vert \psi
\left( z^{1}\right) -\psi \left( y^{1}\right) -y^{\prime }\right\vert
^{2}\right) ^{\frac{n+1-\alpha }{2}}}}{x^{1}-z^{1}} \\
&=&\frac{\frac{s}{\left( s^{2}+\left\vert \psi \left( s+y^{1}\right) -\psi
\left( y^{1}\right) -y^{\prime }\right\vert ^{2}\right) ^{\frac{n+1-\alpha }{%
2}}}-\frac{t}{\left( t^{2}+\left\vert \psi \left( t+y^{1}\right) -\psi
\left( y^{1}\right) -y^{\prime }\right\vert ^{2}\right) ^{\frac{n+1-\alpha }{%
2}}}}{s-t}=\frac{\varphi \left( s\right) -\varphi \left( t\right) }{s-t}\ ,
\end{eqnarray*}%
where with $y$ fixed for the moment, 
\begin{eqnarray*}
\varphi \left( u\right) &=&u\left( u^{2}+\left\vert \psi \left(
u+y^{1}\right) -\psi \left( y^{1}\right) -y^{\prime }\right\vert ^{2}\right)
^{-\frac{n+1-\alpha }{2}}=u\left( u^{2}+a\left( u\right) ^{2}\right) ^{-%
\frac{n+1-\alpha }{2}}; \\
a\left( u\right) ^{2} &\equiv &\left\vert \psi \left( u+y^{1}\right) -\psi
\left( y^{1}\right) -y^{\prime }\right\vert ^{2}.
\end{eqnarray*}%
Now the derivative of $\varphi \left( u\right) $ is%
\begin{eqnarray*}
\frac{d}{du}\varphi \left( u\right) &=&\left( u^{2}+a\left( u\right)
^{2}\right) ^{-\frac{n+1-\alpha }{2}}-\frac{n+1-\alpha }{2}\left(
u^{2}+a\left( u\right) ^{2}\right) ^{-\frac{n+1-\alpha }{2}-1}2\left[ u^{2}+u%
\frac{d}{du}\frac{1}{2}a\left( u\right) ^{2}\right] \\
&=&\left( u^{2}+a\left( u\right) ^{2}\right) ^{-\frac{n+1-\alpha }{2}%
-1}\left\{ \left( u^{2}+a\left( u\right) ^{2}\right) -\left( n+1-\alpha
\right) \left[ u^{2}+u\frac{d}{du}\frac{1}{2}a\left( u\right) ^{2}\right]
\right\} \\
&=&\left( u^{2}+a\left( u\right) ^{2}\right) ^{-\frac{n+1-\alpha }{2}%
-1}\left\{ \left[ a\left( u\right) ^{2}-\left( n+1-\alpha \right) u\frac{d}{%
du}\frac{1}{2}a\left( u\right) ^{2}\right] -\left( n-\alpha \right)
u^{2}\right\} ,
\end{eqnarray*}%
and the derivative of $a\left( u\right) ^{2}$ is%
\begin{eqnarray*}
\frac{d}{du}\frac{1}{2}a\left( u\right) ^{2} &=&\frac{1}{2}\frac{d}{du}%
\left\vert \psi \left( u+y_{1}\right) -\psi \left( y_{1}\right) -y^{\prime
}\right\vert ^{2} \\
&=&\left[ \psi \left( u+y_{1}\right) -\psi \left( y_{1}\right) -y^{\prime }%
\right] \cdot D\psi \left( u+y_{1}\right) .
\end{eqnarray*}%
We now want to conclude that%
\begin{equation}
\left\vert a\left( u\right) ^{2}-\left( n+1-\alpha \right) u\frac{d}{du}%
\frac{1}{2}a\left( u\right) ^{2}\right\vert \leq \frac{1}{2}\left( n-\alpha
\right) u^{2},\ \ \ \ \ \text{for }u\text{ between }s\text{ and }t,
\label{want to conclude}
\end{equation}%
so that $-\left( n-\alpha \right) u^{2}$ is the dominant term inside the
braces in the formula for $\frac{d}{du}\varphi \left( u\right) $. For this
we note that $\left\vert y^{\prime }\right\vert \leq C\frac{1}{\gamma }%
\left\vert u\right\vert $, and so using $2ab\leq \varepsilon a^{2}+\frac{1}{%
\varepsilon }b^{2}$ twice, the left side of (\ref{want to conclude}) is at
most%
\begin{eqnarray*}
&&\left( \left\Vert D\psi \right\Vert _{\infty }\left\vert u\right\vert
+\left\vert y^{\prime }\right\vert \right) ^{2}+\left( n+1-\alpha \right)
\left\vert u\right\vert \left( \left\Vert D\psi \right\Vert _{\infty
}\left\vert u\right\vert +\left\vert y^{\prime }\right\vert \right)
\left\Vert D\psi \right\Vert _{\infty } \\
&\leq &2\left\Vert D\psi \right\Vert _{\infty }^{2}u^{2}+2\left\vert
y^{\prime }\right\vert ^{2}+\left( n+1-\alpha \right) \left\Vert D\psi
\right\Vert _{\infty }^{2}u^{2}+\left( n+1-\alpha \right) \left\Vert D\psi
\right\Vert _{\infty }\left\vert u\right\vert \left\vert y^{\prime
}\right\vert \\
&\leq &\left( n+4-\alpha \right) \left\Vert D\psi \right\Vert _{\infty
}^{2}u^{2}+C_{1}\left( C\frac{1}{\gamma }\left\vert u\right\vert \right)
^{2},
\end{eqnarray*}%
which gives (\ref{want to conclude}) if%
\begin{equation*}
\left( n+4-\alpha \right) \left\Vert D\psi \right\Vert _{\infty
}^{2}+C_{1}C^{2}\frac{1}{\gamma ^{2}}\leq \frac{1}{2}\left( n-\alpha \right)
,
\end{equation*}%
which in turn holds for $\left\Vert D\psi \right\Vert _{\infty }$ and $%
\gamma =\gamma \left( n,\alpha \right) $ satisfying%
\begin{equation}
\left\Vert D\psi \right\Vert _{\infty }<\frac{1}{2}\sqrt{\frac{n-\alpha }{%
n+4-\alpha }}\text{ and }\gamma \gg \frac{1}{\sqrt{n-\alpha }},
\label{provided'}
\end{equation}%
where the first inequality in (\ref{provided'}) follows from (\ref%
{assumptions}), i.e. $\left\Vert D\psi \right\Vert _{\infty }<\frac{1}{8n^{2}%
}\left( n-\alpha \right) <\frac{1}{2}\sqrt{\frac{n-\alpha }{n+4-\alpha }}$,
and the second inequality holds for $\gamma =\gamma \left( n,\alpha \right) $
sufficiently large.

Thus we get%
\begin{equation*}
-\frac{d}{du}\varphi \left( u\right) \approx t^{2}\left( t^{2}+a\left(
t\right) ^{2}\right) ^{-\frac{n+1-\alpha }{2}-1},\ \ \ \ \ \text{for }u\text{
between }s\text{ and }t,
\end{equation*}%
where for $\left( x^{1},0^{\prime }\right) ,\left( z^{1},0^{\prime }\right)
\in J$ with $J\in \mathcal{M}_{\limfunc{deep}}$, the implied constants of
comparability are independent of $y\in $\textsc{E}$\left( J^{\ast }\right) $%
. Finally, since $\left\vert s-t\right\vert \lesssim \frac{1}{\gamma }%
\left\vert t\right\vert \ll \left\vert t\right\vert $, the derivative $\frac{%
d\varphi }{du}$ is essentially constant on the small interval $\left(
s,t\right) $, and we can apply the tangent line approximation to $\varphi $
to obtain $\varphi \left( s\right) -\varphi \left( t\right) \approx \frac{%
d\varphi }{dt}\left( t\right) \left( s-t\right) $, and conclude that for $%
\left( x^{1},0^{\prime }\right) ,\left( z^{1},0^{\prime }\right) \in J$, 
\begin{eqnarray*}
&&\left\vert \int_{\text{\textsc{E}}\left( J^{\ast }\right) }\left\{ \frac{%
\frac{x^{1}-y^{1}}{\left( \left\vert x^{1}-y^{1}\right\vert ^{2}+\left\vert
\psi \left( x^{1}\right) -\psi \left( y^{1}\right) -y^{\prime }\right\vert
^{2}\right) ^{\frac{n+1-\alpha }{2}}}-\frac{z^{1}-y^{1}}{\left( \left\vert
z^{1}-y^{1}\right\vert ^{2}+\left\vert \psi \left( z^{1}\right) -\psi \left(
y^{1}\right) -y^{\prime }\right\vert ^{2}\right) ^{\frac{n+1-\alpha }{2}}}}{%
x^{1}-z^{1}}\right\} d\sigma \left( y\right) \right\vert \\
&\approx &\int_{\text{\textsc{E}}\left( J^{\ast }\right) }\frac{\left\vert
z^{1}-y^{1}\right\vert ^{2}}{\left( \left\vert z^{1}-y^{1}\right\vert
^{2}+\left\vert \psi \left( z^{1}\right) -\psi \left( y^{1}\right)
-y^{\prime }\right\vert ^{2}\right) ^{\frac{n+1-\alpha }{2}+1}}d\sigma
\left( y\right) \\
&\approx &\int_{\text{\textsc{E}}\left( J^{\ast }\right) }\frac{\left\vert
z^{1}-y^{1}\right\vert ^{2}}{\left( \left\vert z^{1}-y^{1}\right\vert
^{2}+\left\vert y^{\prime }\right\vert ^{2}\right) ^{\frac{n+1-\alpha }{2}+1}%
}d\sigma \left( y\right) \approx \frac{\mathrm{P}^{\alpha }\left( J,\mathbf{1%
}_{\text{\textsc{E}}\left( J^{\ast }\right) }\sigma \right) }{\left\vert
J\right\vert ^{\frac{1}{n}}}\ ,
\end{eqnarray*}%
which proves (\ref{diff quotient}).

Thus we have%
\begin{eqnarray}
&&  \label{A123} \\
&&\sum_{J\in \mathcal{M}_{\limfunc{deep}}}\left( \frac{\mathrm{P}^{\alpha
}\left( J,\mathbf{1}_{\text{\textsc{E}}\left( J^{\ast }\right) }\sigma
\right) }{\left\vert J\right\vert ^{\frac{1}{n}}}\right) ^{2}\int_{J\cap
L}\left\vert x^{1}-\mathbb{E}_{J}^{\omega }x^{1}\right\vert ^{2}d\omega
\left( y\right)  \notag \\
&=&\frac{1}{2}\sum_{J\in \mathcal{M}_{\limfunc{deep}}}\left( \frac{\mathrm{P}%
^{\alpha }\left( J,\mathbf{1}_{\text{\textsc{E}}\left( J^{\ast }\right)
}\sigma \right) }{\left\vert J\right\vert ^{\frac{1}{n}}}\right) ^{2}\frac{1%
}{\left\vert J\cap L\right\vert _{\omega }}\int_{J\cap L}\int_{J\cap
L}\left( x^{1}-z^{1}\right) ^{2}d\omega \left( x\right) d\omega \left(
z\right)  \notag \\
&\approx &\sum_{J\in \mathcal{M}_{\limfunc{deep}}}\frac{1}{\left\vert
J\right\vert _{\omega }}\int_{J\cap L}\int_{J\cap L}\left\{ \left( \mathbf{R}%
_{\Psi }^{\alpha ,n}\right) _{1}\left( \mathbf{1}_{\text{\textsc{E}}\left(
J^{\ast }\right) }\sigma \right) \left( x^{1},0^{\prime }\right) -\left( 
\mathbf{R}_{\Psi }^{\alpha ,n}\right) _{1}\left( \mathbf{1}_{\text{\textsc{E}%
}\left( J^{\ast }\right) }\sigma \right) \left( z^{1},0^{\prime }\right)
\right\} ^{2}d\omega \left( x\right) d\omega \left( z\right)  \notag \\
&\lesssim &\sum_{J\in \mathcal{M}_{\limfunc{deep}}}\frac{1}{\left\vert
J\right\vert _{\omega }}\int_{J\cap L}\int_{J\cap L}\left\{ \left( \mathbf{R}%
_{\Psi }^{\alpha ,n}\right) _{1}\left( \mathbf{1}_{I}\sigma \right) \left(
x^{1},0^{\prime }\right) -\left( \mathbf{R}_{\Psi }^{\alpha ,n}\right)
_{1}\left( \mathbf{1}_{I}\sigma \right) \left( z^{1},0^{\prime }\right)
\right\} ^{2}d\omega \left( x\right) d\omega \left( z\right)  \notag \\
&&+\sum_{J\in \mathcal{M}_{\limfunc{deep}}}\frac{1}{\left\vert J\right\vert
_{\omega }}\int_{J\cap L}\int_{J\cap L}\left\{ \left( \mathbf{R}_{\Psi
}^{\alpha ,n}\right) _{1}\left( \mathbf{1}_{J^{\ast }}\sigma \right) \left(
x^{1},0^{\prime }\right) -\left( \mathbf{R}_{\Psi }^{\alpha ,n}\right)
_{1}\left( \mathbf{1}_{J^{\ast }}\sigma \right) \left( z^{1},0^{\prime
}\right) \right\} ^{2}d\omega \left( x\right) d\omega \left( z\right)  \notag
\\
&&+\sum_{J\in \mathcal{M}_{\limfunc{deep}}}\frac{1}{\left\vert J\right\vert
_{\omega }}\int_{J\cap L}\int_{J\cap L}\left\{ \left( \mathbf{R}_{\Psi
}^{\alpha ,n}\right) _{1}\left( \mathbf{1}_{\text{\textsc{S}}\left( J^{\ast
}\right) }\right) \sigma \left( x^{1},0^{\prime }\right) -\left( \mathbf{R}%
_{\Psi }^{\alpha ,n}\right) _{1}\left( \mathbf{1}_{\text{\textsc{S}}\left(
J^{\ast }\right) }\sigma \right) \left( z^{1},0^{\prime }\right) \right\}
^{2}d\omega \left( x\right) d\omega \left( z\right)  \notag \\
&\equiv &A_{1}+A_{2}+A_{3},  \notag
\end{eqnarray}%
since $I=J^{\ast }\dot{\cup}\left( I\setminus J^{\ast }\right) =J^{\ast }%
\dot{\cup}$\textsc{E}$\left( J^{\ast }\right) \dot{\cup}$\textsc{S}$\left(
J^{\ast }\right) $ where $\dot{\cup}$ denotes disjoint union. Now we can
discard the difference in term $A_{1}$ by writing 
\begin{equation*}
\left\vert \left( \mathbf{R}_{\Psi }^{\alpha ,n}\right) _{1}\left( \mathbf{1}%
_{I}\sigma \right) \left( x^{1},0^{\prime }\right) -\left( \mathbf{R}_{\Psi
}^{\alpha ,n}\right) _{1}\left( \mathbf{1}_{I}\sigma \right) \left(
z^{1},0^{\prime }\right) \right\vert \leq \left\vert \left( \mathbf{R}_{\Psi
}^{\alpha ,n}\right) _{1}\left( \mathbf{1}_{I}\sigma \right) \left(
x^{1},0^{\prime }\right) \right\vert +\left\vert \left( \mathbf{R}_{\Psi
}^{\alpha ,n}\right) _{1}\left( \mathbf{1}_{I}\sigma \right) \left(
z^{1},0^{\prime }\right) \right\vert
\end{equation*}%
to obtain from pairwise disjointedness of $J\in \mathcal{M}_{\limfunc{deep}}$%
, 
\begin{equation}
A_{1}\lesssim \sum_{J\in \mathcal{M}_{\limfunc{deep}}}\int_{J\cap
L}\left\vert \left( \mathbf{R}_{\Psi }^{\alpha ,n}\right) _{1}\left( \mathbf{%
1}_{I}\sigma \right) \left( x^{1},0^{\prime }\right) \right\vert ^{2}d\omega
\left( x\right) \leq \int_{I}\left\vert \left( \mathbf{R}_{\Psi }^{\alpha
,n}\right) _{1}\left( \mathbf{1}_{I}\sigma \right) \right\vert ^{2}d\omega
\leq \left( \mathfrak{T}_{\left( \mathbf{R}_{\Psi }^{\alpha ,n}\right)
_{1}}^{\Omega \mathcal{Q}^{n}}\right) ^{2}\left\vert I\right\vert _{\sigma
}\ ,  \label{forward A1}
\end{equation}%
and similarly we can discard the difference in term $A_{2}$, and use the
bounded overlap property (\ref{bounded overlap'}), to obtain%
\begin{eqnarray}
&&  \label{forward A2} \\
A_{2} &\lesssim &\sum_{J\in \mathcal{M}_{\limfunc{deep}}}\int_{J\cap
L}\left\vert \left( \mathbf{R}_{\Psi }^{\alpha ,n}\right) _{1}\left( \mathbf{%
1}_{J^{\ast }}\sigma \right) \left( x^{1},0^{\prime }\right) \right\vert
^{2}d\omega \left( x\right) \leq \sum_{J\in \mathcal{M}_{\limfunc{deep}%
}}\left( \mathfrak{T}_{\left( \mathbf{R}_{\Psi }^{\alpha ,n}\right)
_{1}}^{\Omega \mathcal{Q}^{n}}\right) ^{2}\left\vert J^{\ast }\right\vert
_{\sigma }  \notag \\
&=&\left( \mathfrak{T}_{\left( \mathbf{R}_{\Psi }^{\alpha ,n}\right)
_{1}}^{\Omega \mathcal{Q}^{n}}\right) ^{2}\sum_{r=1}^{\infty }\sum_{J\in 
\mathcal{M}_{\limfunc{deep}}\left( I_{r}\right) }\left\vert J^{\ast
}\right\vert _{\sigma }\leq \left( \mathfrak{T}_{\left( \mathbf{R}_{\Psi
}^{\alpha ,n}\right) _{1}}^{\Omega \mathcal{Q}^{n}}\right)
^{2}\sum_{r=1}^{\infty }\beta \left\vert I_{r}\right\vert _{\sigma }\leq
\beta \left( \mathfrak{T}_{\left( \mathbf{R}_{\Psi }^{\alpha ,n}\right)
_{1}}^{\Omega \mathcal{Q}^{n}}\right) ^{2}\left\vert I\right\vert _{\sigma
}\ .  \notag
\end{eqnarray}

This leaves us to consider the term%
\begin{eqnarray*}
A_{3} &=&\sum_{J\in \mathcal{M}_{\limfunc{deep}}}\frac{1}{\left\vert
J\right\vert _{\omega }}\int_{J\cap L}\int_{J\cap L}\left\{ \left( \mathbf{R}%
_{\Psi }^{\alpha ,n}\right) _{1}\left( \mathbf{1}_{\text{\textsc{S}}\left(
J^{\ast }\right) }\sigma \right) \left( x^{1},0^{\prime }\right) -\left( 
\mathbf{R}_{\Psi }^{\alpha ,n}\right) _{1}\left( \mathbf{1}_{\text{\textsc{S}%
}\left( J^{\ast }\right) }\sigma \right) \left( z^{1},0^{\prime }\right)
\right\} ^{2}d\omega \left( x\right) d\omega \left( z\right) \\
&=&2\sum_{J\in \mathcal{M}_{\limfunc{deep}}}\int_{J\cap L}\left\{ \left( 
\mathbf{R}_{\Psi }^{\alpha ,n}\right) _{1}\left( \mathbf{1}_{\text{\textsc{S}%
}\left( J^{\ast }\right) }\sigma \right) \left( x^{1},0^{\prime }\right) -%
\mathbb{E}_{J\cap L}^{\omega }\left[ \left( \mathbf{R}_{\Psi }^{\alpha
,n}\right) _{1}\left( \mathbf{1}_{\text{\textsc{S}}\left( J^{\ast }\right)
}\sigma \right) \left( z^{1},0^{\prime }\right) \right] \right\} ^{2}d\omega
\left( x\right) ,
\end{eqnarray*}%
in which we do \emph{not} discard the difference. However, because the
average is subtracted off, we can apply the Quasienergy Lemma \ref{ener},
together with duality 
\begin{eqnarray*}
\left\Vert \mathbf{R}_{\Psi }^{\alpha ,n}\left( \nu \right) -\mathbb{E}%
_{J}^{\omega }\mathbf{R}_{\Psi }^{\alpha ,n}\left( \nu \right) \right\Vert
_{L^{2}\left( \omega \right) } &=&\sup_{\left\Vert \Psi _{J}\right\Vert
_{L^{2}\left( \omega \right) }=1}\left\vert \left\langle \mathbf{R}_{\Psi
}^{\alpha ,n}\left( \nu \right) -\mathbb{E}_{J}^{\omega }\mathbf{R}_{\Psi
}^{\alpha ,n}\left( \nu \right) ,\Psi _{J}\right\rangle _{\omega }\right\vert
\\
&=&\sup_{\left\Vert \Psi _{J}\right\Vert _{L^{2}\left( \omega \right)
}=1}\left\vert \left\langle \mathbf{R}_{\Psi }^{\alpha ,n}\left( \nu \right)
,\Psi _{J}\right\rangle _{\omega }\right\vert ,
\end{eqnarray*}%
to each term in this sum to dominate it by,%
\begin{equation}
B=\sum_{J\in \mathcal{M}_{\limfunc{deep}}}\left( \frac{\mathrm{P}^{\alpha
}\left( J,\mathbf{1}_{\text{\textsc{S}}\left( J^{\ast }\right) }\sigma
\right) }{\left\vert J\right\vert ^{\frac{1}{n}}}\right) ^{2}\left\Vert 
\mathsf{P}_{J}^{\omega }\mathbf{x}\right\Vert _{L^{2}\left( \omega \right)
}^{2}\ .  \label{B1 + B2}
\end{equation}

To estimate $B$, we first assume that $n-1\leq \alpha <n$ so that $\mathrm{P}%
^{\alpha }\left( J,\mathbf{1}_{\text{\textsc{S}}\left( J^{\ast }\right)
}\sigma \right) \leq \mathcal{P}^{\alpha }\left( J,\mathbf{1}_{\text{\textsc{%
S}}\left( J^{\ast }\right) }\sigma \right) $, and then use $\left\Vert 
\mathsf{P}_{J}^{\omega }\mathbf{x}\right\Vert _{L^{2}\left( \omega \right)
}^{2}\lesssim \left\vert J\right\vert ^{\frac{2}{n}}\left\vert J\right\vert
_{\omega }$ and apply the $\mathcal{A}_{2}^{\alpha }$ condition with holes
to obtain the following `pivotal reversal' of quasienergy,%
\begin{eqnarray*}
B &\lesssim &\sum_{J\in \mathcal{M}_{\limfunc{deep}}}\mathrm{P}^{\alpha
}\left( J,\mathbf{1}_{\text{\textsc{S}}\left( J^{\ast }\right) }\sigma
\right) ^{2}\left\vert J\right\vert _{\omega }\leq \sum_{J\in \mathcal{M}_{%
\limfunc{deep}}}\mathrm{P}^{\alpha }\left( J,\mathbf{1}_{\text{\textsc{S}}%
\left( J^{\ast }\right) }\sigma \right) \left\{ \mathcal{P}^{\alpha }\left(
J,\mathbf{1}_{\text{\textsc{S}}\left( J^{\ast }\right) }\sigma \right)
\left\vert J\right\vert _{\omega }\right\} \\
&\leq &\mathcal{A}_{2}^{\alpha }\sum_{J\in \mathcal{M}_{\limfunc{deep}}}%
\mathrm{P}^{\alpha }\left( J,\mathbf{1}_{\text{\textsc{S}}\left( J^{\ast
}\right) }\sigma \right) \left\vert J\right\vert ^{1-\frac{\alpha }{n}}=%
\mathcal{A}_{2}^{\alpha }\sum_{J\in \mathcal{M}_{\limfunc{deep}}}\int_{\text{%
\textsc{S}}\left( J^{\ast }\right) }\frac{\left\vert J\right\vert ^{\frac{1}{%
n}}\left\vert J\right\vert ^{1-\frac{\alpha }{n}}}{\left( \left\vert
J\right\vert ^{\frac{1}{n}}+\left\vert y-c_{J}\right\vert \right)
^{n+1-\alpha }}d\sigma \left( y\right) \\
&=&\mathcal{A}_{2}^{\alpha }\sum_{J\in \mathcal{M}_{\limfunc{deep}}}\int_{%
\text{\textsc{S}}\left( J^{\ast }\right) }\left( \frac{\left\vert
J\right\vert ^{\frac{1}{n}}}{\left\vert J\right\vert ^{\frac{1}{n}%
}+\left\vert y-c_{J}\right\vert }\right) ^{n+1-\alpha }d\sigma \left(
y\right) \\
&=&\mathcal{A}_{2}^{\alpha }\int_{I}\left\{ \sum_{J\in \mathcal{M}_{\limfunc{%
deep}}}\left( \frac{\left\vert J\right\vert ^{\frac{1}{n}}}{\left\vert
J\right\vert ^{\frac{1}{n}}+\left\vert y-c_{J}\right\vert }\right)
^{n+1-\alpha }\mathbf{1}_{\text{\textsc{S}}\left( J^{\ast }\right) }\left(
y\right) \right\} d\sigma \left( y\right) \\
&\equiv &\mathcal{A}_{2}^{\alpha }\int_{I}F\left( y\right) d\sigma \left(
y\right) .
\end{eqnarray*}

At this point we claim that $F\left( y\right) \leq C$ with a constant $C$
independent of the decomposition $\mathcal{M}_{\limfunc{deep}}=\overset{%
\cdot }{\dbigcup\limits_{r\geq 1}}\mathcal{M}_{\limfunc{deep}}\left(
I_{r}\right) $. Indeed, if $y$ is fixed, then the quasicubes $J\in \mathcal{M%
}_{\limfunc{deep}}$ for which $y\in $\textsc{S}$\left( J^{\ast }\right) $
satisfy%
\begin{equation}
J\cap \limfunc{Sh}\left( y;\gamma \right) \neq \emptyset ,  \label{satis}
\end{equation}%
where $\limfunc{Sh}\left( y;\gamma \right) $ is the Carleson shadow of the
point $y$ onto the $x_{1}$-axis $L$, defined as the interval on $L$ with
length $\frac{1}{5}\gamma \limfunc{dist}\left( y,L\right) $ and center equal
to the point on $L$ that is closest to $y$. If a quasicube $J$ intersects $%
\limfunc{Sh}\left( y;\gamma \right) $, and $y\in $\textsc{S}$\left( J^{\ast
}\right) $, we must have%
\begin{equation*}
\ell \left( J\right) \leq C_{0}\limfunc{dist}\left( y,L\right) ,
\end{equation*}%
where $C_{0}=C_{0}\left( \gamma ,R_{\limfunc{big}}\right) $ is a positive
constant depending on $\gamma $ and the comparability constant $R_{\limfunc{%
big}}$ for $\Omega $ appearing in (\ref{comp contain}). We have thus shown
that $J\in \mathcal{M}_{\limfunc{deep}}$ and $y\in $\textsc{S}$\left(
J^{\ast }\right) $ imply $J\cap L\subset C\limfunc{Sh}\left( y;\gamma
\right) $ with $C=2C_{0}+1$.

Let $\mathcal{J}=J\cap L$ be the intersection of the quasicube $J$ with $L$,
and note that the linear measure of $\mathcal{J}$ satisfies $\left\vert 
\mathcal{J}\right\vert \lesssim \ell \left( J\right) $. Moreover, $\mathcal{J%
}$ need not be an interval if the quasicube's edge is close to being
parallel to $L$. Fix a point $y$. Then for a quasicube $J$ satisfying both (%
\ref{satis}) and $y\in $\textsc{S}$\left( J^{\ast }\right) $, the set $%
\mathcal{J}$ is contained inside the multiple $C\limfunc{Sh}\left( y;\gamma
\right) $ of the shadow, and 
\begin{equation*}
\left\vert y-c_{J}\right\vert \gtrsim \limfunc{dist}\left( y,L\right) .
\end{equation*}

Now we face two difficulties that do not arise for usual cubes with a side
parallel to $L$. First, as already mentioned, $\mathcal{J}$ need not be an
interval, and in fact may be a quite complicated set, and second that a
quasicube may intersect the line $L$ in a set having linear measure far less
than its side length, for example when a tilted cube intersects $L$ near a
vertex. Both of these difficulties are surmounted using the fact that the
quasicubes $J$ belong to a collection $\mathcal{M}_{\mathbf{r}-\limfunc{deep}%
}\left( I_{r}\right) $ for some $r$ (we are still suppressing the index $%
\ell $). Indeed, we first show that there is a positive constant $C^{\prime
} $ such that for each $r$ we have%
\begin{equation}
\sum_{\substack{ J\in \mathcal{M}_{\mathbf{r}-\limfunc{deep}}\left(
I_{r}\right)  \\ \emptyset \neq \mathcal{J}\subset C\limfunc{Sh}\left(
y;\gamma \right) ,y\in \text{\textsc{S}}\left( J^{\ast }\right) }}\ell
\left( J\right) \leq \beta \left\vert I_{r}\cap C^{\prime }\limfunc{Sh}%
\left( y;\gamma \right) \right\vert ,  \label{each r}
\end{equation}%
where $\beta $ is the constant appearing in the bounded overlap condition (%
\ref{bounded overlap'}). To see this, we note that if $\emptyset \neq 
\mathcal{J=}J\cap L$, then $J^{\ast }=\gamma J$ satisfies $\left\vert
J^{\ast }\cap L\right\vert \geq \ell \left( J\right) $ provided $\gamma $ is
large enough depending on the constant $R_{\limfunc{big}}$ in (\ref{comp
contain}), and we also have $J^{\ast }\subset C^{\prime }\limfunc{Sh}\left(
y;\gamma \right) $ where $C^{\prime }=C^{\prime }\left( C,\gamma ,R_{%
\limfunc{big}}\right) $ is a positive constant depending on $C$, $\gamma $
and $R_{\limfunc{big}}$. Altogether we thus have%
\begin{eqnarray*}
\sum_{\substack{ J\in \mathcal{M}_{\mathbf{r}-\limfunc{deep}}\left(
I_{r}\right)  \\ \emptyset \neq \mathcal{J}\subset C\limfunc{Sh}\left(
y;\gamma \right) ,y\in \text{\textsc{S}}\left( J^{\ast }\right) }}\ell
\left( J\right) &\leq &\sum_{\substack{ J\in \mathcal{M}_{\mathbf{r}-%
\limfunc{deep}}\left( I_{r}\right)  \\ \emptyset \neq \mathcal{J}\subset C%
\limfunc{Sh}\left( y;\gamma \right) ,y\in \text{\textsc{S}}\left( J^{\ast
}\right) }}\left\vert J^{\ast }\cap L\right\vert =\int_{L}\left( \sum 
_{\substack{ J\in \mathcal{M}_{\mathbf{r}-\limfunc{deep}}\left( I_{r}\right) 
\\ \emptyset \neq \mathcal{J}\subset C\limfunc{Sh}\left( y;\gamma \right)
,y\in \text{\textsc{S}}\left( J^{\ast }\right) }}1_{J^{\ast }}\right) dx \\
&\leq &\beta \int_{L\cap I_{r}}\mathbf{1}_{C^{\prime }\limfunc{Sh}\left(
y;\gamma \right) }dx\leq \beta \int_{C^{\prime }\limfunc{Sh}\left( y;\gamma
\right) \cap I_{r}}dx=\beta \left\vert C^{\prime }\limfunc{Sh}\left(
y;\gamma \right) \cap I_{r}\right\vert ,
\end{eqnarray*}%
which proves (\ref{each r}).

Now we continue with the estimate%
\begin{eqnarray}
&&  \label{continue} \\
&&\sum_{\substack{ J\in \mathcal{M}_{\limfunc{deep}}  \\ \mathcal{J}\subset C%
\limfunc{Sh}\left( y;\gamma \right) ,y\in \text{\textsc{S}}\left( J^{\ast
}\right) }}\left( \frac{\left\vert J\right\vert ^{\frac{1}{n}}}{\left\vert
J\right\vert ^{\frac{1}{n}}+\left\vert y-c_{J}\right\vert }\right)
^{n+1-\alpha }=\sum_{r=1}^{\infty }\sum_{\substack{ J\in \mathcal{M}_{%
\mathbf{r}-\limfunc{deep}}\left( I_{r}\right)  \\ \emptyset \neq \mathcal{J}%
\subset C\limfunc{Sh}\left( y;\gamma \right) ,y\in \text{\textsc{S}}\left(
J^{\ast }\right) }}\left( \frac{\left\vert J\right\vert ^{\frac{1}{n}}}{%
\left\vert J\right\vert ^{\frac{1}{n}}+\left\vert y-c_{J}\right\vert }%
\right) ^{n+1-\alpha }  \notag \\
&\lesssim &\sum_{r=1}^{\infty }\sum_{\substack{ J\in \mathcal{M}_{\mathbf{r}-%
\limfunc{deep}}\left( I_{r}\right)  \\ \emptyset \neq \mathcal{J}\subset C%
\limfunc{Sh}\left( y;\gamma \right) ,y\in \text{\textsc{S}}\left( J^{\ast
}\right) }}\left( \frac{\left\vert J\right\vert ^{\frac{1}{n}}}{\left\vert
y-c_{J}\right\vert }\right) ^{n-\alpha }\frac{\left\vert J\right\vert ^{%
\frac{1}{n}}}{\limfunc{dist}\left( y,L\right) }\lesssim \frac{1}{\limfunc{%
dist}\left( y,L\right) }\sum_{r=1}^{\infty }\left\{ \sum_{\substack{ J\in 
\mathcal{M}_{\mathbf{r}-\limfunc{deep}}\left( I_{r}\right)  \\ \emptyset
\neq \mathcal{J}\subset C\limfunc{Sh}\left( y;\gamma \right) ,y\in \text{%
\textsc{S}}\left( J^{\ast }\right) }}\left\vert J\right\vert ^{\frac{1}{n}%
}\right\}  \notag \\
&\lesssim &\frac{1}{\limfunc{dist}\left( y,L\right) }\sum_{r=1}^{\infty
}\beta \left\vert I_{r}\cap C^{\prime }\limfunc{Sh}\left( y;\gamma \right)
\right\vert \lesssim \beta \frac{1}{\limfunc{dist}\left( y,L\right) }%
\left\vert C^{\prime }\limfunc{Sh}\left( y;\gamma \right) \right\vert
\lesssim \beta \gamma ,  \notag
\end{eqnarray}%
because $n-\alpha >0$ and the sets $\left\{ I_{r}\cap C^{\prime }\limfunc{Sh}%
\left( y;\gamma \right) \right\} _{r=1}^{\infty }$ are pairwise disjoint in $%
C^{\prime }\limfunc{Sh}\left( y;\gamma \right) $. It is here that the
one-dimensional nature of $\omega $ permits the summing of the side lengths
of the quasicubes. Thus we have%
\begin{equation*}
B\leq \mathcal{A}_{2}^{\alpha }\int_{I}F\left( y\right) d\sigma \left(
y\right) \leq C\mathcal{A}_{2}^{\alpha }\left\vert I\right\vert _{\sigma }\ ,
\end{equation*}%
which is the desired estimate in the case that $n-1\leq \alpha <n$.

Now we suppose that $0\leq \alpha <n-1$ and use Cauchy-Schwarz to obtain%
\begin{eqnarray*}
\mathrm{P}^{\alpha }\left( J,\mathbf{1}_{\text{\textsc{S}}\left( J^{\ast
}\right) }\sigma \right) &=&\int_{\text{\textsc{S}}\left( J^{\ast }\right) }%
\frac{\left\vert J\right\vert ^{\frac{1}{n}}}{\left( \left\vert J\right\vert
^{\frac{1}{n}}+\left\vert y-c_{J}\right\vert \right) ^{n+1-\alpha }}d\sigma
\left( y\right) \\
&\leq &\left\{ \int_{\text{\textsc{S}}\left( J^{\ast }\right) }\frac{%
\left\vert J\right\vert ^{\frac{1}{n}}}{\left( \left\vert J\right\vert ^{%
\frac{1}{n}}+\left\vert y-c_{J}\right\vert \right) ^{n+1-\alpha }}\left( 
\frac{\left\vert J\right\vert ^{\frac{1}{n}}}{\left\vert J\right\vert ^{%
\frac{1}{n}}+\left\vert y-c_{J}\right\vert }\right) ^{n-1-\alpha }d\sigma
\left( y\right) \right\} ^{\frac{1}{2}} \\
&&\times \left\{ \int_{\text{\textsc{S}}\left( J^{\ast }\right) }\frac{%
\left\vert J\right\vert ^{\frac{1}{n}}}{\left( \left\vert J\right\vert ^{%
\frac{1}{n}}+\left\vert y-c_{J}\right\vert \right) ^{n+1-\alpha }}\left( 
\frac{\left\vert J\right\vert ^{\frac{1}{n}}}{\left\vert J\right\vert ^{%
\frac{1}{n}}+\left\vert y-c_{J}\right\vert }\right) ^{\alpha +1-n}d\sigma
\left( y\right) \right\} ^{\frac{1}{2}} \\
&=&\mathcal{P}^{\alpha }\left( J,\mathbf{1}_{\text{\textsc{S}}\left( J^{\ast
}\right) }\sigma \right) ^{\frac{1}{2}} \\
&&\times \left\{ \int_{\text{\textsc{S}}\left( J^{\ast }\right) }\frac{%
\left( \left\vert J\right\vert ^{\frac{1}{n}}\right) ^{\alpha +2-n}}{\left(
\left\vert J\right\vert ^{\frac{1}{n}}+\left\vert y-c_{J}\right\vert \right)
^{2}}d\sigma \left( y\right) \right\} ^{\frac{1}{2}}.
\end{eqnarray*}%
Then arguing as above we have%
\begin{eqnarray}
&&  \label{small alpha} \\
B &\leq &\sum_{J\in \mathcal{M}_{\limfunc{deep}}}\mathrm{P}^{\alpha }\left(
J^{\ast },\mathbf{1}_{\text{\textsc{S}}\left( J^{\ast }\right) }\sigma
\right) ^{2}\left\vert J\right\vert _{\omega }  \notag \\
&\leq &\sum_{J\in \mathcal{M}_{\limfunc{deep}}}\left\{ \mathcal{P}^{\alpha
}\left( J^{\ast },\mathbf{1}_{\text{\textsc{S}}\left( J^{\ast }\right)
}\sigma \right) \left\vert J\right\vert _{\omega }\right\} \int_{\text{%
\textsc{S}}\left( J^{\ast }\right) }\frac{\left( \left\vert J\right\vert ^{%
\frac{1}{n}}\right) ^{\alpha +2-n}}{\left( \left\vert J\right\vert ^{\frac{1%
}{n}}+\left\vert y-c_{J}\right\vert \right) ^{2}}d\sigma \left( y\right) 
\notag \\
&\leq &\mathcal{A}_{2}^{\alpha }\sum_{J\in \mathcal{M}_{\limfunc{deep}%
}}\left\vert J\right\vert ^{1-\frac{\alpha }{n}}\int_{\text{\textsc{S}}%
\left( J^{\ast }\right) }\frac{\left( \left\vert J\right\vert ^{\frac{1}{n}%
}\right) ^{\alpha +2-n}}{\left( \left\vert J\right\vert ^{\frac{1}{n}%
}+\left\vert y-c_{J}\right\vert \right) ^{2}}d\sigma \left( y\right)  \notag
\\
&=&\mathcal{A}_{2}^{\alpha }\sum_{J\in \mathcal{M}_{\limfunc{deep}}}\int_{%
\text{\textsc{S}}\left( J^{\ast }\right) }\frac{\left\vert J\right\vert ^{%
\frac{2}{n}}}{\left( \left\vert J\right\vert ^{\frac{1}{n}}+\left\vert
y-c_{J}\right\vert \right) ^{2}}d\sigma \left( y\right)  \notag \\
&=&\mathcal{A}_{2}^{\alpha }\int_{I}\left\{ \sum_{J\in \mathcal{M}_{\limfunc{%
deep}}}\left( \frac{\left\vert J\right\vert ^{\frac{1}{n}}}{\left\vert
J\right\vert ^{\frac{1}{n}}+\left\vert y-c_{J}\right\vert }\right) ^{2}%
\mathbf{1}_{\text{\textsc{S}}\left( J^{\ast }\right) }\left( y\right)
\right\} d\sigma \left( y\right)  \notag \\
&\equiv &\mathcal{A}_{2}^{\alpha }\int_{I}F\left( y\right) d\sigma \left(
y\right) ,  \notag
\end{eqnarray}%
and again $F\left( y\right) \leq C$ is the calculation above when $%
n+1-\alpha $ is replaced by $2$:%
\begin{eqnarray*}
\sum_{\substack{ J\in \mathcal{M}_{\limfunc{deep}}  \\ \mathcal{J}\subset C%
\limfunc{Sh}\left( y;\gamma \right) ,y\in \text{\textsc{S}}\left( J^{\ast
}\right) }}\left( \frac{\left\vert J\right\vert ^{\frac{1}{n}}}{\left\vert
J\right\vert ^{\frac{1}{n}}+\left\vert y-c_{J}\right\vert }\right) ^{2}
&=&\sum_{r=1}^{\infty }\sum_{\substack{ J\in \mathcal{M}_{\mathbf{r}-%
\limfunc{deep}}\left( I_{r}\right)  \\ \emptyset \neq \mathcal{J}\subset C%
\limfunc{Sh}\left( y;\gamma \right) ,y\in \text{\textsc{S}}\left( J^{\ast
}\right) }}\left( \frac{\left\vert J\right\vert ^{\frac{1}{n}}}{\left\vert
J\right\vert ^{\frac{1}{n}}+\left\vert y-c_{J}\right\vert }\right) ^{2} \\
&\lesssim &\sum_{r=1}^{\infty }\sum_{\substack{ J\in \mathcal{M}_{\mathbf{r}-%
\limfunc{deep}}\left( I_{r}\right)  \\ \emptyset \neq \mathcal{J}\subset C%
\limfunc{Sh}\left( y;\gamma \right) ,y\in \text{\textsc{S}}\left( J^{\ast
}\right) }}\frac{\left\vert J\right\vert ^{\frac{1}{n}}}{\left\vert
y-c_{J}\right\vert }\frac{\left\vert J\right\vert ^{\frac{1}{n}}}{\limfunc{%
dist}\left( y,L\right) } \\
&\lesssim &\frac{1}{\limfunc{dist}\left( y,L\right) }\sum_{r=1}^{\infty
}\left\{ \sum_{\substack{ J\in \mathcal{M}_{\mathbf{r}-\limfunc{deep}}\left(
I_{r}\right)  \\ \emptyset \neq \mathcal{J}\subset C\limfunc{Sh}\left(
y;\gamma \right) ,y\in \text{\textsc{S}}\left( J^{\ast }\right) }}\left\vert
J\right\vert ^{\frac{1}{n}}\right\} \\
&\lesssim &\frac{1}{\limfunc{dist}\left( y,L\right) }\sum_{r=1}^{\infty
}\beta \left\vert I_{r}\cap C^{\prime }\limfunc{Sh}\left( y;\gamma \right)
\right\vert \\
&\lesssim &\beta \frac{1}{\limfunc{dist}\left( y,L\right) }\left\vert
C^{\prime }\limfunc{Sh}\left( y;\gamma \right) \right\vert \lesssim \beta
\gamma .
\end{eqnarray*}%
Thus we again have%
\begin{equation*}
B\leq \mathcal{A}_{2}^{\alpha }\int_{I}F\left( y\right) d\sigma \left(
y\right) \leq C\mathcal{A}_{2}^{\alpha }\left\vert I\right\vert _{\sigma }\ ,
\end{equation*}%
and this completes the proof of necessity of the quasienergy conditions when
one of the measures is supported on a line.

\subsection{Backward triple testing and quasiweak boundedness property\label%
{Subsec triple}}

In this subsection we show that for a measure supported on a line, the
backward triple quasitesting condition,%
\begin{equation}
\int_{3Q^{\prime }}\left\vert \mathbf{R}_{\Psi }^{\alpha ,n}\left(
1_{Q^{\prime }}\omega \right) \right\vert ^{2}d\sigma \leq \mathfrak{T}_{%
\mathbf{R}_{\Psi }^{\alpha ,n}}^{\Omega \mathcal{Q}^{n},\limfunc{triple},%
\func{dual}}\left\vert Q^{\prime }\right\vert _{\omega }\ ,
\label{triple testing}
\end{equation}%
is controlled by the $\mathcal{A}_{2}^{\alpha }$ conditions with holes and
the two quasitesting conditions, namely%
\begin{equation*}
\mathfrak{T}_{\mathbf{R}_{\Psi }^{\alpha ,n}}^{\Omega \mathcal{Q}^{n},%
\limfunc{triple},\func{dual}}\lesssim \mathfrak{T}_{\mathbf{R}_{\Psi
}^{\alpha ,n}}^{\Omega \mathcal{Q}^{n},\func{dual}}+\sqrt{\mathcal{A}%
_{2}^{\alpha }}+\sqrt{\mathcal{A}_{2}^{\alpha ,\func{dual}}},
\end{equation*}%
provided that $\Omega $ is a $C^{1}$ diffeomorphism and $L$-transverse,
where $L$ is the support of $\omega $. It is then an easy consequence of the
Cauchy-Schwarz inequality that the weak boundedness property is also
controlled by the $A_{2}$ conditions and the two quasitesting conditions,%
\begin{equation*}
\int_{Q}\mathbf{R}_{\Psi }^{\alpha ,n}\left( 1_{Q^{\prime }}\omega \right)
d\sigma \lesssim \left( \mathfrak{T}_{\mathbf{R}_{\Psi }^{\alpha
,n}}^{\Omega \mathcal{Q}^{n},\func{dual}}+\sqrt{\mathcal{A}_{2}^{\alpha }}+%
\sqrt{\mathcal{A}_{2}^{\alpha ,\func{dual}}}\right) \sqrt{\left\vert
Q\right\vert _{\sigma }\left\vert Q^{\prime }\right\vert _{\omega }},
\end{equation*}%
for all pairs of quasicubes $Q$ and $Q^{\prime }$ of size comparable to
their distance apart.

We will use the following two properties of an $L$-transverse $C^{1}$
diffeomorphism $\Omega $ as defined in Definition \ref{L-trans} above.

\begin{lemma}
\label{connected}Suppose that $\Omega :\mathbb{R}^{n}\rightarrow \mathbb{R}%
^{n}$ is a $C^{1}$ diffeomorphism and $L$-transverse. Then if $\mathbf{e}%
_{L} $ is a unit vector in the direction of $L$, we have%
\begin{equation}
\left\vert \left\langle \frac{D\Omega ^{-1}\left( x\right) \mathbf{e}_{L}}{%
\left\vert D\Omega ^{-1}\left( x\right) \mathbf{e}_{L}\right\vert },\mathbf{e%
}_{k}\right\rangle \right\vert \leq \frac{1+\eta }{2},\ \ \ \ \ \text{for }%
x\in \mathbb{R}^{n}\text{ and }1\leq k\leq n,  \label{trans edges}
\end{equation}%
and 
\begin{equation}
Q\cap L\text{ is connected whenever }Q\in \Omega \mathcal{P}^{n}.
\label{conn}
\end{equation}
\end{lemma}

\begin{proof}
Choose a rotation $R\in \digamma _{\mathbf{e}_{L},\eta }$. Then we have%
\begin{eqnarray*}
\left\vert \left\langle \frac{D\Omega ^{-1}\left( x\right) \mathbf{e}_{L}}{%
\left\vert D\Omega ^{-1}\left( x\right) \mathbf{e}_{L}\right\vert }-D\Omega
^{-1}\left( x\right) \mathbf{e}_{L},\mathbf{e}_{k}\right\rangle \right\vert
&\leq &\left\vert \frac{D\Omega ^{-1}\left( x\right) \mathbf{e}_{L}}{%
\left\vert D\Omega ^{-1}\left( x\right) \mathbf{e}_{L}\right\vert }-D\Omega
^{-1}\left( x\right) \mathbf{e}_{L}\right\vert \\
&\leq &\left\vert R\mathbf{e}_{L}-D\Omega ^{-1}\left( x\right) \mathbf{e}%
_{L}\right\vert
\end{eqnarray*}%
since $\left\vert \frac{v}{\left\vert v\right\vert }-v\right\vert =\limfunc{%
dist}\left( v,\mathbb{S}^{n-1}\right) \leq \left\vert v-R\mathbf{e}%
_{L}\right\vert $. Thus using $\left\Vert D\Omega ^{-1}-R\right\Vert
_{\infty }<\frac{1-\eta }{4}$ from the definition of $L$-transverse, we
obtain%
\begin{eqnarray*}
&&\left\vert \left\langle \frac{D\Omega ^{-1}\left( x\right) \mathbf{e}_{L}}{%
\left\vert D\Omega ^{-1}\left( x\right) \mathbf{e}_{L}\right\vert },\mathbf{e%
}_{k}\right\rangle \right\vert \\
&=&\left\vert \left\langle \frac{D\Omega ^{-1}\left( x\right) \mathbf{e}_{L}%
}{\left\vert D\Omega ^{-1}\left( x\right) \mathbf{e}_{L}\right\vert }%
-D\Omega ^{-1}\left( x\right) \mathbf{e}_{L},\mathbf{e}_{k}\right\rangle
+\left\langle R\mathbf{e}_{L},\mathbf{e}_{k}\right\rangle +\left\langle
\left( D\Omega ^{-1}\left( x\right) -R\right) \mathbf{e}_{L},\mathbf{e}%
_{k}\right\rangle \right\vert \\
&\leq &\left\Vert D\Omega ^{-1}-R\right\Vert _{\infty }+\eta +\left\Vert
D\Omega ^{-1}-R\right\Vert _{\infty }<\frac{1+\eta }{2},
\end{eqnarray*}%
which proves (\ref{trans edges}).

Now suppose that $Q\in \Omega \mathcal{P}^{n}$ satisfies $Q\cap L\neq
\emptyset $. Then $\Omega ^{-1}L$ is the image of a differentiable curve.
Let $\varphi :\mathbb{R}\rightarrow \mathbb{R}^{n}$ be a parameterization of 
$\Omega ^{-1}L$. The tangent directions $\frac{D\varphi \left( t\right) }{%
\left\vert D\varphi \left( t\right) \right\vert }$ of the curve $\Omega
^{-1}L$ are given by $\frac{D\Omega ^{-1}\left( \varphi \left( t\right)
\right) \mathbf{e}_{L}}{\left\vert D\Omega ^{-1}\left( \varphi \left(
t\right) \right) \mathbf{e}_{L}\right\vert }$, which satisfy (\ref{trans
edges}). Set $K\equiv \Omega ^{-1}Q$ and note that $K$ is an ordinary half
open half closed cube in $\mathcal{P}^{n}$, which without loss of generality
we may take to be $K=\left[ -1,1\right) ^{n}$. Let $\alpha \equiv \inf $ $%
\varphi ^{-1}K$ and $\beta \equiv \sup $ $\varphi ^{-1}K$. Now assume in
order to derive a contradiction that $\Omega ^{-1}L\cap \Omega ^{-1}Q$ is
not connected. It follows from (\ref{trans edges}) that if $\varphi \left(
t\right) \in \partial K$, then the tangent line at $\varphi \left( t\right) $
intersects the complement of $\overline{K}$ in any neighbourhood of $\varphi
\left( t\right) $, and hence there is $t_{0}\in \left( \alpha ,\beta \right) 
$ such that $\varphi \left( t_{0}\right) =\left( \varphi _{1}\left(
t_{0}\right) ,...,\varphi _{n}\left( t_{0}\right) \right) \notin \overline{K}
$. Thus there is $k_{2}$ such that $\left\vert \varphi _{k_{2}}\left(
t_{0}\right) \right\vert >1$. Let $k_{1}$ be any index other than $k_{2}$.

Let $\mathsf{P}$ be orthogonal projection of $\mathbb{R}^{n}$ onto the $2$%
-plane $\Pi $\ spanned by $\mathbf{e}_{k_{1}}$ and $\mathbf{e}_{k_{2}}$.
Then $\mathsf{P}\varphi $ is a differentiable curve whose image lies in $\Pi 
$ and satisfies the following analogue of (\ref{trans edges}):%
\begin{equation}
\left\vert \left\langle \frac{D\mathsf{P}\varphi \left( t\right) }{%
\left\vert D\mathsf{P}\varphi \left( t\right) \right\vert },\mathbf{e}%
_{k}\right\rangle \right\vert =\left\vert \left\langle \frac{D\mathsf{P}%
\Omega ^{-1}\left( \varphi \left( t\right) \right) \mathbf{e}_{L}}{%
\left\vert D\mathsf{P}\Omega ^{-1}\left( \varphi \left( t\right) \right) 
\mathbf{e}_{L}\right\vert },\mathbf{e}_{k}\right\rangle \right\vert \leq 
\frac{1+\eta }{2}<1,\ \ \ \ \ \text{for }t\in \mathbb{R}\text{ and }%
k=k_{1},k_{2}\ .  \label{trans edges'}
\end{equation}%
Thus the image $\mathsf{P}\Omega ^{-1}L$ of the curve $\mathsf{P}\varphi $
may be written as the graph of a continuously differentiable function $f:%
\mathbb{R}\rightarrow \mathbb{R}$ whose domain is identified with the $%
x_{k_{1}}$-axis and whose range is identified with the $x_{k_{2}}$-axis.
Then we have%
\begin{equation*}
\left\vert f\left( x_{k_{1}}\right) \right\vert =\left\vert \varphi
_{k_{2}}\left( t_{0}\right) \right\vert >1
\end{equation*}%
for $x_{k_{1}}=\varphi _{k_{1}}\left( t_{0}\right) $. The map $g\left(
t\right) =f^{-1}\left( \varphi _{k_{2}}\left( t\right) \right) $ is
differentiable and one-to-one, hence monotone and $g\left( \alpha \right)
,g\left( \beta \right) \in \left[ -1,1\right] $. We may suppose $g$ is
strictly increasing. We also have $f\left( g\left( \alpha \right) \right)
=\varphi _{k_{2}}\left( \alpha \right) ,f\left( g\left( \beta \right)
\right) =\varphi _{k_{2}}\left( \beta \right) \in \left[ -1,1\right] $ since 
$\varphi \left( \alpha \right) ,\varphi \left( \beta \right) \in \overline{K}
$. It follows that $x_{k_{1}}\in \left( g\left( \alpha \right) ,g\left(
\beta \right) \right) $, and hence $f$ must have a relative extreme value at
some point $z\in \left( g\left( \alpha \right) ,g\left( \beta \right)
\right) $. But then $f^{\prime }\left( z\right) =0$ which implies $D\mathsf{P%
}\varphi \left( g^{-1}\left( z\right) \right) $ is parallel to $\mathbf{e}%
_{k_{1}}$, and so contradicts $\left\vert \left\langle \frac{D\mathsf{P}%
\varphi \left( g^{-1}\left( z\right) \right) }{\left\vert D\mathsf{P}\varphi
\left( g^{-1}\left( z\right) \right) \right\vert },\mathbf{e}%
_{k_{1}}\right\rangle \right\vert <1$ from (\ref{trans edges'}).
\end{proof}

To prove the backward triple quasitesting condition in (\ref{triple testing}%
),%
\begin{equation*}
\int_{3Q^{\prime }}\left\vert \mathbf{R}_{\Psi }^{\alpha ,n}\left(
1_{Q^{\prime }}\omega \right) \right\vert ^{2}d\sigma \leq \mathfrak{T}_{%
\mathbf{R}_{\Psi }^{\alpha ,n}}^{\Omega \mathcal{Q}^{n},\limfunc{triple},%
\func{dual}}\left\vert Q^{\prime }\right\vert _{\omega }\ ,
\end{equation*}%
it suffices to decompose the triple $3Q^{\prime }$ into $3^{n}$ dyadic
quasicubes $Q$ of side length that of $Q^{\prime }$, and then apply backward
testing to $Q=Q^{\prime }$ which gives $\mathfrak{T}_{\mathbf{R}_{\Psi
}^{\alpha ,n}}^{\Omega \mathcal{Q}^{n},\func{dual}}\left\vert Q^{\prime
}\right\vert _{\omega }$, and then to prove%
\begin{equation*}
\int_{Q}\left\vert \mathbf{R}_{\Psi }^{\alpha ,n}\left( 1_{Q^{\prime
}}\omega \right) \right\vert ^{2}d\sigma \leq \left( \sqrt{\mathcal{A}%
_{2}^{\alpha }}+\sqrt{\mathcal{A}_{2}^{\alpha ,\func{dual}}}\right)
\left\vert Q^{\prime }\right\vert _{\omega }\ ,
\end{equation*}%
where $Q$ and $Q^{\prime }$ are distinct quasicubes of equal side length in
a common dyadic quasigrid that share an $\left( n-1\right) $-dimensional
quasiface $\mathbb{F}$ in common. We also assume that $\omega $ is supported
on a line $L$ that is parallel to a coordinate axis. The cases when the
quasicubes $Q$ and $Q^{\prime }$ meet in an `edge' of smaller dimension, is
handled in similar fashion.

From Lemma \ref{connected} we see that the line $L$ meets the
quasihyperplane $\mathbb{H}$ containing the quasiface $\mathbb{F}$ at an
angle at least $\varepsilon >0$, and that the intersection $Q^{\prime }\cap
L $ is an interval. Now select the smallest possible dyadic subquasicube $%
Q^{\prime \prime }$ of $Q^{\prime }$ such that $Q^{\prime }\cap L=Q^{\prime
\prime }\cap L$ to obtain that the length of the interval $Q^{\prime \prime
}\cap L$ is comparable to $\ell \left( Q^{\prime \prime }\right) $. Then
since $Q\setminus 3Q^{\prime \prime }$ is well separated from $Q^{\prime
\prime }$ we have 
\begin{equation*}
\int_{Q\setminus 3Q^{\prime \prime }}\left\vert \mathbf{R}_{\Psi }^{\alpha
,n}\left( 1_{Q^{\prime \prime }}\omega \right) \right\vert ^{2}d\sigma \leq
\left( \sqrt{\mathcal{A}_{2}^{\alpha }}+\sqrt{\mathcal{A}_{2}^{\alpha ,\func{%
dual}}}\right) \left\vert Q^{\prime \prime }\right\vert _{\omega }\ ,
\end{equation*}%
and it remains to consider the integrals $\int_{Q^{\prime \prime \prime
}}\left\vert \mathbf{R}_{\Psi }^{\alpha ,n}\left( 1_{Q^{\prime \prime
}}\omega \right) \right\vert ^{2}d\sigma $ as $Q^{\prime \prime \prime }$
ranges over all dyadic quasicubes in $3Q^{\prime \prime }\cap Q$ with side
length $\ell \left( Q^{\prime \prime }\right) $. Now we relabel $Q^{\prime
\prime \prime }$ and $Q^{\prime \prime }$ as $Q$ and $Q^{\prime }$, and then
without loss of essential generality, we may assume that with $R$ denoting
an appropriate rotation, we have

\begin{enumerate}
\item $Q=\Psi K$ where $K=R\left( \left[ -1,0\right] \times \left[ -\frac{1}{%
2},\frac{1}{2}\right] ^{n-1}\right) $,

\item $Q^{\prime }=\Psi K^{\prime }$ where $K^{\prime }=R\left( \left[ 0,1%
\right] \times \left[ -\frac{1}{2},\frac{1}{2}\right] ^{n-1}\right) $,

\item $\limfunc{supp}\omega \subset L=\left( -\infty ,\infty \right) \times
\left\{ \left( 0,...,0\right) \right\} $ the $x_{1}$ -axis

\item $Q^{\prime }\cap L$ is an interval of length comparable to $\ell
\left( Q^{\prime }\right) $.
\end{enumerate}

Then the restriction $\omega _{Q^{\prime }}$ of $\omega $ to the quasicube $%
Q^{\prime }$ has support $\limfunc{supp}\omega _{Q^{\prime }}$ contained in
the line segment $S\equiv Q^{\prime }\cap L$, while the restriction $\sigma
_{Q}$ of $\sigma $ to the quasicube $Q$ has support $\limfunc{supp}\sigma
_{Q}$ contained in the quasicube $Q$. We exploit the distinguished role
played by the unique point in $\partial Q\cap \partial Q^{\prime }\cap L$,
which we relabel as the origin, by writing $y=t\xi \in Q$ where $%
t=\left\vert y\right\vert $ and $\xi \in \mathbb{S}^{n-1}$, and by writing $%
x=\left( s,\mathbf{0}\right) \in Q^{\prime }\cap L$, so that for an
appropriate $a\approx 1$, we have%
\begin{eqnarray*}
&&\int_{Q}\left\vert \mathbf{R}_{\Psi }^{\alpha ,n}\left( 1_{Q^{\prime
}}\omega \right) \right\vert ^{2}d\sigma \lesssim \int_{Q}\left\{
\int_{Q^{\prime }}\left\vert y-x\right\vert ^{\alpha -n}d\omega \left(
x\right) \right\} ^{2}d\sigma \left( y\right) \\
&=&\int_{Q}\left\{ \int_{0}^{a}\left\vert t\xi -\left( s,\mathbf{0}\right)
\right\vert ^{\alpha -n}d\omega \left( s,\mathbf{0}\right) \right\}
^{2}d\sigma \left( t\xi \right) \approx \int_{Q}\left\{ \int_{0}^{a}\left(
t+s\right) ^{\alpha -n}d\omega \left( s,\mathbf{0}\right) \right\}
^{2}d\sigma \left( t\xi \right) \\
&\equiv &\int_{0}^{\infty }\left\{ \int_{0}^{\infty }\left( t+s\right)
^{\alpha -n}d\widetilde{\omega }\left( s\right) \right\} ^{2}d\widetilde{%
\sigma }\left( t\right) =\int_{0}^{\infty }\left\{ \left(
\int_{0}^{t}+\int_{t}^{\infty }\right) \left( t+s\right) ^{\alpha -n}d%
\widetilde{\omega }\left( s\right) \right\} ^{2}d\widetilde{\sigma }\left(
t\right) \\
&\approx &\int_{0}^{\infty }\left\{ \int_{0}^{t}d\widetilde{\omega }\left(
s\right) \right\} ^{2}t^{2\alpha -2n}d\widetilde{\sigma }\left( t\right)
+\int_{0}^{\infty }\left\{ \int_{t}^{\infty }s^{\alpha -n}d\widetilde{\omega 
}\left( s\right) \right\} ^{2}d\widetilde{\sigma }\left( t\right) \equiv
I+II,
\end{eqnarray*}%
where the one dimensional measures $\widetilde{\omega }$ and $\widetilde{%
\sigma }$ are uniquely determined by $\omega ,Q^{\prime }$ and $\sigma ,Q$
respectively by the passage from the second line to the third line above.
Note also that the approximation in the second line above follows from (\ref%
{trans edges}). Now as in \cite{LaSaShUr2}, we apply Muckenhoupt's two
weight Hardy inequality for general measures (see Hyt\"{o}nen \cite{Hyt2}
for a proof), to obtain%
\begin{equation*}
\int_{0}^{\infty }\left\{ \int_{\left( 0,t\right] }f\left( s\right) d\mu
\left( s\right) \right\} ^{2}d\nu \left( t\right) \lesssim \left\{
\sup_{0<r<\infty }\left( \int_{\left[ r,\infty \right) }d\nu \right) \left(
\int_{\left( 0,r\right] }d\mu \right) \right\} \ \int_{0}^{\infty }f\left(
s\right) ^{2}d\mu \left( s\right) ,
\end{equation*}%
with $\mu =\widetilde{\omega }$, $d\nu \left( t\right) =t^{2\alpha -2n}d%
\widetilde{\sigma }\left( t\right) $ and $f=1$ to obtain that%
\begin{eqnarray*}
I &=&\int_{0}^{\infty }\left\{ \int_{\left( 0,t\right] }d\widetilde{\omega }%
\left( s\right) \right\} ^{2}t^{2\alpha -2n}d\widetilde{\sigma }\left(
t\right) \\
&\lesssim &\left\{ \sup_{0<r<\infty }\left( \int_{\left[ r,\infty \right)
}t^{2\alpha -2n}d\widetilde{\sigma }\left( t\right) \right) \left(
\int_{\left( 0,r\right] }d\widetilde{\omega }\left( s\right) \right)
\right\} \ \int_{0}^{\infty }d\widetilde{\omega }\left( s\right) \lesssim 
\mathcal{A}_{2}^{\alpha }\ \left\vert Q^{\prime }\right\vert _{\omega }
\end{eqnarray*}%
since $\int_{0}^{\infty }d\widetilde{\omega }\left( s\right) =\left\vert
Q^{\prime }\right\vert _{\omega }$ and 
\begin{eqnarray*}
&&\left( \int_{\left[ r,\infty \right) }t^{2\alpha -2n}d\widetilde{\sigma }%
\left( t\right) \right) \left( \int_{\left( 0,r\right] }d\widetilde{\omega }%
\left( s\right) \right) \\
&\approx &\int_{Q\cap \left\{ \left\vert y\right\vert \geq r\right\}
}\left\vert y-x\right\vert ^{2\alpha -2n}d\sigma \left( y\right)
\int_{Q^{\prime }\cap \left\{ \left\vert x\right\vert \leq r\right\}
}d\omega \left( x\right) \\
&\approx &\int_{Q\cap \left\{ \left\vert y\right\vert \geq r\right\} }\left( 
\frac{r}{\left( \left\vert y\right\vert +r\right) ^{2}}\right) ^{n-\alpha
}d\sigma \left( y\right) \ r^{\alpha -n}\left\vert Q^{\prime }\cap \left\{
\left\vert x\right\vert \leq r\right\} \right\vert _{\omega }\lesssim 
\mathcal{A}_{2}^{\alpha }\ .
\end{eqnarray*}%
Then we apply the two weight dual Hardy inequality%
\begin{equation*}
\int_{0}^{\infty }\left\{ \int_{\left[ t,\infty \right) }f\left( s\right)
d\mu \left( s\right) \right\} ^{2}d\nu \left( t\right) \leq \left\{
\sup_{0<r<\infty }\left( \int_{\left( 0,r\right] }d\nu \right) \left( \int_{%
\left[ r,\infty \right) }d\mu \right) \right\} \ \int_{0}^{\infty }f\left(
s\right) ^{2}d\mu \left( s\right) ,
\end{equation*}%
with $d\mu \left( s\right) =s^{2\alpha -2n}d\widetilde{\omega }\left(
s\right) $, $d\nu \left( t\right) =d\widetilde{\sigma }\left( t\right) $ and 
$f\left( s\right) =s^{n-\alpha }$ to obtain that%
\begin{eqnarray*}
II &=&\int_{0}^{\infty }\left\{ \int_{\left[ t,\infty \right) }s^{\alpha -n}d%
\widetilde{\omega }\left( s\right) \right\} ^{2}d\widetilde{\sigma }\left(
t\right) =\int_{0}^{\infty }\left\{ \int_{\left[ s,\infty \right)
}s^{n-\alpha }d\mu \left( s\right) \right\} ^{2}d\widetilde{\sigma }\left(
t\right) \\
&\lesssim &\left\{ \sup_{0<r<\infty }\left( \int_{\left( 0,r\right] }d%
\widetilde{\sigma }\left( t\right) \right) \left( \int_{\left[ r,\infty
\right) }s^{2\alpha -2n}d\widetilde{\omega }\left( s\right) \right) \right\}
\ \int_{0}^{\infty }s^{2n-2\alpha }d\mu \left( s\right) \lesssim \mathcal{A}%
_{2}^{\alpha }\ \left\vert Q^{\prime }\right\vert _{\omega }
\end{eqnarray*}%
since $\int_{0}^{\infty }s^{2n-2\alpha }d\mu \left( s\right)
=\int_{0}^{\infty }d\widetilde{\omega }\left( s\right) =\left\vert Q^{\prime
}\right\vert _{\omega }$ and $\left( \int_{\left( 0,r\right] }d\widetilde{%
\sigma }\left( t\right) \right) \left( \int_{\left[ r,\infty \right)
}s^{2\alpha -2n}d\widetilde{\omega }\left( s\right) \right) \lesssim 
\mathcal{A}_{2}^{\alpha ,\func{dual}}$ just as above.

\begin{remark}
In the case where one measure is supported on the $x$-axis, we need only
test over cubes with sides parallel to the coordinate axes. For the Cauchy
operator this is in \cite{LaSaShUrWi} and for the higher dimensional case
see the earlier versions of the current paper on the arXiv.
\end{remark}

\section{One measure compactly supported on a $C^{1,\protect\delta }$ curve}

Suppose that $\Psi :\mathbb{R}^{n}\rightarrow \mathbb{R}^{n}$ is a $%
C^{1,\delta }$ diffeomorphism. Recall the associated class of conformal
Riesz vector transforms $\mathbf{R}_{\Psi }^{\alpha ,n}$ whose kernels $%
\mathbf{K}_{\Psi }^{\alpha ,n}$ are given by%
\begin{equation*}
\mathbf{K}_{\Psi }^{\alpha ,n}\left( x,y\right) =\frac{y-x}{\left\vert \Psi
\left( y\right) -\Psi \left( x\right) \right\vert ^{n+1-\alpha }}.
\end{equation*}%
First we investigate the effect of a change of variable on the statement of
the $T1$ theorem.

\subsection{Changes of variable}

Suppose that $\Psi :\mathbb{R}^{n}\rightarrow \mathbb{R}^{n}$ is a $%
C^{1,\delta }$ diffeomorphism, i.e. that both $\Psi $ and its inverse $\Psi
^{-1}$ are globally $C^{1,\delta }$ maps. In particular we have that%
\begin{eqnarray}
&&\left\vert \Psi \left( y\right) -\Psi \left( x\right) \right\vert \leq
\left\Vert \Psi \right\Vert _{C^{1}}\left\vert y-x\right\vert ,
\label{rigidity} \\
&&\left\vert \Psi ^{-1}\left( w\right) -\Psi ^{-1}\left( z\right)
\right\vert \leq \left\Vert \Psi ^{-1}\right\Vert _{C^{1}}\left\vert
w-z\right\vert ,  \notag \\
&\Longrightarrow &\frac{1}{\left\Vert \Psi ^{-1}\right\Vert _{C^{1}}}\leq 
\frac{\left\vert \Psi \left( y\right) -\Psi \left( x\right) \right\vert }{%
\left\vert y-x\right\vert }=\frac{\left\vert w-z\right\vert }{\left\vert
\Psi ^{-1}\left( w\right) -\Psi ^{-1}\left( z\right) \right\vert }\leq
\left\Vert \Psi \right\Vert _{C^{1}},  \notag \\
&\Longrightarrow &\frac{1}{\left\Vert \Psi ^{-1}\right\Vert _{C^{1}}}\leq
\left\Vert D\Psi \right\Vert _{\infty }\leq \left\Vert \Psi \right\Vert
_{C^{1}}\ .  \notag
\end{eqnarray}%
Let 
\begin{equation}
\Psi ^{\ast }\mathbf{K}^{\alpha ,n}\left( x,y\right) \equiv \mathbf{K}%
^{\alpha ,n}\left( \Psi \left( x\right) ,\Psi \left( y\right) \right) =\frac{%
\Psi \left( y\right) -\Psi \left( x\right) }{\left\vert \Psi \left( y\right)
-\Psi \left( x\right) \right\vert ^{n+1-\alpha }},
\label{def kernel pullback}
\end{equation}%
be the pullback of the kernel $\mathbf{K}^{\alpha ,n}$ under $\Psi $, and
define the corresponding operator%
\begin{equation*}
\left( \Psi ^{\ast }\mathbf{R}^{\alpha ,n}\right) \left( f\mu \right) \left(
x\right) =\int \Psi ^{\ast }\mathbf{K}^{\alpha ,n}\left( x,y\right) f\left(
y\right) d\mu \left( y\right) .
\end{equation*}%
We claim the equality 
\begin{equation}
\mathfrak{N}_{\mathbf{R}^{\alpha ,n}}\left( \sigma ,\omega \right) =%
\mathfrak{N}_{\Psi ^{\ast }\mathbf{R}^{\alpha ,n}}\left( \Psi ^{\ast }\sigma
,\Psi ^{\ast }\omega \right) ,  \label{norm equality}
\end{equation}%
where $\Psi ^{\ast }\sigma =\left( \Psi ^{-1}\right) _{\ast }\sigma $
denotes the pushforward of $\sigma $ under $\Psi ^{-1}$, but as $\Psi $ is a
homeomorphism we abuse notation by writing $\Psi ^{\ast }$ for $\left( \Psi
^{-1}\right) _{\ast }$, and where $\mathfrak{N}_{\mathbf{R}^{\alpha
,n}}\left( \sigma ,\omega \right) $ is the best constant in the inequality%
\begin{equation}
\int \left\vert \mathbf{R}^{\alpha ,n}f\sigma \right\vert ^{2}d\omega \leq 
\mathfrak{N}_{\mathbf{R}^{\alpha ,n}}\left( \sigma ,\omega \right) \int
\left\vert f\right\vert ^{2}d\sigma ,  \label{norm omega}
\end{equation}%
and similarly for $\mathfrak{N}_{\Psi ^{\ast }\mathbf{R}^{\alpha ,n}}\left(
\Psi ^{\ast }\sigma ,\Psi ^{\ast }\omega \right) $. Indeed, with the change
of variable $x^{\prime }=\Psi \left( x\right) $, $y^{\prime }=\Psi \left(
y\right) $, and setting $\Psi ^{\ast }f=f\circ \Psi $, etc., we have%
\begin{eqnarray*}
\int \left\vert \mathbf{R}^{\alpha ,n}f\sigma \left( x^{\prime }\right)
\right\vert ^{2}d\omega \left( x^{\prime }\right) &=&\int \left\vert \mathbf{%
R}^{\alpha ,n}f\sigma \left( \Psi \left( x\right) \right) \right\vert
^{2}d\Psi ^{\ast }\omega \left( x\right) ; \\
\mathbf{R}^{\alpha ,n}f\sigma \left( \Psi \left( x\right) \right) &=&\int 
\mathbf{K}^{\alpha ,n}\left( \Psi \left( x\right) ,y^{\prime }\right)
f\left( y^{\prime }\right) d\sigma \left( y^{\prime }\right) \\
&=&\int \mathbf{K}^{\alpha ,n}\left( \Psi \left( x\right) ,\Psi \left(
y\right) \right) f\left( \Psi \left( y\right) \right) d\Psi ^{\ast }\sigma
\left( y\right) \\
&=&\int \Psi ^{\ast }\mathbf{K}^{\alpha ,n}\left( x,y\right) \Psi ^{\ast
}f\left( y\right) d\Psi ^{\ast }\sigma \left( y\right) \\
&=&\left( \Psi ^{\ast }\mathbf{R}^{\alpha ,n}\right) \left( \Psi ^{\ast
}f\Psi ^{\ast }\sigma \right) \left( x\right) ; \\
\int \left\vert f\left( y^{\prime }\right) \right\vert ^{2}d\sigma \left(
y^{\prime }\right) &=&\int \left\vert \Psi ^{\ast }f\left( y\right)
\right\vert ^{2}d\Psi ^{\ast }\sigma \left( y\right) ,
\end{eqnarray*}%
which shows that (\ref{norm omega}) becomes%
\begin{equation*}
\int \left\vert \left( \Psi ^{\ast }\mathbf{R}^{\alpha ,n}\right) \left(
\Psi ^{\ast }f\Psi ^{\ast }\sigma \right) \left( x\right) \right\vert
^{2}d\Psi ^{\ast }\omega \left( x\right) \leq \mathfrak{N}_{\mathbf{R}%
^{\alpha ,n}}\int \left\vert \Psi ^{\ast }f\left( y\right) \right\vert
^{2}d\Psi ^{\ast }\sigma \left( y\right) ,
\end{equation*}%
and hence that (\ref{norm equality}) holds.

Now the operator $\Psi ^{\ast }\mathbf{R}^{\alpha ,n}$ is easily seen to be
a standard fractional singular integral, but it fails to be a conformal
Riesz transform in general because the phase $\Psi \left( y\right) -\Psi
\left( x\right) $ in the kernel in (\ref{def kernel pullback}) is not $y-x$.
We will rectify this drawback by showing that the boundedness of the
conformal Riesz transform $\mathbf{R}_{\Psi }^{\alpha ,n}$ is equivalent to
that of $\Psi ^{\ast }\mathbf{R}^{\alpha ,n}$, and that the appropriate
testing conditions are equivalent as well. So consider the two inequalities (%
\ref{norm omega}) and%
\begin{equation}
\int \left\vert \mathbf{R}_{\Psi }^{\alpha ,n}\left( f\Psi ^{\ast }\sigma
\right) \right\vert ^{2}d\Psi ^{\ast }\omega \leq \mathfrak{N}_{\mathbf{R}%
_{\Psi }^{\alpha ,n}}\left( \left( \Psi ^{\ast }\sigma ,\Psi ^{\ast }\omega
\right) \right) \int \left\vert f\right\vert ^{2}d\Psi ^{\ast }\sigma ,
\label{norm omega'}
\end{equation}%
where we recall that the measures $\Psi ^{\ast }\omega $ and $\Psi ^{\ast
}\sigma $ are the pushforwards under $\Psi ^{-1}$ of the measures $\omega $
and $\sigma $ respectively. Here the constants $\mathfrak{N}_{\mathbf{R}%
^{\alpha ,n}}\left( \sigma ,\omega \right) $ and $\mathfrak{N}_{\mathbf{R}%
_{\Psi }^{\alpha ,n}}\left( \Psi ^{\ast }\sigma ,\Psi ^{\ast }\omega \right) 
$ are the smallest constants in their respective inequalities.

At this point we fix a collection of quasicubes $\Omega \mathcal{Q}^{n}$
with $\Omega $ biLipschitz, and recall the Muckenhoupt and energy constants 
\begin{eqnarray*}
&&\mathcal{A}_{2}^{\alpha }\left( \sigma ,\omega \right) ,\ \mathcal{A}%
_{2}^{\alpha ,\func{dual}}\left( \sigma ,\omega \right) ,\ A_{2}^{\alpha ,%
\limfunc{punct}}\left( \sigma ,\omega \right) ,\ A_{2}^{\alpha ,\limfunc{%
punct},\func{dual}}\left( \sigma ,\omega \right) ,\ \mathcal{E}_{\alpha
}^{\Omega \mathcal{Q}^{n}}\left( \sigma ,\omega \right) ,\ \mathcal{E}%
_{\alpha }^{\Omega \mathcal{Q}^{n},\func{dual}}\left( \sigma ,\omega \right)
; \\
&&\mathcal{A}_{2}^{\alpha }\left( \Psi ^{\ast }\sigma ,\Psi ^{\ast }\omega
\right) ,\ \mathcal{A}_{2}^{\alpha ,\func{dual}}\left( \Psi ^{\ast }\sigma
,\Psi ^{\ast }\omega \right) ,\ A_{2}^{\alpha ,\limfunc{punct}}\left( \Psi
^{\ast }\sigma ,\Psi ^{\ast }\omega \right) ,\ A_{2}^{\alpha ,\limfunc{punct}%
,\func{dual}}\left( \Psi ^{\ast }\sigma ,\Psi ^{\ast }\omega \right) , \\
&&\ \ \ \ \ \ \ \ \ \ \ \ \ \ \ \ \ \ \ \ \ \ \ \ \ \ \ \ \ \ \ \ \ \ \ 
\text{and }\mathcal{E}_{\alpha }^{\Omega \mathcal{Q}^{n}}\left( \Psi ^{\ast
}\sigma ,\Psi ^{\ast }\omega \right) ,\ \mathcal{E}_{\alpha }^{\Omega 
\mathcal{Q}^{n},\func{dual}}\left( \Psi ^{\ast }\sigma ,\Psi ^{\ast }\omega
\right) ,
\end{eqnarray*}%
that depend only on the measures and the quasicubes, and the testing
constants 
\begin{eqnarray*}
&&\mathfrak{T}_{\mathbf{R}^{\alpha ,n}}^{\Omega \mathcal{Q}^{n}}\left(
\sigma ,\omega \right) ,\ \mathfrak{T}_{\mathbf{R}^{\alpha ,n}}^{\Omega 
\mathcal{Q}^{n},\func{dual}}\left( \sigma ,\omega \right) ; \\
&&\mathfrak{T}_{\mathbf{R}_{\Psi }^{\alpha ,n}}^{\Omega \mathcal{Q}%
^{n}}\left( \Psi ^{\ast }\sigma ,\Psi ^{\ast }\omega \right) ,\ \mathfrak{T}%
_{\mathbf{R}_{\Psi }^{\alpha ,n}}^{\Omega \mathcal{Q}^{n},\func{dual}}\left(
\Psi ^{\ast }\sigma ,\Psi ^{\ast }\omega \right) ,
\end{eqnarray*}%
that depend on the measures and the quasicubes as well as the fractional
singular integral and its tangent line truncations. We sometimes suppress
the dependence $\left( \sigma ,\omega \right) $ on the measures when they
are understood from the context, or when they do not play a significant role.

Finally, we define an even more general testing condition. Let $\mathcal{F}$
be a collection of bounded Borel sets, and let $\mathbf{T}^{\alpha ,n}$ be
an $\alpha $-fractional singular integral. Then define $\mathfrak{T}_{%
\mathbf{T}^{\alpha ,n}}^{\mathcal{F}}=\mathfrak{T}_{\mathbf{T}^{\alpha ,n}}^{%
\mathcal{F}}\left( \sigma ,\omega \right) $ to be the smallest constant in
the inequality%
\begin{equation*}
\int \left\vert \mathbf{T}^{\alpha ,n}\mathbf{1}_{F}\sigma \right\vert
^{2}d\omega \leq \mathfrak{T}_{\mathbf{T}^{\alpha ,n}}^{\mathcal{F}%
}\left\vert F\right\vert _{\sigma },\ \ \ \ \ F\in \mathcal{F},
\end{equation*}%
and similarly for the dual $\mathfrak{T}_{\mathbf{T}^{\alpha ,n,\func{dual}%
}}^{\mathcal{F},\func{dual}}=\mathfrak{T}_{\mathbf{T}^{\alpha ,n,\func{dual}%
}}^{\mathcal{F},\func{dual}}\left( \sigma ,\omega \right) $. Note that our
testing conditions above are with $\mathcal{F}=\Omega \mathcal{Q}^{n}$ and $%
\mathbf{T}^{\alpha ,n}=\mathbf{R}^{\alpha ,n}$ or $\mathbf{R}_{\Psi
}^{\alpha ,n}$. Given $\Psi $ and $\mathcal{F}$ as above, denote by $\Psi
^{\ast }\mathcal{F}\equiv \left\{ \Psi ^{-1}\left( F\right) :F\in \mathcal{F}%
\right\} $ the pullback of $\mathcal{F}$ under the map $\Psi $, i.e. $\Psi
^{\ast }\mathbf{1}_{F}=\mathbf{1}_{F}\circ \Psi =\mathbf{1}_{\Psi
^{-1}\left( F\right) }$. Of particular interest for us is the set of
quasicubes $\mathcal{Q}=\Omega \mathcal{Q}^{n}$ which is used in the
versions given above of the testing conditions. Then we have $\Psi ^{\ast }%
\mathcal{Q}=\left\{ \Psi ^{-1}\left( Q\right) :Q\in \mathcal{Q}\right\} $,
and the sets $\Psi ^{-1}\left( Q\right) $ form a new family of quasicubes
since $\Psi ^{-1}\circ \Omega $ is a globally biLipschitz map, which if
necessary we will refer to as $\Psi ^{-1}\circ \Omega $-quasicubes.

Our first proposition concerns the equivalence of the Muckenhoupt conditions
under a biLipschitz change of variable, and the next proposition considers
norm inequalities and testing conditions.

\begin{proposition}
\label{Muck equiv}Suppose $\Omega :\mathbb{R}^{n}\rightarrow \mathbb{R}^{n}$
is a globally biLipschitz map and that the Muckenhoupt conditions are
defined by taking supremums over the collection $\Omega \mathcal{Q}^{n}$ of $%
\Omega $-quasicubes. Let $\Psi $ be another globally biLipschitz map, and
let $\sigma $ and $\omega $ be positive Borel measures possibly having
common point masses. Then we have the following three equivalences:%
\begin{eqnarray*}
A_{2}^{\alpha ,\limfunc{punct}}\left( \sigma ,\omega \right) &\approx
&A_{2}^{\alpha ,\limfunc{punct}}\left( \Psi ^{\ast }\sigma ,\Psi ^{\ast
}\omega \right) , \\
A_{2}^{\alpha ,\limfunc{punct},\func{dual}}\left( \sigma ,\omega \right)
&\approx &A_{2}^{\alpha ,\limfunc{punct},\func{dual}}\left( \Psi ^{\ast
}\sigma ,\Psi ^{\ast }\omega \right) , \\
\mathfrak{A}_{2}^{\alpha }\left( \sigma ,\omega \right) &=&\mathfrak{A}%
_{2}^{\alpha }\left( \Psi ^{\ast }\sigma ,\Psi ^{\ast }\omega \right) .
\end{eqnarray*}
\end{proposition}

Note the absence of any statement regarding the one-tailed Muckenhoupt
conditions with holes, $\mathcal{A}_{2}^{\alpha }$ and $\mathcal{A}%
_{2}^{\alpha ,\func{dual}}$, where it is not obvious that an equivalence is
possible.

\begin{proof}
For convenience we set $\widetilde{\sigma }=\Psi ^{\ast }\sigma $, $%
\widetilde{\omega }=\Psi ^{\ast }\omega $ and $\widetilde{Q}=\Psi \left(
Q\right) $. In order to show that%
\begin{equation*}
A_{2}^{\alpha ,\limfunc{punct}}\left( \Psi ^{\ast }\sigma ,\Psi ^{\ast
}\omega \right) \lesssim A_{2}^{\alpha ,\limfunc{punct}}\left( \sigma
,\omega \right) ,
\end{equation*}%
it suffices to show that%
\begin{equation*}
\frac{\widetilde{\omega }\left( K,\mathfrak{P}_{\left( \widetilde{\sigma },%
\widetilde{\omega }\right) }\right) }{\left\vert K\right\vert ^{1-\frac{%
\alpha }{n}}}\frac{\left\vert K\right\vert _{\widetilde{\sigma }}}{%
\left\vert K\right\vert ^{1-\frac{\alpha }{n}}}\lesssim \sup_{Q\in \Omega 
\mathcal{Q}^{n}}\frac{\omega \left( Q,\mathfrak{P}_{\left( \sigma ,\omega
\right) }\right) }{\left\vert Q\right\vert ^{1-\frac{\alpha }{n}}}\frac{%
\left\vert Q\right\vert _{\sigma }}{\left\vert Q\right\vert ^{1-\frac{\alpha 
}{n}}}
\end{equation*}%
for all $\Omega $-quasicubes $K\in \Omega \mathcal{Q}^{n}$. Now a change of
variable shows that%
\begin{equation*}
\frac{\widetilde{\omega }\left( K,\mathfrak{P}_{\left( \widetilde{\sigma },%
\widetilde{\omega }\right) }\right) }{\left\vert K\right\vert ^{1-\frac{%
\alpha }{n}}}\frac{\left\vert K\right\vert _{\widetilde{\sigma }}}{%
\left\vert K\right\vert ^{1-\frac{\alpha }{n}}}=\frac{\omega \left( 
\widetilde{K},\mathfrak{P}_{\left( \sigma ,\omega \right) }\right) }{%
\left\vert K\right\vert ^{1-\frac{\alpha }{n}}}\frac{\left\vert \widetilde{K}%
\right\vert _{\sigma }}{\left\vert K\right\vert ^{1-\frac{\alpha }{n}}},
\end{equation*}%
where of course $\left\vert K\right\vert \approx \left\vert \widetilde{K}%
\right\vert $. Now choose a quasicube $Q$ containing $\widetilde{K}$ with $%
\ell \left( Q\right) \leq C_{\Omega }\ell \left( K\right) $ so that we have $%
\left\vert K\right\vert \approx \left\vert Q\right\vert $. If a largest
common point mass for $\omega $ in $\widetilde{K}$ (respectively $Q$) occurs
at $x$ (respectively $y$), then $\omega \left( \left\{ x\right\} \right)
\leq \omega \left( \left\{ y\right\} \right) $ and so we have%
\begin{equation*}
\omega \left( \widetilde{K},\mathfrak{P}_{\left( \sigma ,\omega \right)
}\right) =\left\vert \widetilde{K}\right\vert _{\omega }-\omega \left(
\left\{ x\right\} \right) \delta _{x}\leq \left\vert Q\right\vert _{\omega
}-\omega \left( \left\{ y\right\} \right) \delta _{y}=\omega \left( Q,%
\mathfrak{P}_{\left( \sigma ,\omega \right) }\right) ,
\end{equation*}%
since if $x=y$ we use$\left\vert \widetilde{K}\right\vert _{\omega }\leq
\left\vert Q\right\vert _{\omega }$, while if $x\neq y$, then $y\notin 
\widetilde{K}$ and we use $\left\vert Q\right\vert _{\omega }-\omega \left(
\left\{ y\right\} \right) \delta _{y}\geq \left\vert \widetilde{K}%
\right\vert _{\omega }$. Thus%
\begin{eqnarray*}
\frac{\widetilde{\omega }\left( K,\mathfrak{P}_{\left( \widetilde{\sigma },%
\widetilde{\omega }\right) }\right) }{\left\vert K\right\vert ^{1-\frac{%
\alpha }{n}}}\frac{\left\vert K\right\vert _{\widetilde{\sigma }}}{%
\left\vert K\right\vert ^{1-\frac{\alpha }{n}}} &\leq &\frac{\omega \left( Q,%
\mathfrak{P}_{\left( \sigma ,\omega \right) }\right) }{\left\vert
K\right\vert ^{1-\frac{\alpha }{n}}}\frac{\left\vert \widetilde{K}%
\right\vert _{\sigma }}{\left\vert K\right\vert ^{1-\frac{\alpha }{n}}} \\
&\lesssim &\frac{\omega \left( Q,\mathfrak{P}_{\left( \sigma ,\omega \right)
}\right) }{\left\vert Q\right\vert ^{1-\frac{\alpha }{n}}}\frac{\left\vert
Q\right\vert _{\sigma }}{\left\vert Q\right\vert ^{1-\frac{\alpha }{n}}}\leq
A_{2}^{\alpha ,\limfunc{punct}}\left( \sigma ,\omega \right) ,
\end{eqnarray*}%
which completes the proof of the first assertion in Proposition \ref{Muck
equiv}. The second assertion is proved in similar fashion.

Now we turn to the third assertion in Proposition \ref{Muck equiv}, where in
view of what we have just shown, it suffices to show that%
\begin{equation*}
\mathcal{A}_{2}^{\alpha }\left( \widetilde{\sigma },\widetilde{\omega }%
\right) +\mathcal{A}_{2}^{\alpha ,\func{dual}}\left( \widetilde{\sigma },%
\widetilde{\omega }\right) \lesssim \mathfrak{A}_{2}^{\alpha }\left( \sigma
,\omega \right) .
\end{equation*}%
By symmetry it is then enough to show%
\begin{equation*}
\mathcal{P}\left( K,\mathbf{1}_{K^{c}}\widetilde{\sigma }\right) \frac{%
\left\vert K\right\vert _{\widetilde{\omega }}}{\left\vert K\right\vert ^{1-%
\frac{\alpha }{n}}}\lesssim \mathfrak{A}_{2}^{\alpha }\left( \sigma ,\omega
\right) ,
\end{equation*}%
for all $\Omega $-quasicubes $K\in \Omega \mathcal{Q}^{n}$. Now a change of
variable shows that%
\begin{eqnarray*}
\frac{\left\vert K\right\vert _{\widetilde{\omega }}}{\left\vert
K\right\vert ^{1-\frac{\alpha }{n}}}\mathcal{P}\left( K,\mathbf{1}_{K^{c}}%
\widetilde{\sigma }\right) &=&\frac{\left\vert K\right\vert _{\widetilde{%
\omega }}}{\left\vert K\right\vert ^{1-\frac{\alpha }{n}}}\int_{\mathbb{R}%
^{n}\setminus K}\left( \frac{\left\vert K\right\vert ^{\frac{1}{n}}}{\left(
\left\vert K\right\vert ^{\frac{1}{n}}+\left\vert x-c_{K}\right\vert \right)
^{2}}\right) ^{n-\alpha }d\widetilde{\sigma }\left( x\right) \\
&\approx &\frac{\left\vert \widetilde{K}\right\vert _{\omega }}{\left\vert
K\right\vert ^{1-\frac{\alpha }{n}}}\int_{\mathbb{R}^{n}\setminus \widetilde{%
K}}\left( \frac{\left\vert K\right\vert ^{\frac{1}{n}}}{\left( \left\vert
K\right\vert ^{\frac{1}{n}}+\left\vert x^{\prime }-c_{\widetilde{K}%
}\right\vert \right) ^{2}}\right) ^{n-\alpha }d\sigma \left( x^{\prime
}\right) \\
&\approx &\frac{\left\vert \widetilde{K}\right\vert _{\omega }}{\left\vert
K\right\vert ^{1-\frac{\alpha }{n}}}\mathcal{P}\left( \widetilde{K},\mathbf{1%
}_{\widetilde{K}^{c}}\sigma \right) \approx \frac{\left\vert \widetilde{K}%
\right\vert _{\omega }}{\left\vert K\right\vert ^{1-\frac{\alpha }{n}}}%
\mathcal{P}\left( K,\mathbf{1}_{\widetilde{K}^{c}}\sigma \right) ,
\end{eqnarray*}%
where of course $\left\vert K\right\vert \approx \left\vert \widetilde{K}%
\right\vert $ and $\mathcal{P}\left( K,\mu \right) \approx \mathcal{P}\left( 
\widetilde{K},\mu \right) $. We have written $\widetilde{K}$ in place of $K$
in the final equivalence only when it matters. Now choose quasicubes $Q,P\in
\Omega \mathcal{Q}^{n}$ such that%
\begin{equation*}
Q\subset \widetilde{K}\subset P\text{ and }\ell \left( P\right) \leq
C_{\Omega }\ell \left( Q\right) ,
\end{equation*}%
so that $\left\vert Q\right\vert \approx \left\vert \widetilde{K}\right\vert
\approx \left\vert P\right\vert $, and $\mathcal{P}\left( Q,\mu \right)
\approx \mathcal{P}\left( \widetilde{K},\mu \right) \approx \mathcal{P}%
\left( P,\mu \right) $ for any postive measure $\mu $. Let $y\in P$ be a
point where the largest common point mass of $\sigma $ occurs, and let $z\in
P$ be a point where the largest common point mass of $\omega $ occurs. Define%
\begin{equation*}
\dot{\sigma}=\sigma -\sigma \left( \left\{ y\right\} \right) \text{ and }%
\dot{\omega}=\omega -\omega \left( \left\{ z\right\} \right) .
\end{equation*}

Now we have the two `punctured' inequalities,%
\begin{eqnarray}
\frac{\left\vert \widetilde{K}\right\vert _{\omega }}{\left\vert
K\right\vert ^{1-\frac{\alpha }{n}}}\mathcal{P}\left( K,\mathbf{1}_{%
\widetilde{K}^{c}}\dot{\sigma}\right) &\leq &\frac{\left\vert P\right\vert
_{\omega }}{\left\vert K\right\vert ^{1-\frac{\alpha }{n}}}\mathcal{P}\left(
K,\mathbf{1}_{P}\dot{\sigma}\right) +\frac{\left\vert P\right\vert _{\omega }%
}{\left\vert K\right\vert ^{1-\frac{\alpha }{n}}}\mathcal{P}\left( K,\mathbf{%
1}_{P^{c}}\sigma \right)  \label{ineq 1} \\
&\lesssim &A_{2}^{\alpha ,\limfunc{punct}}\left( \sigma ,\omega \right) +%
\mathcal{A}_{2}^{\alpha }\left( \sigma ,\omega \right) \leq \mathfrak{A}%
_{2}^{\alpha }\left( \sigma ,\omega \right) ,  \notag
\end{eqnarray}%
and%
\begin{eqnarray}
\frac{\left\vert \widetilde{K}\right\vert _{\dot{\omega}}}{\left\vert
K\right\vert ^{1-\frac{\alpha }{n}}}\mathcal{P}\left( K,\mathbf{1}_{%
\widetilde{K}^{c}}\sigma \right) &\leq &\frac{\left\vert P\right\vert _{\dot{%
\omega}}}{\left\vert K\right\vert ^{1-\frac{\alpha }{n}}}\mathcal{P}\left( K,%
\mathbf{1}_{P}\sigma \right) +\frac{\left\vert P\right\vert _{\omega }}{%
\left\vert K\right\vert ^{1-\frac{\alpha }{n}}}\mathcal{P}\left( K,\mathbf{1}%
_{P^{c}}\sigma \right)  \label{ineq 2} \\
&\lesssim &A_{2}^{\alpha ,\limfunc{punct},\func{dual}}\left( \sigma ,\omega
\right) +\mathcal{A}_{2}^{\alpha }\left( \sigma ,\omega \right) \leq 
\mathfrak{A}_{2}^{\alpha }\left( \sigma ,\omega \right) .  \notag
\end{eqnarray}

Next, we claim that if $y\neq z$ , then%
\begin{equation}
\frac{\sigma \left( \left\{ y\right\} \right) }{\left\vert P\right\vert ^{1-%
\frac{\alpha }{n}}}\frac{\omega \left( \left\{ z\right\} \right) }{%
\left\vert P\right\vert ^{1-\frac{\alpha }{n}}}\lesssim \mathcal{A}%
_{2}^{\alpha }\left( \sigma ,\omega \right) +\mathcal{A}_{2}^{\alpha ,\func{%
dual}}\left( \sigma ,\omega \right) \leq \mathfrak{A}_{2}^{\alpha }\left(
\sigma ,\omega \right) .  \label{ineq 3}
\end{equation}%
Indeed, it is easy to find a quasicube $R\subset P$ with half the side
length of $P$ such that exactly one of $y$ and $z$ lies in $R$. For the
purpose of clarifying the remainder of this argument, we assume these
quasicubes are all ordinary cubes, and the reader can then easily modify the
argument to quasicubes. If $y$ and $z$ lie on opposite sides of a horizontal
or vertical line $L$, then $P\setminus L$ consists of two disjoint
rectangles, each containing one of the points $y$ and $z$. Clearly, the
larger of the two rectangles contains a cube $R$ with side length at least $%
\frac{1}{2}\ell \left( P\right) $ and containing one of $y$ and $z$.

With such a quasicube $R$ in hand, say with $y\in R$ and $z\in R^{c}$, then%
\begin{equation*}
\frac{\sigma \left( \left\{ y\right\} \right) }{\left\vert P\right\vert ^{1-%
\frac{\alpha }{n}}}\frac{\omega \left( \left\{ z\right\} \right) }{%
\left\vert P\right\vert ^{1-\frac{\alpha }{n}}}\lesssim \frac{\left\vert
R\right\vert _{\sigma }}{\left\vert R\right\vert ^{1-\frac{\alpha }{n}}}%
\mathcal{P}\left( R,\mathbf{1}_{R^{c}}\omega \right) \leq \mathcal{A}%
_{2}^{\alpha ,\func{dual}}\left( \sigma ,\omega \right) \leq \mathfrak{A}%
_{2}^{\alpha }\left( \sigma ,\omega \right) .
\end{equation*}

Now we write%
\begin{eqnarray*}
\frac{\left\vert \widetilde{K}\right\vert _{\omega }}{\left\vert
K\right\vert ^{1-\frac{\alpha }{n}}}\mathcal{P}\left( K,\mathbf{1}_{%
\widetilde{K}^{c}}\sigma \right) &=&\frac{\left\vert \widetilde{K}%
\right\vert _{\dot{\omega}+\omega \left( \left\{ z\right\} \right) }}{%
\left\vert K\right\vert ^{1-\frac{\alpha }{n}}}\mathcal{P}\left( K,\mathbf{1}%
_{\widetilde{K}^{c}}\left[ \dot{\sigma}+\sigma \left( \left\{ y\right\}
\right) \right] \right) \\
&\leq &\frac{\left\vert \widetilde{K}\right\vert _{\dot{\omega}}}{\left\vert
K\right\vert ^{1-\frac{\alpha }{n}}}\mathcal{P}\left( K,\mathbf{1}_{%
\widetilde{K}^{c}}\sigma \right) +\frac{\left\vert \widetilde{K}\right\vert
_{\omega }}{\left\vert K\right\vert ^{1-\frac{\alpha }{n}}}\mathcal{P}\left(
K,\mathbf{1}_{\widetilde{K}^{c}}\dot{\sigma}\right) +\frac{\sigma \left(
\left\{ y\right\} \right) }{\left\vert P\right\vert ^{1-\frac{\alpha }{n}}}%
\frac{\omega \left( \left\{ z\right\} \right) }{\left\vert P\right\vert ^{1-%
\frac{\alpha }{n}}}.
\end{eqnarray*}%
If $y\neq z$ then we have 
\begin{equation*}
\frac{\left\vert \widetilde{K}\right\vert _{\omega }}{\left\vert
K\right\vert ^{1-\frac{\alpha }{n}}}\mathcal{P}\left( K,\mathbf{1}_{%
\widetilde{K}^{c}}\sigma \right) \lesssim \mathfrak{A}_{2}^{\alpha }\left(
\sigma ,\omega \right)
\end{equation*}%
by the three inequalities (\ref{ineq 1}), (\ref{ineq 2}) and (\ref{ineq 3})
proved above. On the other hand, if $y=z$, then either $y\in \widetilde{K}$
or $z\in P\setminus \widetilde{K}$. If $y\in \widetilde{K}$ we have%
\begin{equation*}
\frac{\left\vert \widetilde{K}\right\vert _{\omega }}{\left\vert
K\right\vert ^{1-\frac{\alpha }{n}}}\mathcal{P}\left( K,\mathbf{1}_{%
\widetilde{K}^{c}}\sigma \right) \leq \frac{\left\vert \widetilde{K}%
\right\vert _{\omega }}{\left\vert K\right\vert ^{1-\frac{\alpha }{n}}}%
\mathcal{P}\left( K,\mathbf{1}_{\widetilde{K}^{c}}\dot{\sigma}\right)
\lesssim \mathfrak{A}_{2}^{\alpha }\left( \sigma ,\omega \right)
\end{equation*}%
by (\ref{ineq 1}), and if $z\in P\setminus \widetilde{K}$ we have%
\begin{equation*}
\frac{\left\vert \widetilde{K}\right\vert _{\omega }}{\left\vert
K\right\vert ^{1-\frac{\alpha }{n}}}\mathcal{P}\left( K,\mathbf{1}_{%
\widetilde{K}^{c}}\sigma \right) \leq \frac{\left\vert \widetilde{K}%
\right\vert _{\dot{\omega}}}{\left\vert K\right\vert ^{1-\frac{\alpha }{n}}}%
\mathcal{P}\left( K,\mathbf{1}_{\widetilde{K}^{c}}\sigma \right) \lesssim 
\mathfrak{A}_{2}^{\alpha }\left( \sigma ,\omega \right)
\end{equation*}%
by (\ref{ineq 2}). This completes the proof of Proposition \ref{Muck equiv}.
\end{proof}

\begin{proposition}
\label{change of variable}Suppose $\Psi :\mathbb{R}^{n}\rightarrow \mathbb{R}%
^{n}$ is a $C^{1,\delta }$ diffeomorphism, i.e. both $\Psi $ and its inverse 
$\Psi ^{-1}$ are globally $C^{1,\delta }$ maps, let $\sigma $ and $\omega $
be positive Borel measures (possibly having common point masses) with one of
the measures supported in a \emph{compact} subset $K$ of $\mathbb{R}^{n}$,
and let $\mathcal{F}$ be a collection of bounded Borel sets. Then with the
fractional Riesz transform $\mathbf{R}^{\alpha ,n}$ and the conformal
fractional Riesz transform $\mathbf{R}_{\Psi }^{\alpha ,n}$ as above, we
have the following three equivalences:%
\begin{eqnarray*}
&&\mathbf{1.}\ \mathfrak{N}_{\mathbf{R}^{\alpha ,n}}\left( \sigma ,\omega
\right) +\sqrt{\mathfrak{A}_{2}^{\alpha }\left( \sigma ,\omega \right) }%
\approx \mathfrak{N}_{\mathbf{R}_{\Psi }^{\alpha ,n}}\left( \Psi ^{\ast
}\sigma ,\Psi ^{\ast }\omega \right) +\sqrt{\mathfrak{A}_{2}^{\alpha }\left(
\sigma ,\omega \right) }, \\
&&\mathbf{2.}\ \mathfrak{T}_{\mathbf{R}^{\alpha ,n}}^{\mathcal{F}}\left(
\sigma ,\omega \right) +\sqrt{\mathfrak{A}_{2}^{\alpha }\left( \sigma
,\omega \right) }\approx \mathfrak{T}_{\mathbf{R}_{\Psi }^{\alpha ,n}}^{\Psi
^{\ast }\mathcal{F}}\left( \Psi ^{\ast }\sigma ,\Psi ^{\ast }\omega \right) +%
\sqrt{\mathfrak{A}_{2}^{\alpha }\left( \sigma ,\omega \right) }, \\
&&\mathbf{3.}\ \mathfrak{T}_{\mathbf{R}^{\alpha ,n}}^{\mathcal{F},\func{dual}%
}\left( \sigma ,\omega \right) +\sqrt{\mathfrak{A}_{2}^{\alpha }\left(
\sigma ,\omega \right) }\approx \mathfrak{T}_{\mathbf{R}_{\Psi }^{\alpha
,n}}^{\Psi ^{\ast }\mathcal{F},\func{dual}}\left( \Psi ^{\ast }\sigma ,\Psi
^{\ast }\omega \right) +\sqrt{\mathfrak{A}_{2}^{\alpha }\left( \sigma
,\omega \right) },
\end{eqnarray*}%
where the implied constants depend only $n,\alpha ,\limfunc{diam}\left(
K\right) ,\left\Vert \Gamma \right\Vert _{1,\delta }$ and $\left\Vert \Psi
\right\Vert _{C^{1,\delta }}+\left\Vert \Psi ^{-1}\right\Vert _{C^{1,\delta
}}$.
\end{proposition}

In particular, we see that in the presence of the Muckenhoupt conditions $%
\mathfrak{A}_{2}^{\alpha }$, and when one of the measures is supported in a
compact set, the testing conditions for $\mathbf{T}_{\Gamma }^{\alpha }$ on
indicators of quasicubes $Q$ are equivalent to the testing conditions for $%
\mathbf{T}_{\Gamma _{\Psi }}^{\alpha ,n}$ on indicators of the new
quasicubes $\Psi ^{-1}\left( Q\right) $. Thus in order to deform the sets
over which we test from $Q$ to $\Psi ^{-1}\left( Q\right) $, we need only
push the measures forward and alter the conformal factor $\Gamma $ to the
associated conformal factor $\Gamma _{\Psi }$ in the operator, keeping the
critical phase $y-x$ in the numerator of the kernel unchanged. The same
results hold for the inverse $C^{1,\delta }$ diffeomorphism $\Psi ^{-1}$ in
place of $\Psi $.

Before beginning the proof it is convenient to introduce two auxilliary
operators $\Psi ^{\ast ,\tan ,1}\mathbf{R}^{\alpha ,n}$ and $\Psi ^{\ast
,\tan ,2}\mathbf{R}^{\alpha ,n}$ with kernels related to the pullback kernel 
$\Psi ^{\ast }\mathbf{K}^{\alpha ,n}$ defined above by%
\begin{equation*}
\Psi ^{\ast }\mathbf{K}^{\alpha ,n}\left( x,y\right) =\frac{\Psi \left(
y\right) -\Psi \left( x\right) }{\left\vert \Psi \left( y\right) -\Psi
\left( x\right) \right\vert ^{n+1-\alpha }}.
\end{equation*}%
We define the kernels of $\Psi ^{\ast ,\tan ,1}\mathbf{R}^{\alpha ,n}$ and $%
\Psi ^{\ast ,\tan ,2}\mathbf{R}^{\alpha ,n}$ by%
\begin{eqnarray*}
\Psi ^{\ast ,\tan ,1}\mathbf{K}^{\alpha ,n}\left( x,y\right) &\equiv &\Psi
^{\prime }\left( x\right) \ \frac{y-x}{\left\vert \Psi \left( y\right) -\Psi
\left( x\right) \right\vert ^{n+1-\alpha }}, \\
\Psi ^{\ast ,\tan ,2}\mathbf{K}^{\alpha ,n}\left( x,y\right) &\equiv &\Psi
^{\prime }\left( y\right) \ \frac{y-x}{\left\vert \Psi \left( y\right) -\Psi
\left( x\right) \right\vert ^{n+1-\alpha }}.
\end{eqnarray*}%
The superscript $\tan ,1$ (respectively $\tan ,2$) indicates that we are
replacing the phase function $\Psi \left( y\right) -\Psi \left( x\right) $
with its tangent line approximation at $x$ (respectively $y$). Now we prove
Proposition \ref{change of variable}.

\begin{proof}
We begin with the first statement, where we may assume that $\omega $ is
supported in a compact ball $B$. Moreover, we may assume the cubes below are
usual cubes since we invoke no testing or energy in the proof of this second
statement. Then we may also assume that $\sigma $ is supported in the double 
$2B$. Indeed, if $\omega $ is supported in a ball $B$ and $\sigma $ is
supported outside the double $2B$ of the ball $B$, then the associated norm
inequality for a fractional singular integral operator $\mathbf{T}^{\alpha
,n}$ is easily seen to be controlled solely in terms of the Muckenhoupt
constant $\mathcal{A}_{2}^{\alpha }$ with holes:%
\begin{eqnarray}
&&  \label{separated} \\
\int_{\mathbb{R}^{n}\setminus 2B}\left\vert \mathbf{T}^{\alpha ,n}\mathbf{1}%
_{B}g\omega \right\vert ^{2}d\sigma &\lesssim &\int_{\mathbb{R}^{n}\setminus
2B}\left\vert \int_{B}\left\vert x-y\right\vert ^{\alpha -n}g\left( x\right)
d\omega \left( x\right) \right\vert ^{2}d\sigma \left( y\right)  \notag \\
&\lesssim &\left( \int_{B}\left\vert g\right\vert ^{2}d\omega \right) \int_{%
\mathbb{R}^{n}\setminus 2B}\left( \int_{B}\left\vert x-y\right\vert
^{2\alpha -2n}d\omega \left( x\right) \right) d\sigma \left( y\right)  \notag
\\
&\lesssim &\left\Vert g\right\Vert _{L^{2}\left( \omega \right) }^{2}\frac{%
\left\vert B\right\vert _{\omega }}{\left\vert B\right\vert ^{1-\frac{\alpha 
}{n}}}\mathcal{P}^{\alpha }\left( B,\sigma \right) \lesssim \mathcal{A}%
_{2}^{\alpha }\left\Vert g\right\Vert _{L^{2}\left( \omega \right) }^{2}\ , 
\notag
\end{eqnarray}%
where we have used that $\left\vert x-y\right\vert ^{2\alpha -2n}\approx
\left\vert c_{B}-y\right\vert ^{2\alpha -2n}$ for $y\in \mathbb{R}%
^{n}\setminus 2B$ and $x\in B$.

We write the pullback $\Psi ^{\ast }\mathbf{K}^{\alpha ,n}\left( x,y\right) $
of the vector kernel $\mathbf{K}^{\alpha ,n}\left( x,y\right) =\frac{y-x}{%
\left\vert y-x\right\vert ^{n+1-\alpha }}$, given in formula (\ref{def
kernel pullback}), in the form%
\begin{eqnarray}
\Psi ^{\ast }\mathbf{K}^{\alpha ,n}\left( x,y\right) &=&\frac{\Psi \left(
y\right) -\Psi \left( x\right) }{\left\vert \Psi \left( y\right) -\Psi
\left( x\right) \right\vert ^{n+1-\alpha }}  \label{decomp} \\
&\equiv &\frac{\Psi ^{\prime }\left( x\right) \left( y-x\right) }{\left\vert
\Psi \left( y\right) -\Psi \left( x\right) \right\vert ^{n+1-\alpha }}+%
\mathbf{E}_{1}^{\alpha }\left( x,y\right)  \notag \\
&=&\Psi ^{\ast ,\tan ,1}\mathbf{K}^{\alpha ,n}\left( x,y\right) +\mathbf{E}%
_{1}^{\alpha }\left( x,y\right) ,  \notag
\end{eqnarray}%
where we write $\Psi ^{\prime }$ for the derivative $D\Psi $, and where $%
\Psi ^{\ast ,\tan ,1}\mathbf{K}^{\alpha ,n}$ is the first of our auxilliary
kernels defined above. We claim that the error kernel $\mathbf{E}%
_{1}^{\alpha }\left( x,y\right) $ satisfies the improved local estimate%
\begin{equation}
\mathbf{E}_{1}^{\alpha }\left( x,y\right) =O\left( \left\vert y-x\right\vert
^{\alpha -n+\delta };M\right)  \label{improved est}
\end{equation}%
for some constant $M$ depending on $\left\Vert \Gamma \right\Vert _{1,\delta
}$, $\left\Vert \Psi \right\Vert _{C^{1,\delta }}$ and $\left\Vert \left(
\Psi ^{\prime }\right) ^{-1}\right\Vert _{\infty }$. Indeed, we have%
\begin{eqnarray*}
&&\frac{\Psi \left( y\right) -\Psi \left( x\right) }{\left\vert \Psi \left(
y\right) -\Psi \left( x\right) \right\vert ^{n+1-\alpha }}-\frac{\Psi
^{\prime }\left( x\right) \left( y-x\right) }{\left\vert \Psi \left(
y\right) -\Psi \left( x\right) \right\vert ^{n+1-\alpha }} \\
&=&\frac{1}{\left\vert \Psi \left( y\right) -\Psi \left( x\right)
\right\vert ^{n+1-\alpha }}\left\{ \Psi \left( y\right) -\Psi \left(
x\right) -\Psi ^{\prime }\left( x\right) \left( y-x\right) \right\} .
\end{eqnarray*}%
Now we use the estimate%
\begin{equation}
\left\vert \Psi \left( y\right) -\Psi \left( x\right) -\Psi ^{\prime }\left(
x\right) \left( y-x\right) \right\vert \leq \left\Vert \Psi \right\Vert
_{C^{1+\delta }}\left\vert y-x\right\vert ^{1+\delta }=O\left( \left\vert
y-x\right\vert ^{1+\delta };\left\Vert \Psi \right\Vert _{C^{1+\delta
}}\right) ,  \label{first}
\end{equation}%
together with the bound $\left\vert \Gamma \left( \Psi \left( x\right) ,\Psi
\left( y\right) \right) \right\vert \leq \left\Vert \Gamma \right\Vert
_{\infty }$, to obtain (\ref{improved est}):%
\begin{equation*}
\left\vert \mathbf{E}_{1}^{\alpha }\left( x,y\right) \right\vert \lesssim 
\frac{1}{\left\vert y-x\right\vert ^{n+1-\alpha }}\left\vert y-x\right\vert
^{1+\delta }=\left\vert y-x\right\vert ^{\alpha -n+\delta },
\end{equation*}%
for $\left\vert y-x\right\vert \leq C$.

Recall that both $\widetilde{\sigma }$ and $\widetilde{\omega }$ are
supported in a fixed compact set $K$. Let $C_{2B}$ be a sufficiently large
constant exceeding the diameter of $2B$. Now we bound the operator norm of
the error term. For this we write $\widetilde{\sigma }=\Psi ^{\ast }\sigma $
and $\widetilde{\omega }=\Psi ^{\ast }\omega $ for convenience, and then
observe that the norm of the error operator%
\begin{equation*}
\mathcal{E}_{1}^{\alpha }f\widetilde{\sigma }\left( x\right) \equiv \int 
\mathbf{E}_{1}^{\alpha }\left( x,y\right) f\left( y\right) d\widetilde{%
\sigma }\left( y\right)
\end{equation*}%
as a map from $L^{2}\left( \widetilde{\sigma }\right) $ to $L^{2}\left( 
\widetilde{\omega }\right) $ is controlled by the offset $A_{2}^{\alpha }$
constant. Indeed, from the definition of $\mathbf{E}_{1}^{\alpha }\left(
x,y\right) $ in (\ref{decomp}), we see that the kernel $\mathbf{E}%
_{1}^{\alpha }$ vanishes on the diagonal, and so for $a\sim \log _{2}\frac{1%
}{C_{K}}$,%
\begin{eqnarray*}
\left\vert \mathcal{E}_{1}^{\alpha }f\widetilde{\sigma }\left( x\right)
\right\vert &\leq &\sum_{k=a}^{\infty }\int_{B\left( x,2^{-k}\right)
\setminus B\left( x,2^{-k-1}\right) }M\left\vert y-x\right\vert ^{\alpha
+\delta -n}\left\vert f\left( y\right) \right\vert d\widetilde{\sigma }%
\left( y\right) \\
&\lesssim &\sum_{k=a}^{\infty }2^{-k\delta }\left\vert B\left(
x,2^{-k}\right) \right\vert ^{\frac{\alpha }{n}-1}\int_{B\left(
x,2^{-k}\right) \setminus B\left( x,2^{-k-1}\right) }\left\vert f\left(
y\right) \right\vert d\widetilde{\sigma }\left( y\right) \\
&\lesssim &\sum_{k=a}^{\infty }2^{-k\delta }\mathcal{A}_{\alpha }^{k}\left( f%
\widetilde{\sigma }\right) \left( x\right) ,
\end{eqnarray*}%
where $\mathcal{A}_{\alpha ,C_{K}}^{k}$ is the \emph{annular }$\alpha $%
-averaging operator given by 
\begin{equation*}
\mathcal{A}_{\alpha }^{k}\left( f\widetilde{\sigma }\right) \left( x\right)
\equiv \left\vert B\left( x,2^{-k}\right) \right\vert ^{\frac{\alpha }{n}%
-1}\int_{B\left( x,2^{-k}\right) \setminus B\left( x,2^{-k-1}\right)
}\left\vert f\right\vert d\widetilde{\sigma }.
\end{equation*}

We now claim that the boundedness of $\mathcal{A}_{\alpha }^{k}$, and hence
also that of $\mathcal{E}_{1}^{\alpha }$, is controlled by the offset $%
A_{2}^{\alpha }$ constant. Indeed, for a sufficiently small positive
constant $c$, we have%
\begin{eqnarray*}
\left\Vert \mathcal{A}_{\alpha }^{k}\left( f\widetilde{\sigma }\right)
\right\Vert _{L^{2}\left( \widetilde{\omega }\right) }^{2} &\leq &\int_{%
\mathbb{R}^{n}}\left( 2^{-k\left( \alpha -n\right) }\int_{B\left(
x,2^{-k}\right) \setminus B\left( x,2^{-k-1}\right) }\left\vert f\right\vert
d\widetilde{\sigma }\right) ^{2}d\widetilde{\omega }\left( x\right) \\
&\leq &2^{-2k\left( \alpha -n\right) }\int_{\mathbb{R}^{n}}\left(
\int_{B\left( x,2^{-k}\right) }\left\vert f\right\vert ^{2}d\widetilde{%
\sigma }\right) \left\vert B\left( x,2^{-k}\right) \setminus B\left(
x,2^{-k-1}\right) \right\vert _{\widetilde{\sigma }}d\widetilde{\omega }%
\left( x\right) \\
&=&2^{-2k\left( \alpha -n\right) }\sum_{z\in \mathbb{Z}^{n}}\int_{B\left(
c2^{-k}z,\sqrt{n}c2^{-k}\right) }\left( \int_{B\left( x,2^{-k}\right)
}\left\vert f\right\vert ^{2}d\widetilde{\sigma }\right) \left\vert B\left(
x,2^{-k}\right) \setminus B\left( x,2^{-k-1}\right) \right\vert _{\widetilde{%
\sigma }}d\widetilde{\omega }\left( x\right) \\
&\leq &2^{-2k\left( \alpha -n\right) }\sum_{z\in \mathbb{Z}%
^{n}}\int_{B\left( c2^{-k}z,\sqrt{n}c2^{-k}\right) }\left( \int_{B\left(
c2^{-k}z,\left( \sqrt{n}c+1\right) 2^{-k}\right) }\left\vert f\right\vert
^{2}d\widetilde{\sigma }\right) \\
&&\ \ \ \ \ \ \ \ \ \ \ \ \ \ \ \ \ \ \ \ \ \ \ \ \ \ \ \ \ \ \times
\left\vert B\left( c2^{-k}z,\left( \sqrt{n}c+1\right) 2^{-k}\right)
\setminus B\left( c2^{-k}z,nc2^{-k}\right) \right\vert _{\widetilde{\sigma }%
}d\widetilde{\omega }\left( x\right) , \\
&=&\sum_{z\in \mathbb{Z}^{n}}\frac{\left\vert B\left( c2^{-k}z,\sqrt{n}%
c2^{-k}\right) \right\vert _{\widetilde{\omega }}\left\vert B\left(
c2^{-k}z,\left( \sqrt{n}c+1\right) 2^{-k}\right) \setminus B\left(
cn2^{-k}z,nc2^{-k}\right) \right\vert _{\widetilde{\sigma }}}{2^{2k\left(
\alpha -n\right) }} \\
&&\ \ \ \ \ \ \ \ \ \ \ \ \ \ \ \ \ \ \ \ \ \ \ \ \ \ \ \ \ \ \times
\int_{B\left( 2^{-k}z,\left( \sqrt{n}+1\right) 2^{-k}\right) }\left\vert
f\right\vert ^{2}d\widetilde{\sigma }\ .
\end{eqnarray*}%
Using the separation between $B\left( c2^{-k}z,\left( \sqrt{n}c+1\right)
2^{-k}\right) \setminus B\left( cn2^{-k}z,nc2^{-k}\right) $ and $B\left(
c2^{-k}z,\sqrt{n}c2^{-k}\right) $, it is easy to see that%
\begin{equation*}
\frac{\left\vert B\left( c2^{-k}z,\sqrt{n}c2^{-k}\right) \right\vert _{%
\widetilde{\omega }}\ \left\vert B\left( c2^{-k}z,\left( \sqrt{n}c+1\right)
2^{-k}\right) \setminus B\left( cn2^{-k}z,nc2^{-k}\right) \right\vert _{%
\widetilde{\sigma }}}{2^{2k\left( \alpha -n\right) }}\lesssim A_{2}^{\alpha
}\ .
\end{equation*}%
Combining inequalities we then obtain%
\begin{eqnarray*}
\left\Vert \mathcal{A}_{\alpha }^{k}\left( f\widetilde{\sigma }\right)
\right\Vert _{L^{2}\left( \widetilde{\omega }\right) }^{2} &\lesssim
&\sum_{z\in \mathbb{Z}^{n}}A_{2}^{\alpha }\int_{B\left( 2^{-k}z,\left( \sqrt{%
n}+1\right) 2^{-k}\right) }\left\vert f\right\vert ^{2}d\widetilde{\sigma }
\\
&=&A_{2}^{\alpha }\int_{\mathbb{R}^{n}}\sum_{z\in \mathbb{Z}^{n}}\mathbf{1}%
_{B\left( 2^{-k}z,\left( \sqrt{n}+1\right) 2^{-k}\right) }\left\vert
f\right\vert ^{2}d\widetilde{\sigma } \\
&\lesssim &A_{2}^{\alpha }\int_{\mathbb{R}^{n}}\left\vert f\right\vert ^{2}d%
\widetilde{\sigma }=A_{2}^{\alpha }\left\Vert f\right\Vert _{L^{2}\left( 
\widetilde{\sigma }\right) }^{2}\ ,
\end{eqnarray*}%
and hence%
\begin{equation*}
\left\Vert \mathcal{E}_{1}^{\alpha }f\widetilde{\sigma }\right\Vert
_{L^{2}\left( \widetilde{\omega }\right) }\lesssim \sum_{k=a}^{\infty
}2^{-k\delta }\left\Vert \mathcal{A}_{\alpha }^{k}\left( f\widetilde{\sigma }%
\right) \right\Vert _{L^{2}\left( \widetilde{\omega }\right) }\lesssim
\sum_{k=a}^{\infty }2^{-k\delta }A_{2}^{\alpha }\left\Vert f\right\Vert
_{L^{2}\left( \widetilde{\sigma }\right) }\lesssim C_{2B}^{\delta
}A_{2}^{\alpha }\left\Vert f\right\Vert _{L^{2}\left( \widetilde{\sigma }%
\right) }\ .
\end{equation*}

This completes the proof that the norm of the error operator $\mathcal{E}%
_{1}^{\alpha }$ as a map from $L^{2}\left( \widetilde{\sigma }\right) $ to $%
L^{2}\left( \widetilde{\omega }\right) $ is controlled by the offset $%
A_{2}^{\alpha }$ constant. For reference in proving statements (2) and (3)
below, we record that in similar fashion, using the reduction that $%
\widetilde{\sigma }$ also has compact support, that the norm of the dual
error operator%
\begin{eqnarray*}
\mathcal{E}_{2}^{\alpha }f\widetilde{\sigma }\left( x\right) &\equiv &\int 
\mathbf{E}_{2}^{\alpha }\left( x,y\right) f\left( y\right) d\widetilde{%
\sigma }\left( y\right) ; \\
\Psi ^{\ast }\mathbf{K}^{\alpha ,n}\left( x,y\right) &=&\Psi ^{\ast ,\tan ,2}%
\mathbf{K}^{\alpha ,n}\left( x,y\right) +\mathbf{E}_{2}^{\alpha }\left(
x,y\right) ,
\end{eqnarray*}%
as a map from $L^{2}\left( \widetilde{\omega }\right) $ to $L^{2}\left( 
\widetilde{\sigma }\right) $ is controlled by the offset $A_{2}^{\alpha }$
constant.

Now we further analyze the first kernel on the right hand side of (\ref%
{decomp}), namely%
\begin{equation*}
\Psi ^{\ast ,\tan ,1}\mathbf{K}^{\alpha ,n}\left( x,y\right) =\frac{\Psi
^{\prime }\left( x\right) \left( y-x\right) }{\left\vert \Psi \left(
y\right) -\Psi \left( x\right) \right\vert ^{n+1-\alpha }},
\end{equation*}%
by writing%
\begin{equation*}
\Psi ^{\ast ,\tan ,1}\mathbf{K}^{\alpha ,n}\left( x,y\right) =\Psi ^{\prime
}\left( x\right) \frac{y-x}{\left\vert \Psi \left( y\right) -\Psi \left(
x\right) \right\vert ^{n+1-\alpha }}=\Psi ^{\prime }\left( x\right) \mathbf{K%
}_{\Psi }^{\alpha ,n}\left( x,y\right) .
\end{equation*}%
We now compute that%
\begin{eqnarray}
&&\int_{\mathbb{R}^{n}}\left\vert \Psi ^{\ast ,\tan ,1}\mathbf{R}^{\alpha
,n}f\widetilde{\sigma }\right\vert ^{2}d\widetilde{\omega }
\label{compute that} \\
&=&\int_{\mathbb{R}^{n}}\left\vert \int \Psi ^{\ast ,\tan ,1}\mathbf{K}%
^{\alpha ,n}\left( x,y\right) f\left( y\right) d\widetilde{\sigma }\left(
y\right) \right\vert ^{2}d\widetilde{\omega }\left( x\right)  \notag \\
&=&\int_{\mathbb{R}^{n}}\left\vert \int \Psi ^{\prime }\left( x\right) 
\mathbf{K}_{\Psi }^{\alpha ,n}\left( x,y\right) f\left( y\right) d\widetilde{%
\sigma }\left( y\right) \right\vert ^{2}d\widetilde{\omega }\left( x\right) 
\notag \\
&=&\int_{\mathbb{R}^{n}}\left\vert \Psi ^{\prime }\left( x\right) \left\{
\int \mathbf{K}_{\Psi }^{\alpha ,n}\left( x,y\right) f\left( y\right) d%
\widetilde{\sigma }\left( y\right) \right\} \right\vert ^{2}d\widetilde{%
\omega }\left( x\right)  \notag \\
&=&\int_{\mathbb{R}^{n}}\left\vert \Psi ^{\prime }\left( x\right) \int 
\mathbf{R}_{\Psi }^{\alpha ,n}\left( f\widetilde{\sigma }\right) \left(
x\right) \right\vert ^{2}d\widetilde{\omega }\left( x\right) ,  \notag
\end{eqnarray}%
where the matrix $\Psi ^{\prime }\left( x\right) $ is acting on the vector $%
\mathbf{T}_{\Gamma _{\Psi }}^{\alpha ,n}f\widetilde{\sigma }\left( x\right)
=\int \mathbf{K}_{\Gamma _{\Psi }}^{\alpha ,n}\left( x,y\right) f\left(
y\right) d\widetilde{\sigma }\left( y\right) $. Using the inequality%
\begin{equation}
\left\vert \Psi ^{\prime }\left( x\right) v\right\vert \approx \left\vert
v\right\vert ,\ \ \ \ \ \text{uniformly in }x,  \label{uniform}
\end{equation}%
we conclude that%
\begin{equation*}
\int_{\mathbb{R}^{n}}\left\vert \Psi ^{\ast ,\tan ,1}\mathbf{R}^{\alpha ,n}f%
\widetilde{\sigma }\right\vert ^{2}d\widetilde{\omega }\approx \int_{\mathbb{%
R}^{n}}\left\vert \mathbf{R}_{\Psi }^{\alpha ,n}f\widetilde{\sigma }%
\right\vert ^{2}d\widetilde{\omega },
\end{equation*}%
which shows that 
\begin{equation*}
\mathfrak{N}_{\mathbf{R}_{\Psi }^{\alpha ,n}}\left( \widetilde{\sigma },%
\widetilde{\omega }\right) \approx \mathfrak{N}_{\Psi ^{\ast ,\tan ,1}%
\mathbf{R}^{\alpha ,n}}\left( \widetilde{\sigma },\widetilde{\omega }\right)
.
\end{equation*}%
Similarly we have%
\begin{equation*}
\mathfrak{N}_{\mathbf{R}_{\Psi }^{\alpha ,n}}\left( \widetilde{\sigma },%
\widetilde{\omega }\right) \approx \mathfrak{N}_{\Psi ^{\ast ,\tan ,2}%
\mathbf{R}^{\alpha ,n}}\left( \widetilde{\sigma },\widetilde{\omega }\right)
.
\end{equation*}

Reverting to the notation with $\Psi ^{\ast }$ it now follows from this, and
then the boundedness of the error operator $\mathcal{E}^{\alpha }$, that%
\begin{eqnarray*}
\mathfrak{N}_{\mathbf{R}_{\Psi }^{\alpha }}\left( \Psi ^{\ast }\sigma ,\Psi
^{\ast }\omega \right) &\approx &\mathfrak{N}_{\Psi ^{\ast ,\tan ,1}\mathbf{R%
}^{\alpha ,n}}\left( \Psi ^{\ast }\sigma ,\Psi ^{\ast }\omega \right) \\
&\lesssim &\mathfrak{N}_{\Psi ^{\ast }\mathbf{R}^{\alpha ,n}}\left( \Psi
^{\ast }\sigma ,\Psi ^{\ast }\omega \right) +\left[ \mathcal{A}_{2}^{\alpha
}+\mathcal{A}_{2}^{\alpha ,\func{dual}}\right] \\
&=&\mathfrak{N}_{\mathbf{R}^{\alpha }}\left( \sigma ,\omega \right) +\left[ 
\mathcal{A}_{2}^{\alpha }+\mathcal{A}_{2}^{\alpha ,\func{dual}}\right] ,
\end{eqnarray*}%
where the final equality is (\ref{norm equality}). The reverse inequality in
the second statement of Proposition \ref{change of variable} is proved in
similar fashion, or simply by replacing $\Psi $ with $\Psi ^{-1}$.

The second and third statements are proved in the same way as the first
statement just proved above, but with the following difference. The
functions $f$ under consideration are restricted to indicators $f=\mathbf{1}%
_{E}$ with $E\in \Psi ^{\ast }\mathcal{F}$, and as a result we have from (%
\ref{compute that}), and its dual version, the two indentities%
\begin{equation*}
\int_{\mathbb{R}^{n}}\left\vert \Psi ^{\ast ,\tan ,1}\mathbf{R}^{\alpha ,n}%
\mathbf{1}_{E}\widetilde{\sigma }\right\vert ^{2}d\widetilde{\omega }=\int_{%
\mathbb{R}^{n}}\left\vert \Psi ^{\prime }\left( x\right) \mathbf{R}_{\Psi
}^{\alpha ,n}\left( \mathbf{1}_{E}\widetilde{\sigma }\right) \left( x\right)
\right\vert ^{2}d\widetilde{\omega }\left( x\right) ,
\end{equation*}%
and%
\begin{equation*}
\int_{\mathbb{R}^{n}}\left\vert \Psi ^{\ast ,\tan ,2}\mathbf{R}^{\alpha ,n}%
\mathbf{1}_{E}\widetilde{\omega }\right\vert ^{2}d\widetilde{\sigma }=\int_{%
\mathbb{R}^{n}}\left\vert \Psi ^{\prime }\left( y\right) \mathbf{R}_{\Psi
}^{\alpha ,n}\left( \mathbf{1}_{E}\widetilde{\omega }\right) \left( y\right)
\right\vert ^{2}d\widetilde{\sigma }\left( y\right) ,
\end{equation*}%
since the kernels $\mathbf{K}^{\alpha ,n}\left( x,y\right) $ and $\mathbf{K}%
_{\Gamma _{\Psi }}^{\alpha ,n}\left( x,y\right) $ are antisymmetric. Just as
for the norm estimate above, we use (\ref{uniform}) to obtain both%
\begin{eqnarray*}
\int_{\mathbb{R}^{n}}\left\vert \Psi ^{\ast ,\tan ,1}\mathbf{R}^{\alpha ,n}%
\mathbf{1}_{E}\widetilde{\sigma }\right\vert ^{2}d\widetilde{\omega }
&\approx &\int_{\mathbb{R}^{n}}\left\vert \mathbf{R}_{\Psi }^{\alpha
,n}\left( x,y\right) \mathbf{1}_{E}d\widetilde{\sigma }\right\vert ^{2}d%
\widetilde{\omega }; \\
\mathfrak{T}_{\Psi ^{\ast ,\tan ,1}\mathbf{R}^{\alpha ,n}}^{\Psi ^{\ast }%
\mathcal{F}}\left( \Psi ^{\ast }\sigma ,\Psi ^{\ast }\omega \right) &\approx
&\mathfrak{T}_{\mathbf{R}_{\Psi }^{\alpha }}^{\Psi ^{\ast }\mathcal{F}%
}\left( \Psi ^{\ast }\sigma ,\Psi ^{\ast }\omega \right) ,
\end{eqnarray*}%
and%
\begin{eqnarray*}
\int_{\mathbb{R}^{n}}\left\vert \Psi ^{\ast ,\tan ,2}\mathbf{R}^{\alpha ,n}%
\mathbf{1}_{E}\widetilde{\omega }\right\vert ^{2}d\widetilde{\sigma }
&\approx &\int_{\mathbb{R}^{n}}\left\vert \mathbf{R}_{\Psi }^{\alpha
,n}\left( x,y\right) \mathbf{1}_{E}d\widetilde{\omega }\right\vert ^{2}d%
\widetilde{\sigma }; \\
\mathfrak{T}_{\Psi ^{\ast ,\tan ,2}\mathbf{R}^{\alpha ,n}}^{\Psi ^{\ast }%
\mathcal{F},\func{dual}}\left( \Psi ^{\ast }\sigma ,\Psi ^{\ast }\omega
\right) &\approx &\mathfrak{T}_{\mathbf{R}_{\Psi }^{\alpha }}^{\Psi ^{\ast }%
\mathcal{F},\func{dual}}\left( \Psi ^{\ast }\sigma ,\Psi ^{\ast }\omega
\right) .
\end{eqnarray*}%
Now, noting that both of the measures $\widetilde{\sigma }$ and $\widetilde{%
\omega }$ are compactly supported, we use that both of the error operators $%
\mathcal{E}_{1}^{\alpha }f\widetilde{\sigma }\left( x\right) \equiv \int 
\mathbf{E}_{1}^{\alpha }\left( x,y\right) f\left( y\right) d\widetilde{%
\sigma }\left( y\right) $ and $\mathcal{E}_{2}^{\alpha }f\widetilde{\omega }%
\left( x\right) \equiv \int \mathbf{E}_{2}^{\alpha }\left( x,y\right)
f\left( y\right) d\widetilde{\omega }\left( y\right) $ have norms controlled
by the offset $A_{2}^{\alpha }$ condition to obtain%
\begin{equation*}
\mathfrak{T}_{\Psi ^{\ast ,\tan ,1}\mathbf{R}^{\alpha ,n}}^{\Psi ^{\ast }%
\mathcal{F}}\left( \Psi ^{\ast }\sigma ,\Psi ^{\ast }\omega \right) +\left[ 
\mathcal{A}_{2}^{\alpha }+\mathcal{A}_{2}^{\alpha ,\func{dual}}\right]
\approx \mathfrak{T}_{\mathbf{R}^{\alpha ,n}}^{\mathcal{F}}\left( \sigma
,\omega \right) +\left[ \mathcal{A}_{2}^{\alpha }+\mathcal{A}_{2}^{\alpha ,%
\func{dual}}\right] ,
\end{equation*}%
and%
\begin{equation*}
\mathfrak{T}_{\Psi ^{\ast ,\tan ,2}\mathbf{R}^{\alpha ,n}}^{\Psi ^{\ast }%
\mathcal{F},\func{dual}}\left( \Psi ^{\ast }\sigma ,\Psi ^{\ast }\omega
\right) +\left[ \mathcal{A}_{2}^{\alpha }+\mathcal{A}_{2}^{\alpha ,\func{dual%
}}\right] \approx \mathfrak{T}_{\mathbf{R}^{\alpha ,n}}^{\mathcal{F},\func{%
dual}}\left( \sigma ,\omega \right) +\left[ \mathcal{A}_{2}^{\alpha }+%
\mathcal{A}_{2}^{\alpha ,\func{dual}}\right] .
\end{equation*}%
Note that once again we need the one-tailed Muckenhoupt conditions to reduce
to the case where both measures have common compact support. Combining
inequalities we have%
\begin{eqnarray*}
\mathfrak{T}_{\mathbf{R}_{\Psi }^{\alpha }}^{\Psi ^{\ast }\mathcal{F}}\left(
\Psi ^{\ast }\sigma ,\Psi ^{\ast }\omega \right) +\left[ \mathcal{A}%
_{2}^{\alpha }+\mathcal{A}_{2}^{\alpha ,\func{dual}}\right] &\approx &%
\mathfrak{T}_{\mathbf{R}^{\alpha ,n}}^{\mathcal{F}}\left( \sigma ,\omega
\right) +\left[ \mathcal{A}_{2}^{\alpha }+\mathcal{A}_{2}^{\alpha ,\func{dual%
}}\right] , \\
\mathfrak{T}_{\mathbf{R}_{\Psi }^{\alpha }}^{\Psi ^{\ast }\mathcal{F},\func{%
dual}}\left( \Psi ^{\ast }\sigma ,\Psi ^{\ast }\omega \right) +\left[ 
\mathcal{A}_{2}^{\alpha }+\mathcal{A}_{2}^{\alpha ,\func{dual}}\right]
&\approx &\mathfrak{T}_{\mathbf{R}^{\alpha ,n}}^{\mathcal{F},\func{dual}%
}\left( \sigma ,\omega \right) +\left[ \mathcal{A}_{2}^{\alpha }+\mathcal{A}%
_{2}^{\alpha ,\func{dual}}\right] ,
\end{eqnarray*}%
and this completes the proof of Proposition \ref{change of variable}.
\end{proof}

\subsection{A preliminary $T1$ theorem}

We can use just the change of variable Proposition \ref{change of variable}
and Theorem \ref{final} to prove the preliminary Theorem \ref{curve}. Recall
that $\mathcal{L}$ is presented as the graph of a $C^{1,\delta }$ function $%
\psi :\mathbb{R}\rightarrow \mathbb{R}^{n}$ given by%
\begin{equation*}
\psi \left( t\right) =\left( \psi ^{2}\left( t\right) ,\psi ^{3}\left(
t\right) ,...,\psi ^{n}\left( t\right) \right) ,
\end{equation*}%
and that both%
\begin{eqnarray*}
\Psi \left( x\right) &=&\left( x^{1},x^{2}-\psi ^{2}\left( x^{1}\right)
,x^{3}-\psi ^{3}\left( x^{1}\right) ,...,x^{n}-\psi ^{n}\left( x^{1}\right)
\right) =x-\left( 0,\psi \left( x^{1}\right) \right) , \\
\Psi ^{-1}\left( \xi \right) &=&\left( \xi ^{1},\xi ^{2}+\psi ^{2}\left( \xi
_{1}\right) ,\xi ^{3}+\psi ^{3}\left( \xi _{1}\right) ,...,\xi ^{n}+\psi
^{n}\left( \xi _{1}\right) \right) =\xi +\psi \left( \xi _{1}\right) ,
\end{eqnarray*}%
are $C^{1,\delta }$ maps, and that $\Psi $ is a $C^{1,\delta }$
homeomorphism from the curve $\mathcal{L}$ to the $x_{1}$-axis. Recall $\Psi
_{\ast }\mathcal{Q}=\left( \Psi ^{-1}\right) ^{\ast }\mathcal{Q}=\left\{
\Psi Q:Q\in \mathcal{Q}\right\} $. In the next subsection, the small
Lipschitz assumption (\ref{small Lip}) will be removed, and the testing
conditions below will be permitted to be taken over usual cubes.

Finally recall that in Theorem \ref{curve}, we assume the small Lipschitz
condition (\ref{small Lip}), i.e. 
\begin{equation*}
\left\Vert D\psi \right\Vert _{\infty }<\frac{1}{8n}\left( 1-\frac{\alpha }{n%
}\right) ,
\end{equation*}%
and that $\omega $ and $\sigma $ are positive Borel measures (possibly
having common point masses) with $\omega $ compactly supported in $\mathcal{L%
}$, and that $\mathcal{R}^{n}\mathcal{=}R\mathcal{P}^{n}$ where $R$ is a
fixed rotation that is $L$-transverse when $L$ is the $x_{1}$-axis. The
conclusion of Theorem \ref{curve} is then that%
\begin{eqnarray*}
&&\int_{\Psi Q}\left\vert \mathbf{R}^{\alpha ,n}\left( \mathbf{1}_{\Psi
Q}\sigma \right) \right\vert ^{2}d\omega \leq \mathfrak{T}_{\mathbf{R}%
^{\alpha ,n}}^{\Psi \mathcal{Q}}\left\vert \Psi Q\right\vert _{\sigma }, \\
&&\int_{\Psi Q}\left\vert \mathbf{R}^{\alpha ,n,\func{dual}}\left( \mathbf{1}%
_{\Psi Q}\omega \right) \right\vert ^{2}d\sigma \leq \mathfrak{T}_{\mathbf{R}%
^{\alpha ,n}}^{\Psi \mathcal{Q},\func{dual}}\left\vert \Psi Q\right\vert
_{\omega }, \\
&&\ \ \ \ \ \text{for all cubes }Q\in \mathcal{R}^{n}.
\end{eqnarray*}

\begin{proof}[Proof of Theorem \protect\ref{curve}]
By the testing equivalences (2) and (3) of Proposition \ref{change of
variable} with $\mathcal{F}=\Psi \mathcal{Q}$, and using $\Psi ^{\ast }%
\mathcal{F=}\Psi ^{\ast }\Psi \mathcal{Q=Q}$, we have 
\begin{eqnarray*}
&&\sqrt{\mathfrak{A}_{2}^{\alpha }\left( \sigma ,\omega \right) }+\mathfrak{T%
}_{\mathbf{R}^{\alpha ,n}}^{\Psi \mathcal{Q}}\left( \sigma ,\omega \right) +%
\mathfrak{T}_{\mathbf{R}^{\alpha ,n}}^{\Psi \mathcal{Q},\func{dual}}\left(
\sigma ,\omega \right) \\
&\approx &\sqrt{\mathfrak{A}_{2}^{\alpha }\left( \sigma ,\omega \right) }+%
\mathfrak{T}_{\mathbf{R}_{\Psi }^{\alpha ,n}}^{\mathcal{Q}}\left( \Psi
^{\ast }\sigma ,\Psi ^{\ast }\omega \right) +\mathfrak{T}_{\mathbf{R}_{\Psi
}^{\alpha ,n}}^{\mathcal{Q},\func{dual}}\left( \Psi ^{\ast }\sigma ,\Psi
^{\ast }\omega \right) \\
&\approx &\sqrt{\mathfrak{A}_{2}^{\alpha }\left( \sigma ,\omega \right) }+%
\mathfrak{N}_{\mathbf{R}_{\Psi }^{\alpha ,n}}\left( \Psi ^{\ast }\sigma
,\Psi ^{\ast }\omega \right) ,
\end{eqnarray*}%
where the final line follows from Theorem \ref{final} because $\Psi ^{\ast
}\omega $ is supported on a line. Then we continue with equivalence (1) of
Proposition \ref{change of variable} to obtain%
\begin{equation*}
\sqrt{\mathfrak{A}_{2}^{\alpha }\left( \sigma ,\omega \right) }+\mathfrak{N}%
_{\mathbf{R}_{\Psi }^{\alpha ,n}}\left( \Psi ^{\ast }\sigma ,\Psi ^{\ast
}\omega \right) \approx \sqrt{\mathfrak{A}_{2}^{\alpha }\left( \sigma
,\omega \right) }+\mathfrak{N}_{\mathbf{R}^{\alpha ,n}}\left( \sigma ,\omega
\right) .
\end{equation*}%
Altogether we now obtain from this that%
\begin{equation*}
\mathfrak{N}_{\mathbf{R}^{\alpha ,n}}\left( \sigma ,\omega \right) \approx 
\sqrt{\mathfrak{A}_{2}^{\alpha }\left( \sigma ,\omega \right) }+\mathfrak{T}%
_{\mathbf{R}^{\alpha ,n}}^{\Psi \mathcal{Q}}\left( \sigma ,\omega \right) +%
\mathfrak{T}_{\mathbf{R}^{\alpha ,n}}^{\Psi \mathcal{Q},\func{dual}}\left(
\sigma ,\omega \right) ,
\end{equation*}%
and this completes the proof of Theorem \ref{curve}.
\end{proof}

\begin{remark}
At this point, one can obtain a `$T1$ type' theorem when $\omega $ is
compactly supported on a $C^{1,\delta }$ curve $\mathcal{L}$, \emph{without}
any additional restriction on the Lipschitz constant of the curve, by
decomposing $\omega =\sum_{i=1}^{N}\omega _{i}$ into finitely many measures $%
\omega _{i}$ with support so small that the supporting curve is presented as
a graph of a $C^{1,\delta }$ function $\psi _{i}$ relative to a rotated
axis, and such that $\left\Vert D\psi _{i}\right\Vert _{\infty }<\frac{1}{8n}%
\left( 1-\frac{\alpha }{n}\right) $ (this requires some work). Then Theorem %
\ref{curve} applies to each measure pair $\left( \omega _{i},\sigma \right) $
(appropriately rotated), and the corresponding rotated quasitesting
conditions must now be taken over the finitely many measure pairs $\left(
\omega _{i},\sigma \right) $. In the next subsection we will improve on this
observation by eliminating the small Lipschitz assumption, and by taking the
testing conditions over the single measure pair $\left( \omega ,\sigma
\right) $.
\end{remark}

\subsection{The $T1$ theorem for a measure supported on a regular $C^{1,%
\protect\delta }$ curve\label{Sub curve}}

Now we prove our main result, Theorem \ref{general T1}.

\begin{proof}[Proof of Theorem \protect\ref{general T1}]
\textbf{Step 1}: Given $0\leq \alpha <n$, we define%
\begin{equation*}
\varepsilon \equiv \frac{1}{8n}\left( 1-\frac{\alpha }{n}\right) ,
\end{equation*}%
where the right hand side is the constant appearing in (\ref{small Lip}) in
Theorem \ref{curve'}. Now let $0<\varepsilon ^{\prime }<\varepsilon $ and
choose a finite collection of points $\left\{ \xi _{j}\right\}
_{j=1}^{j}\subset \mathbb{S}^{n-1}$ in the unit sphere such that the
spherical balls $\left\{ \mathcal{B}\left( \xi _{j},\frac{\varepsilon
^{\prime }}{4}\right) \right\} _{j=1}^{J}$ cover $\mathbb{S}^{n-1}$. Observe
that our curve $\Phi $ and its derivative $\Phi ^{\prime }$ are \emph{%
uniformly} continuous. We now claim that we can decompose the curve $%
\mathcal{L}$ into finitely many consecutive pieces $\left\{ \mathcal{L}%
_{i}\right\} _{i=0}^{N}$ such that with $\widehat{\mathcal{L}_{i}}$ defined
to be $\dbigcup\limits_{k=-C_{n}}^{C_{n}}\mathcal{L}_{i+k}$, the union of $%
\mathcal{L}_{i}$ and the previous and subsequent $C_{n}$ pieces, and $%
\mathcal{L}_{i}^{\ast }$ defined to be $\dbigcup\limits_{k=-2C_{n}}^{2C_{n}}%
\mathcal{L}_{i+k}$, where $C_{n}=5\widetilde{C_{n}}$ and $\widetilde{C_{n}}$
is a dimensional constant defined in (\ref{latter}) in step 2 below (without
circularity), the following three properties hold:

(\textbf{1}): there is $\eta >0$ such that $\mathcal{L}_{i}=\limfunc{range}%
\Phi _{i}$ where $\Phi _{i}=\Phi \mid _{\left[ \eta i,\eta \left( i+1\right)
\right) }$ is the restriction of $\Phi $ to the interval $\left[ \eta i,\eta
\left( i+1\right) \right) $ for all $i$, and

(\textbf{2}): for each $i$, there is $j=j\left( i\right) $ depending on $i$
such that, after a rotation $R_{i}$ that takes the point $\xi _{j}$ to the
point $\left( 0,...,0,1\right) $, followed by an appropriate translation $%
T_{i}$, the curve $T_{i}R_{i}\widehat{\mathcal{L}_{i}}$ is the graph of the
restriction $\psi _{i}\mid _{\left[ -\zeta _{i},\zeta _{i}\right) }$ of a
globally defined $C^{1,\delta }$ function $\psi _{i}:\mathbb{R}\rightarrow 
\mathbb{R}^{n-1}$ with $\left\Vert D\psi _{i}\right\Vert _{\infty
}<\varepsilon $, and

(\textbf{3}): $\psi _{i}\left( t\right) =\left( 0,...,0\right) \in \mathbb{R}%
^{n-1}$ for all $\left\vert t\right\vert >4\zeta _{i}$.

Thus we are claiming that we can locally rotate and translate the curve so
that it is given locally as part of the graph of a globally defined $%
C^{1,\delta }$ function $\psi _{i}$ with $\left\Vert D\psi _{i}\right\Vert
_{\infty }<\varepsilon $, where in view of the definition of $\varepsilon $,
this latter inequality is what is required in (\ref{small Lip}) of Theorem %
\ref{curve}. Note that $\zeta _{i}\approx C_{n}\eta $.

To see that these three properties can be obtained, we use uniform
continuity of $\Phi ^{\prime }$ to take a small piece $\mathcal{L}_{i}^{\ast
}$ of the curve, such that the oscillation of tangent lines across the piece 
$\mathcal{L}_{i}^{\ast }$ is less than $\varepsilon ^{\prime }$, and then
translate and rotate the chord joining its endpoints so as to lie on the $%
x_{1}$-axis. Note that with this done, the resulting curve is the graph of a
function $\psi _{i}^{\ast }\left( t\right) $ defined for $t\in I_{i}^{\ast }$%
, which satisfies 
\begin{equation*}
\left\vert \psi _{i}^{\ast }\left( t\right) \right\vert \leq C\frac{%
\varepsilon ^{\prime }}{2}\eta ,\ \ \ \ \ t\in I_{i}^{\ast },
\end{equation*}%
since $\psi _{i}^{\ast }=T_{i}R_{i}\Phi \mid _{I_{i}^{\ast }}$. Here we are
using the convention that $I_{i}$ is the parameter interval of $\Phi $
corresponding to the image $\mathcal{L}_{i}$, and similarly for $\widehat{%
I_{i}}$ and $I_{i}^{\ast }$.

Then we construct the extended function $\psi _{i}\left( t\right) $ so that
its graph includes the translated and rotated piece $\widehat{\mathcal{L}_{i}%
}$, and so that away from $\widehat{I_{i}}$ the function $\psi _{i}$
smoothly straightens out from $\psi _{i}^{\ast }$ so as to vanish on the
remaining $x_{1}$-axis, and in such a way that $\left\Vert D\psi
_{i}\right\Vert _{\infty }<\varepsilon $. This is most easily seen by taking 
$\psi _{i}\left( t\right) \equiv \psi _{i}^{\ast }\left( t\right) \rho
\left( t\right) $ where $\rho \left( t\right) $ is an appropriate smooth
bump function that is identically $1$ on $\widehat{I_{i}}$ and vanishes
outside $I_{i}^{\ast }$. Note then that 
\begin{equation*}
D\psi _{i}\left( t\right) =D\psi _{i}^{\ast }\left( t\right) \rho \left(
t\right) +\psi _{i}^{\ast }\left( t\right) D\rho \left( t\right)
\end{equation*}%
satisfies $\left\vert D\psi _{i}\left( t\right) \right\vert \leq
C\varepsilon ^{\prime }$ since $\left\vert \psi _{i}^{\ast }\left( t\right)
\right\vert \leq C\frac{\varepsilon ^{\prime }}{2}\eta $ and $\left\vert
D\rho \left( t\right) \right\vert \leq C\frac{1}{\eta }$. Consequently we
have 
\begin{equation*}
\left\Vert D\psi _{i}\right\Vert _{\infty }\leq C\varepsilon ^{\prime
}<\varepsilon .
\end{equation*}%
Of course the function $\psi _{i}=\psi _{i}^{\ast }\left( t\right) \rho
\left( t\right) $ is $C^{1,\delta }$ since $\limfunc{supp}\rho $ is
contained in the interior of the interval $I_{i}^{\ast }$. This completes
the verification of properties (1), (2) and (3) above.

In the next step, we will restrict $\omega $ to the small piece $\mathcal{L}%
_{i}$ and it will be important that we can straighten out the larger piece $%
\widehat{\mathcal{L}_{i}}$ via a global $C^{1,\delta }$ diffeomorphism $\Psi
_{i}^{-1}$ of $\mathbb{R}^{n}$ (defined using $\psi _{i}$), so that we can
derive a \emph{tripled} quasitesting condition for intervals in $I_{i}$
whose triples are contained in $\widehat{I_{i}}$, the straightened out
portion of $\widehat{\mathcal{L}_{i}}$.\\*[0.25in]

\textbf{Step 2}: We now apply Theorem \ref{main} to the pullbacks $\widehat{%
I_{i}}$ under $\psi _{i}$ of the localized pieces $\widehat{\mathcal{L}_{i}}$
as follows. Fix $i$ and denote by $\omega _{i}$ the restriction of $\omega $
to $\mathcal{L}_{i}$. We are assuming the usual cube testing conditions $%
\mathfrak{T}_{\mathbf{R}^{\alpha ,n}}^{\mathcal{Q}^{n}}\left( \sigma ,\omega
\right) $ and $\mathfrak{T}_{\mathbf{R}^{\alpha ,n}}^{\mathcal{Q}^{n},\func{%
dual}}\left( \sigma ,\omega \right) $ on the weight pair $\left( \sigma
,\omega \right) $ over all cubes $Q\in \mathcal{Q}^{n}$. Under the change of
variable given by the $C^{1,\delta }$ map $\Psi _{i}^{-1}:\mathbb{R}%
^{n}\rightarrow \mathbb{R}^{n}$, corresponding to $\Psi _{i}\left( x\right)
=\left( x_{1},x^{\prime }+\psi _{i}\left( x_{1}\right) \right) $, the pair
of measures $\left( \sigma ,\omega \right) $ is transformed to the pullback
pair $\left( \widetilde{\sigma },\widetilde{\omega }\right) $ (since $i$ is
fixed we suppress the dependence of the change of variable on $i$ and simply
write $\widetilde{\sigma }$ and $\widetilde{\omega }$, but we will use the
subscript $i$ to emphasize restrictions of $\widetilde{\omega }$). Now
define $\widetilde{\omega _{i}}$ to be the transform of the small piece of
measure $\omega _{i}$, and note that it is supported on the $x_{1}$-axis,
and moreover that the transform $\widehat{I_{i}}$ of the larger piece $%
\widehat{\mathcal{L}_{i}}$ is also supported on the $x_{1}$-axis. By
Proposition \ref{change of variable}, and in the presence of $\mathfrak{A}%
_{2}^{\alpha }$, the testing conditions $\mathfrak{T}_{\mathbf{R}^{\alpha
,n}}^{\mathcal{Q}^{n}}\left( \sigma ,\omega \right) $ and $\mathfrak{T}_{%
\mathbf{R}^{\alpha ,n}}^{\mathcal{Q}^{n},\func{dual}}\left( \sigma ,\omega
\right) $ for the measure pair $\left( \sigma ,\omega \right) $ over cubes
in $\mathcal{Q}^{n}$ for the $\alpha $-fractional Riesz transform $\mathbf{R}%
^{\alpha ,n}$ are transformed into the testing conditions $\mathfrak{T}_{%
\mathbf{R}_{\Psi }^{\alpha ,n}}^{\Psi _{i}^{\ast }\mathcal{Q}^{n}}\left(
\sigma ,\omega \right) $ and $\mathfrak{T}_{\mathbf{R}_{\Psi }^{\alpha
,n}}^{\Psi _{i}^{\ast }\mathcal{Q}^{n},\func{dual}}\left( \sigma ,\omega
\right) $ for the measure pair $\left( \widetilde{\sigma },\widetilde{\omega 
}\right) $ over quasicubes in $\Psi _{i}^{\ast }\mathcal{Q}^{n}$ for the
conformal $\alpha $-fractional Riesz transform $\mathbf{R}_{\Psi }^{\alpha
,n}$.

Now in order to apply Theorem \ref{main} to the conformal $\alpha $%
-fractional Riesz transform $\mathbf{R}_{\Psi }^{\alpha ,n}$, we will choose
below a specific rotation $\mathcal{R}^{n}=R\mathcal{P}^{n}$ of the
collection of cubes $\mathcal{P}^{n}$. Then we consider the testing
conditions for the pair $\left( \widetilde{\sigma },\widetilde{\omega }%
\right) $ over these quasicubes $\Psi _{i}^{\ast }\mathcal{R}^{n}$ that form
a subset of the quasicubes $\Psi _{i}^{\ast }\mathcal{Q}^{n}$. Provided we
choose $\varepsilon >0$ small enough in Step 1, the map $\Psi _{i}^{\ast }$
will have its derivative $D\Psi _{i}^{\ast }$ close to the identity $I$.
Choose a rotation $R$ that is $L$-transverse when $L$ is the $x_{1}$-axis.
From Lemma \ref{connected} applied with $\Omega =\Psi _{i}^{-1}\circ R$, we
then obtain (\ref{trans edges}) and the key geometric property:%
\begin{eqnarray}
&&\text{The intersection of any }Q\in \Psi _{i}^{\ast }\mathcal{R}^{n}\text{
with the }x_{1}\text{-axis}  \label{geometric} \\
&&\ \ \ \ \ \ \ \ \ \ \ \ \ \ \ \text{is an \textbf{interval }in the }x_{1}%
\text{-axis}.  \notag
\end{eqnarray}%
Using (\ref{trans edges}) and this geometric property we will now deduce,
for the special fractional Riesz transform $\mathbf{R}_{\Psi }^{\alpha ,n}$,
that in the presence of the $\mathcal{A}_{2}^{\alpha }$ conditions, the $%
\Psi _{i}^{\ast }\mathcal{R}^{n}$-quasicube testing conditions for the pair $%
\left( \widetilde{\sigma },\widetilde{\omega _{i}}\right) $ follow from the $%
\Psi _{i}^{\ast }\mathcal{R}^{n}$-quasicube testing conditions for the pair $%
\left( \widetilde{\sigma },\widetilde{\omega }\right) ${\Large .}

Indeed, fix a quasicube $Q\in \Psi _{i}^{\ast }\mathcal{R}^{n}$ and consider
the left hand sides of the two dual testing conditions, namely%
\begin{equation*}
\int_{Q}\left\vert \mathbf{R}_{\Psi }^{\alpha ,n}\mathbf{1}_{Q}\widetilde{%
\sigma }\right\vert ^{2}d\widetilde{\omega _{i}}\text{ and }%
\int_{Q}\left\vert \mathbf{R}_{\Psi }^{\alpha ,n}\mathbf{1}_{Q}\widetilde{%
\omega _{i}}\right\vert ^{2}d\widetilde{\sigma }.
\end{equation*}%
Now the first integral is trivially dominated by $\int_{Q}\left\vert \mathbf{%
R}_{\Psi }^{\alpha ,n}\mathbf{1}_{Q}\widetilde{\sigma }\right\vert ^{2}d%
\widetilde{\omega }\leq \left( \mathfrak{T}_{\mathbf{R}_{\Psi }^{\alpha
,n}}^{\Psi _{i}^{\ast }\mathcal{R}^{n}}\right) ^{2}\left\vert Q\right\vert _{%
\widetilde{\sigma }}$ as required. To estimate the second integral, we first
use (\ref{geometric}) to choose a quasi cube $Q^{\prime }\in \Psi ^{\ast }%
\mathcal{R}^{n}$ such that $Q^{\prime }\subset Q$ and $\mathbf{1}_{Q^{\prime
}}\widetilde{\omega }=\mathbf{1}_{Q}\widetilde{\omega _{i}}$, and in
addition that 
\begin{equation}
\ell \left( Q^{\prime }\right) \leq \widetilde{C_{n}}\ell \left( I\right)
\label{latter}
\end{equation}%
where $I=Q\cap \left\{ x_{1}-axis\right\} $. This latter condition (\ref%
{latter}) simply means that $Q^{\prime }$ is taken essentially as small as
possible so that $\mathbf{1}_{Q^{\prime }}\widetilde{\omega }=\mathbf{1}_{Q}%
\widetilde{\omega _{i}}$. Then we write%
\begin{eqnarray*}
\int_{Q}\left\vert \mathbf{R}_{\Psi }^{\alpha ,n}\mathbf{1}_{Q}\widetilde{%
\omega _{i}}\right\vert ^{2}d\widetilde{\sigma } &=&\int_{Q}\left\vert 
\mathbf{R}_{\Psi }^{\alpha ,n}\mathbf{1}_{Q^{\prime }}\widetilde{\omega }%
\right\vert ^{2}d\widetilde{\sigma } \\
&=&\int_{Q\setminus 3Q^{\prime }}\left\vert \mathbf{R}_{\Psi }^{\alpha ,n}%
\mathbf{1}_{Q^{\prime }}\widetilde{\omega }\right\vert ^{2}d\widetilde{%
\sigma }+\int_{Q\cap 3Q^{\prime }}\left\vert \mathbf{R}_{\Psi }^{\alpha ,n}%
\mathbf{1}_{Q^{\prime }}\widetilde{\omega }\right\vert ^{2}d\widetilde{%
\sigma } \\
&=&I+II.
\end{eqnarray*}%
Now $I\lesssim \mathcal{A}_{2}^{\alpha ,\func{dual}}\left\vert Q^{\prime
}\right\vert _{\widetilde{\omega }}=\mathcal{A}_{2}^{\alpha ,\func{dual}%
}\left\vert Q\right\vert _{\widetilde{\omega _{i}}}$ by a standard
calculation similar to that in (\ref{separated}), and 
\begin{equation*}
II\leq \int_{3Q^{\prime }}\left\vert \mathbf{R}_{\Psi }^{\alpha ,n}\mathbf{1}%
_{Q^{\prime }}\widetilde{\omega }\right\vert ^{2}d\widetilde{\sigma }\leq
\left( \mathfrak{T}_{\mathbf{R}_{\Psi }^{\alpha ,n}}^{\Psi _{i}^{\ast }%
\mathcal{R}^{n},\limfunc{triple},\func{dual}}\right) ^{2}\left\vert
Q^{\prime }\right\vert _{\widetilde{\omega }}=\left( \mathfrak{T}_{\mathbf{R}%
_{\Psi }^{\alpha ,n}}^{\Psi _{i}^{\ast }\mathcal{R}^{n},\limfunc{triple},%
\func{dual}}\right) ^{2}\left\vert Q\right\vert _{\widetilde{\omega _{i}}},
\end{equation*}%
by the \emph{local} backward triple quasicube testing condition for the pair 
$\left( \widetilde{\sigma },\widetilde{\omega }\right) $, whose necessity
was proved in Theorem \ref{main} above with the measure $\widetilde{\widehat{%
\omega _{i}}}$, the restriction of $\widetilde{\omega }$ to $\widehat{I_{i}}$%
, which is compactly supported on the real axis. By \emph{local} backward
triple quasicube testing here we mean that we are restricting attention to
those triples $3Q^{\prime }$ such that $3Q^{\prime }\cap \left\{
x_{1}-axis\right\} \subset \left[ -\zeta _{i},\zeta _{i}\right) $, the image
of the larger piece $\widehat{\mathcal{L}_{i}}$ under $\psi _{i}^{-1}$. This
restriction is necessary since our arguments for necessity of triple testing
require support on a line. Now we use condition (\ref{latter}) in our choice
of quasicube $Q^{\prime }$. Indeed, with this choice we then have $%
3Q^{\prime }\cap \left\{ x_{1}-axis\right\} \subset \left[ -\zeta _{i},\zeta
_{i}\right) $, and so have the backward tripled quasicube testing condition
at our disposal.\\*[0.25in]

\textbf{Step 3}: Now we use Theorem \ref{main} again to obtain the
quasienergy conditions for the conformal fractional Riesz transform $\mathbf{%
R}_{\Psi }^{\alpha ,n}$ for each pair $\left( \widetilde{\sigma },\widetilde{%
\omega _{i}}\right) $. Thus, \emph{assuming only the }$\mathfrak{A}%
_{2}^{\alpha }$\emph{\ conditions, and }$\mathcal{Q}^{n}$\emph{\ cube
testing conditions for the weight pair }$\left( \sigma ,\omega \right) $%
\emph{\ for the }$\alpha $\emph{-fractional Riesz transform }$\mathbf{R}%
^{\alpha ,n}$, we have established that the weight pair $\left( \widetilde{%
\sigma },\widetilde{\omega _{i}}\right) $ satisfies the $\mathfrak{A}%
_{2}^{\alpha }$ conditions, the quasienergy conditions, the quasitesting
conditions and the quasiweak boundedness property (which follows from the
backward triple quasitesting condition) all for the conformal $\alpha $%
-fractional Riesz transform $\mathbf{R}_{\Psi }^{\alpha ,n}$. Now Theorem %
\ref{T1 theorem} for conformal Riesz transforms (Conclusion \ref{con thm})
and parts (2) and (3) of Proposition \ref{change of variable} apply to show
that%
\begin{equation*}
\mathfrak{N}_{\mathbf{R}_{\Psi }^{\alpha ,n}}\left( \widetilde{\sigma },%
\widetilde{\omega _{i}}\right) \lesssim \sqrt{\mathfrak{A}_{2}^{\alpha
}\left( \sigma ,\omega \right) }+\mathfrak{T}_{\mathbf{R}^{\alpha ,n}}\left(
\sigma ,\omega \right) +\mathfrak{T}_{\mathbf{R}^{\alpha ,n}}^{\func{dual}%
}\left( \sigma ,\omega \right)
\end{equation*}%
for each $i$. Then by part (1) of Proposition \ref{change of variable} we
have%
\begin{equation*}
\mathfrak{N}_{\mathbf{R}_{\sigma }^{\alpha ,n}}\left( \sigma ,\omega
_{i}\right) \lesssim \mathfrak{N}_{\mathbf{R}_{\Psi }^{\alpha ,n}}\left( 
\widetilde{\sigma },\widetilde{\omega _{i}}\right) ,
\end{equation*}%
and we have 
\begin{equation*}
\mathfrak{N}_{\mathbf{R}_{\sigma }^{\alpha ,n}}\left( \sigma ,\omega \right)
\leq \sum_{i=1}^{N}\mathfrak{N}_{\mathbf{R}_{\sigma }^{\alpha ,n}}\left(
\sigma ,\omega _{i}\right) .
\end{equation*}%
This completes the proof of Theorem \ref{general T1}.
\end{proof}

\section{Appendix}

Here we state and prove extensions of Theorems \ref{main} and \ref{final}
that hold for measures $\sigma $ and $\omega $ supported in a $\left(
k_{1}+1\right) $-plane and $\left( k_{2}+1\right) $-plane respectively that
intersect in the $x_{1}$-axis at right angles. We begin with the following
extension of Theorem \ref{main} to perpendicular subspaces.

\begin{theorem}
\label{main'}Fix a collection of $\Omega $-quasicubes in $\mathbb{R}^{n}=%
\mathbb{R}\times \mathbb{R}^{k_{1}}\times \mathbb{R}^{k_{2}}$. Let 
\begin{eqnarray*}
S &=&\left\{ \left( x_{1},x^{\prime },0\right) \in \mathbb{R}\times \mathbb{R%
}^{k_{1}}\times \mathbb{R}^{k_{2}}:\left( x_{1},x^{\prime }\right) \in 
\mathbb{R}\times \mathbb{R}^{k_{1}}\right\} , \\
W &=&\left\{ \left( x_{1},0,x^{\prime \prime }\right) \in \mathbb{R}\times 
\mathbb{R}^{k_{1}}\times \mathbb{R}^{k_{2}}:\left( x_{1},x^{\prime \prime
}\right) \in \mathbb{R}\times \mathbb{R}^{k_{2}}\right\} , \\
L &=&S\cap W=\left\{ \left( x_{1},0,0\right) \in \mathbb{R}\times \mathbb{R}%
^{k_{1}}\times \mathbb{R}^{k_{2}}:x_{1}\in \mathbb{R}\right\} ,
\end{eqnarray*}%
be $\left( k_{1}+1\right) $-, $\left( k_{2}+1\right) $- and $1$- dimensional
subspaces respectively of $\mathbb{R}^{n}$. Let $\sigma $ and $\omega $ be
locally finite positive Borel measures supported on $S$ and $W$ respectively
(possibly having common point masses in the intersection $L$ of their
supports). Suppose that $\mathbf{R}_{\Psi }^{\alpha ,n}$ is a conformal $%
\alpha $-fractional Riesz transform with $0\leq \alpha <n$ and graphing
function $\Psi \left( x\right) =x-\left( 0,\psi \left( x_{1}\right) \right) $
that satisfies (\ref{small Lip}), i.e.%
\begin{equation}
\left\Vert D\psi \right\Vert _{\infty }<\frac{1}{8n}\left( 1-\frac{\alpha }{n%
}\right) ,  \label{small Lip'}
\end{equation}%
and consider the tangent line truncations for $\mathbf{R}_{\Psi }^{\alpha
,n} $ in the $\Omega $-quasitesting conditions. Then%
\begin{equation*}
\mathcal{E}_{\alpha }^{\Omega \mathcal{Q}^{n}}\lesssim \sqrt{\mathcal{A}%
_{2}^{\alpha }}+\mathfrak{T}_{\mathbf{R}_{\Psi }^{\alpha ,n}}^{\Omega 
\mathcal{Q}^{n}}\text{ and }\mathcal{E}_{\alpha }^{\Omega \mathcal{Q}^{n},%
\func{dual}}\lesssim \sqrt{\mathcal{A}_{2}^{\alpha ,\func{dual}}}+\mathfrak{T%
}_{\mathbf{R}_{\Psi }^{\alpha ,n}}^{\Omega \mathcal{Q}^{n},\func{dual}}\ .
\end{equation*}%
In addition if $\Omega $ is a $C^{1}$ diffeomorphism that is $L$-transverse,
then%
\begin{equation*}
\mathcal{WBP}_{\mathbf{R}_{\Psi }^{\alpha ,n}}^{\Omega \mathcal{Q}%
^{n}}\lesssim \mathfrak{T}_{\mathbf{R}_{\Psi }^{\alpha ,n}}^{\Omega \mathcal{%
Q}^{n},\limfunc{triple},\func{dual}}\lesssim \mathfrak{T}_{\mathbf{R}_{\Psi
}^{\alpha ,n}}^{\Omega \mathcal{Q}^{n},\func{dual}}+\sqrt{\mathcal{A}%
_{2}^{\alpha }}+\sqrt{\mathcal{A}_{2}^{\alpha ,\func{dual}}}\ .
\end{equation*}
\end{theorem}

\begin{proof}
In our current situation, the assumptions on $\sigma $ and $\omega $ are
symmetric so that it is enough to prove just that the forward quasienergy
condition $\mathcal{E}_{\alpha }^{\Omega \mathcal{Q}^{n}}$ is bounded by a
constant multiple of $\mathfrak{T}_{\mathbf{R}_{\Psi }^{\alpha ,n}}^{\Omega 
\mathcal{Q}^{n}}+\sqrt{\mathcal{A}_{2}^{\alpha }}$, where $\mathcal{A}%
_{2}^{\alpha }$ is the Muckenhoupt condition with holes. We must show%
\begin{equation*}
\sup_{\ell \geq 0}\sum_{r=1}^{\infty }\sum_{J\in \mathcal{M}_{\limfunc{deep}%
}^{\ell }\left( I_{r}\right) }\left( \frac{\mathrm{P}^{\alpha }\left( J,%
\mathbf{1}_{I\setminus J^{\ast }}\sigma \right) }{\left\vert J\right\vert ^{%
\frac{1}{n}}}\right) ^{2}\left\Vert \mathsf{P}_{J}^{\omega }\mathbf{x}%
\right\Vert _{L^{2}\left( \omega \right) }^{2}\lesssim \left( \left( 
\mathfrak{T}_{\mathbf{R}_{\Psi }^{\alpha ,n}}^{\Omega \mathcal{Q}%
^{n}}\right) ^{2}+\mathcal{A}_{2}^{\alpha }\right) \left\vert I\right\vert
_{\sigma }\ ,
\end{equation*}%
for all partitions of a dyadic quasicube $I=\overset{\cdot }{%
\dbigcup\limits_{r\geq 1}}I_{r}$ into dyadic subquasicubes $I_{r}$. We again
fix $\ell \geq 0$ and suppress both $\ell $ and $\mathbf{r}$ in the notation 
$\mathcal{M}_{\limfunc{deep}}\left( I_{r}\right) =\mathcal{M}_{\mathbf{r}-%
\limfunc{deep}}^{\ell }\left( I_{r}\right) $. We may assume that all the
quasicubes $J$ intersect $\limfunc{supp}\omega $, since otherwise $%
\left\Vert \mathsf{P}_{J}^{\omega }\mathbf{x}\right\Vert _{L^{2}\left(
\omega \right) }^{2}=0$, hence that all the quasicubes $I_{r}$ and $J$
intersect $W$, which contains $\limfunc{supp}\omega $. Thus what we must
show is%
\begin{equation}
\sum_{r=1}^{\infty }\sum_{J\in \mathcal{M}_{\limfunc{deep}}\left(
I_{r}\right) }\left( \frac{\mathrm{P}^{\alpha }\left( J,\mathbf{1}%
_{I\setminus J^{\ast }}\sigma \right) }{\left\vert J\right\vert ^{\frac{1}{n}%
}}\right) ^{2}\left\Vert \mathsf{P}_{J}^{\omega }\mathbf{x}\right\Vert
_{L^{2}\left( \omega \right) }^{2}\lesssim \left( \left( \mathfrak{T}_{%
\mathbf{R}_{\Psi }^{\alpha ,n}}^{\Omega \mathcal{Q}^{n}}\right) ^{2}+%
\mathcal{A}_{2}^{\alpha }\right) \left\vert I\right\vert _{\sigma }\ ,
\label{must show}
\end{equation}%
where since $\omega $ is supported in $W$,%
\begin{equation}
\left\Vert \mathsf{P}_{J}^{\omega }\mathbf{x}\right\Vert _{L^{2}\left(
\omega \right) }^{2}=\left\Vert \mathsf{P}_{J}^{\omega }x^{1}\right\Vert
_{L^{2}\left( \omega \right) }^{2}+\sum_{j=k_{1}+2}^{n}\left\Vert \mathsf{P}%
_{J}^{\omega }x^{j}\right\Vert _{L^{2}\left( \omega \right) }^{2}\ .
\label{W energy}
\end{equation}

Let 
\begin{equation*}
\mathcal{C}\left( I\right) \equiv \left\{ J\in \overset{\cdot }{%
\dbigcup\limits_{r\geq 1}}\mathcal{M}_{\limfunc{deep}}\left( I_{r}\right)
:J\cap W\neq \emptyset \right\}
\end{equation*}%
be the collection of all quasicubes $J$ arising in (\ref{must show}). We
divide the quasicubes $J$ in $\mathcal{C}\left( I\right) $ into two separate
collections $\mathcal{C}_{i}\left( I\right) $, $1\leq i\leq 2$, according to
how close $J$ is to the plane $S$ containing the support of $\sigma $:%
\begin{eqnarray*}
\mathcal{C}_{1}\left( I\right) &\equiv &\left\{ J\in \mathcal{C}\left(
I\right) :3J\cap S=\emptyset \right\} , \\
\mathcal{C}_{2}\left( I\right) &\equiv &\left\{ J\in \mathcal{C}\left(
I\right) :3J\cap S\neq \emptyset \right\} .
\end{eqnarray*}%
For the first collection $\mathcal{C}_{1}\left( I\right) $ we estimate the
corresponding sum in (\ref{must show}) using weak reversal of energy.
Indeed, for $J\in \mathcal{C}_{1}\left( I\right) $ we use the argument in (%
\ref{weak 1}) and (\ref{weak 2}), but with $\sigma $ and $\omega $
interchanged, to see that for any $x\in J$,%
\begin{eqnarray*}
&&\left( \frac{\mathrm{P}^{\alpha }\left( J,\mathbf{1}_{I\setminus J^{\ast
}}\sigma \right) }{\left\vert J\right\vert ^{\frac{1}{n}}}\right)
^{2}\left\Vert \mathsf{P}_{J}^{\omega }\mathbf{x}\right\Vert _{L^{2}\left(
\omega \right) }^{2} \\
&\lesssim &\left( \frac{\mathrm{P}^{\alpha }\left( J,\mathbf{1}_{I\setminus
J^{\ast }}\sigma \right) }{\left\vert J\right\vert ^{\frac{1}{n}}}\right)
^{2}\left\vert J\right\vert ^{\frac{2}{n}}\left\vert J\right\vert _{\omega }=%
\mathrm{P}^{\alpha }\left( J,\mathbf{1}_{I\setminus J^{\ast }}\sigma \right)
^{2}\left\vert J\right\vert _{\omega }\leq \mathrm{P}^{\alpha }\left( J,%
\mathbf{1}_{I}\sigma \right) ^{2}\left\vert J\right\vert _{\omega } \\
&\lesssim &\left\vert \mathbf{R}_{\Psi }^{\alpha ,n}\left( \mathbf{1}%
_{I}\sigma \right) \left( x\right) \right\vert ^{2}\left\vert J\right\vert
_{\omega }\ ,
\end{eqnarray*}%
and so using that the $J\in \mathcal{C}\left( I\right) $ are pairwise
disjoint and contained in $I$, we get%
\begin{eqnarray*}
&&\sum_{J\in \mathcal{C}_{1}\left( I\right) }\left( \frac{\mathrm{P}^{\alpha
}\left( J,\mathbf{1}_{I\setminus J^{\ast }}\sigma \right) }{\left\vert
J\right\vert ^{\frac{1}{n}}}\right) ^{2}\left\Vert \mathsf{P}_{J}^{\omega }%
\mathbf{x}\right\Vert _{L^{2}\left( \omega \right) }^{2}\lesssim \sum_{J\in 
\mathcal{C}_{1}\left( I\right) }\inf_{x\in J}\ \left\vert \mathbf{R}_{\Psi
}^{\alpha ,n}\left( \mathbf{1}_{I}\sigma \right) \left( x\right) \right\vert
^{2}\ \left\vert J\right\vert _{\omega } \\
&&\ \ \ \ \ \ \ \ \ \ \ \ \ \ \ \ \ \ \ \ \ \leq \int_{I}\left\vert \mathbf{R%
}_{\Psi }^{\alpha ,n}\left( \mathbf{1}_{I}\sigma \right) \left( x\right)
\right\vert ^{2}d\omega \left( x\right) \leq \left( \mathfrak{T}_{\mathbf{R}%
_{\Psi }^{\alpha ,n}}^{\Omega \mathcal{Q}^{n}}\right) ^{2}\left\vert
I\right\vert _{\sigma }\ .
\end{eqnarray*}

Now we consider the quasicubes $J$ in the second collection $\mathcal{C}%
_{2}\left( I\right) $. In this case we apply the arguments used above for
the forward quasienergy condition in Subsection \ref{Subsec forward}. Given $%
J\in \mathcal{C}_{2}\left( I\right) $ we consider the decomposition 
\begin{equation*}
I\setminus J^{\ast }=\text{\textsc{E}}\left( J^{\ast }\right) \dot{\cup}%
\text{\textsc{S}}\left( J^{\ast }\right)
\end{equation*}%
of $I\setminus J^{\ast }$ into \emph{end} \textsc{E}$\left( J^{\ast }\right) 
$ and \emph{side} \textsc{S}$\left( J^{\ast }\right) $ disjoint pieces
defined by%
\begin{eqnarray*}
\text{\textsc{E}}\left( J^{\ast }\right) &\equiv &\left( I\setminus J^{\ast
}\right) \cap \left\{ \left( y^{1},y^{\prime \prime \prime }\right)
:\left\vert y^{\prime \prime \prime }-c_{J}^{\prime \prime \prime
}\right\vert \leq \frac{10}{\gamma }\left\vert y^{1}-c_{J}^{1}\right\vert
\right\} ; \\
\text{\textsc{S}}\left( J^{\ast }\right) &\equiv &\left( I\setminus J^{\ast
}\right) \setminus \text{\textsc{E}}\left( J^{\ast }\right) \ ,
\end{eqnarray*}%
where we write $y=\left( y^{1},y^{\prime },y^{\prime \prime }\right) =\left(
y^{1},y^{\prime \prime \prime }\right) $ with $y^{\prime \prime \prime
}=\left( y^{\prime },y^{\prime \prime }\right) \in \mathbb{R}^{k_{1}+k_{2}}$%
. Then by (\ref{W energy}) it suffices to show 
\begin{eqnarray*}
A^{j} &\equiv &\sum_{J\in \mathcal{C}_{2}\left( I\right) }\left( \frac{%
\mathrm{P}^{\alpha }\left( J,\mathbf{1}_{\text{\textsc{E}}\left( J^{\ast
}\right) }\sigma \right) }{\left\vert J\right\vert ^{\frac{1}{n}}}\right)
^{2}\left\Vert \mathsf{P}_{J}^{\omega }x^{j}\right\Vert _{L^{2}\left( \omega
\right) }^{2}\lesssim \left( \left( \mathfrak{T}_{\mathbf{R}_{\Psi }^{\alpha
,n}}^{\Omega \mathcal{Q}^{n}}\right) ^{2}+\mathcal{A}_{2}^{\alpha }\right)
\left\vert I\right\vert _{\sigma }\ , \\
B &\equiv &\sum_{J\in \mathcal{C}_{2}\left( I\right) }\left( \frac{\mathrm{P}%
^{\alpha }\left( J,\mathbf{1}_{\text{\textsc{S}}\left( J^{\ast }\right)
}\sigma \right) }{\left\vert J\right\vert ^{\frac{1}{n}}}\right)
^{2}\left\Vert \mathsf{P}_{J}^{\omega }\mathbf{x}\right\Vert _{L^{2}\left(
\omega \right) }^{2}\lesssim \left( \left( \mathfrak{T}_{\mathbf{R}_{\Psi
}^{\alpha ,n}}^{\Omega \mathcal{Q}^{n}}\right) ^{2}+\mathcal{A}_{2}^{\alpha
}\right) \left\vert I\right\vert _{\sigma }\ ,
\end{eqnarray*}%
for $j=1$ and $k_{1}+2\leq j\leq n$.

We estimate $A^{1}$ involving the ends \textsc{E}$\left( J^{\ast }\right) $
as before, obtaining first a strong reversal of the $x^{1}$-energy $%
\left\Vert \mathsf{P}_{J}^{\omega }x^{1}\right\Vert _{L^{2}\left( \omega
\right) }^{2}$, followed by an `$\mathcal{A}_{2}^{\alpha }$ reversal' of the
other partial energies $\left\Vert \mathsf{P}_{J}^{\omega }x^{j}\right\Vert
_{L^{2}\left( \omega \right) }^{2}$ for $k_{1}+2\leq j\leq n$. In
particular, the strong reversal of $x^{1}$-energy that we obtain here is
analogous to that appearing in (\ref{diff quotient}), and to that appearing
in (\ref{other reversal}) with $\sigma $ and $\omega $ interchanged, and it
then delivers the following estimate analogous to (\ref{A123}),%
\begin{eqnarray*}
&&A^{1}\equiv \sum_{J\in \mathcal{C}_{2}\left( I\right) }\left( \frac{%
\mathrm{P}^{\alpha }\left( J,\mathbf{1}_{\text{\textsc{E}}\left( J^{\ast
}\right) }\sigma \right) }{\left\vert J\right\vert ^{\frac{1}{n}}}\right)
^{2}\int_{J\cap L}\left\vert x^{1}-\mathbb{E}_{J}^{\omega }x^{1}\right\vert
^{2}d\omega \left( x\right) \\
&\lesssim &\sum_{J\in \mathcal{C}_{2}\left( I\right) }\frac{1}{\left\vert
J\right\vert _{\omega }}\int_{J\cap L}\int_{J\cap L}\left\{ \left( \mathbf{R}%
_{\Psi }^{\alpha ,n}\right) _{1}\left( \mathbf{1}_{I}\sigma \right) \left(
x^{1},0^{\prime },x^{\prime \prime }\right) -\left( \mathbf{R}_{\Psi
}^{\alpha ,n}\right) _{1}\left( \mathbf{1}_{I}\sigma \right) \left(
z^{1},0^{\prime },z^{\prime \prime }\right) \right\} ^{2}d\omega \left(
x\right) d\omega \left( z\right) \\
&&+\sum_{J\in \mathcal{C}_{2}\left( I\right) }\frac{1}{\left\vert
J\right\vert _{\omega }}\int_{J\cap L}\int_{J\cap L}\left\{ \left( \mathbf{R}%
_{\Psi }^{\alpha ,n}\right) _{1}\left( \mathbf{1}_{J^{\ast }}\sigma \right)
\left( x^{1},0^{\prime },x^{\prime \prime }\right) -\left( \mathbf{R}_{\Psi
}^{\alpha ,n}\right) _{1}\left( \mathbf{1}_{J^{\ast }}\sigma \right) \left(
z^{1},0^{\prime },z^{\prime \prime }\right) \right\} ^{2}d\omega \left(
x\right) d\omega \left( z\right) \\
&&+\sum_{J\in \mathcal{C}_{2}\left( I\right) }\frac{1}{\left\vert
J\right\vert _{\omega }}\int_{J\cap L}\int_{J\cap L}\left\{ \left( \mathbf{R}%
_{\Psi }^{\alpha ,n}\right) _{1}\left( \mathbf{1}_{\text{\textsc{S}}\left(
J^{\ast }\right) }\sigma \right) \left( x^{1},0^{\prime },x^{\prime \prime
}\right) -\left( \mathbf{R}_{\Psi }^{\alpha ,n}\right) _{1}\left( \mathbf{1}%
_{\text{\textsc{S}}\left( J^{\ast }\right) }\sigma \right) \left(
z^{1},0^{\prime },z^{\prime \prime }\right) \right\} ^{2}d\omega \left(
x\right) d\omega \left( z\right) \\
&\equiv &A_{1}^{1}+A_{2}^{1}+A_{3}^{1},
\end{eqnarray*}%
where we have used the `paraproduct trick' $I=J^{\ast }\dot{\cup}\left(
I\setminus J^{\ast }\right) =J^{\ast }\dot{\cup}$\textsc{E}$\left( J^{\ast
}\right) \dot{\cup}$\textsc{S}$\left( J^{\ast }\right) $.

Now we can discard the differences in both of the terms $A_{1}^{1}$ and $%
A_{2}^{1}$ and use pairwise disjointedness of $J$ and bounded overlap of $%
J^{\ast }$ to control each of $A_{1}^{1}$ and $A_{2}^{1}$ by $\left( 
\mathfrak{T}_{\left( \mathbf{R}_{\Psi }^{\alpha ,n}\right) _{1}}^{\Omega 
\mathcal{Q}^{n}}\right) ^{2}\left\vert I\right\vert _{\sigma }$. Just as in
the previous argument, term $A_{3}^{1}$ is dominated by term $B$.

Now for $k_{1}+2\leq j\leq n$ we again use the `paraproduct trick' $%
I\setminus J^{\ast }=$\textsc{E}$\left( J^{\ast }\right) \dot{\cup}$\textsc{S%
}$\left( J^{\ast }\right) $ to dominate%
\begin{equation*}
A^{j}\equiv \sum_{J\in \mathcal{C}_{2}\left( I\right) }\left( \frac{\mathrm{P%
}^{\alpha }\left( J,\mathbf{1}_{\text{\textsc{E}}\left( J^{\ast }\right)
}\sigma \right) }{\left\vert J\right\vert ^{\frac{1}{n}}}\right)
^{2}\int_{J\cap L}\left\vert x^{j}-\mathbb{E}_{J}^{\omega }x^{j}\right\vert
^{2}d\omega \left( x\right)
\end{equation*}%
by%
\begin{eqnarray*}
&&\sum_{J\in \mathcal{C}_{2}\left( I\right) }\left( \frac{\mathrm{P}^{\alpha
}\left( J,\mathbf{1}_{I\setminus J^{\ast }}\sigma \right) }{\left\vert
J\right\vert ^{\frac{1}{n}}}\right) ^{2}\int_{J\cap L}\left\vert x^{j}-%
\mathbb{E}_{J}^{\omega }x^{j}\right\vert ^{2}d\omega \left( x\right) \\
&&+\sum_{J\in \mathcal{C}_{2}\left( I\right) }\left( \frac{\mathrm{P}%
^{\alpha }\left( J,\mathbf{1}_{\text{\textsc{S}}\left( J^{\ast }\right)
}\sigma \right) }{\left\vert J\right\vert ^{\frac{1}{n}}}\right)
^{2}\int_{J\cap L}\left\vert x^{j}-\mathbb{E}_{J}^{\omega }x^{j}\right\vert
^{2}d\omega \left( x\right) \\
&\equiv &A_{1}^{j}+A_{2}^{j},
\end{eqnarray*}%
where, since the directions of $x^{j}$ are perpendicular to the support of $%
\sigma $ for $k_{1}+2\leq j\leq n$, the term $A_{1}^{j}$ is treated using
`weak energy reversal'. Namely, for $x\in J\cap L$, we have as in (\ref{weak
1}) that,%
\begin{eqnarray*}
\frac{\left\vert x_{j}\right\vert }{\left\vert J\right\vert ^{\frac{1}{n}}}%
\mathrm{P}^{\alpha }\left( J,\mathbf{1}_{I\setminus J^{\ast }}\sigma \right)
&\lesssim &\frac{\left\vert x_{j}\right\vert }{\left\vert J\right\vert ^{%
\frac{1}{n}}}\int_{I\setminus J^{\ast }}\frac{\left\vert J\right\vert ^{%
\frac{1}{n}}}{\left( \left\vert J\right\vert ^{\frac{1}{n}}+\left\vert
y-x\right\vert \right) ^{n+1-\alpha }}d\sigma \left( y\right) \\
&\approx &\int_{I\setminus J^{\ast }}\frac{\left\vert x_{j}\right\vert }{%
\left\vert \Psi \left( y\right) -\Psi \left( x\right) \right\vert
^{n+1-\alpha }}d\sigma \left( y\right) \\
&\approx &\left\vert \left( \mathbf{R}_{\Psi }^{\alpha ,n}\right) _{j}\left( 
\mathbf{1}_{I\setminus J^{\ast }}\sigma \right) \left( x\right) \right\vert
\lesssim \left\vert \mathbf{R}_{\Psi }^{\alpha ,n}\left( \mathbf{1}%
_{I}\sigma \right) \left( x\right) \right\vert +\left\vert \mathbf{R}_{\Psi
}^{\alpha ,n}\left( \mathbf{1}_{J^{\ast }}\sigma \right) \left( x\right)
\right\vert ,
\end{eqnarray*}
which gives%
\begin{eqnarray*}
A_{1}^{j} &=&\sum_{J\in \mathcal{C}_{2}\left( I\right) }\left( \frac{\mathrm{%
P}^{\alpha }\left( J,\mathbf{1}_{I\setminus J^{\ast }}\sigma \right) }{%
\left\vert J\right\vert ^{\frac{1}{n}}}\right) ^{2}\int_{J\cap L}\left\vert
x^{j}-\mathbb{E}_{J}^{\omega }x^{j}\right\vert ^{2}d\omega \left( x\right) \\
&\leq &\sum_{J\in \mathcal{C}_{2}\left( I\right) }\left( \frac{\mathrm{P}%
^{\alpha }\left( J,\mathbf{1}_{I\setminus J^{\ast }}\sigma \right) }{%
\left\vert J\right\vert ^{\frac{1}{n}}}\right) ^{2}\int_{J\cap L}\left\vert
x^{j}\right\vert ^{2}d\omega \left( x\right) \\
&=&\sum_{J\in \mathcal{C}_{2}\left( I\right) }\int_{J\cap L}\left( \frac{%
\left\vert x^{j}\right\vert }{\left\vert J\right\vert ^{\frac{1}{n}}}\mathrm{%
P}^{\alpha }\left( J,\mathbf{1}_{I\setminus J^{\ast }}\sigma \right) \right)
^{2}d\omega \left( x\right) \\
&\lesssim &\sum_{J\in \mathcal{C}_{2}\left( I\right) }\int_{J\cap L}\left(
\left\vert \mathbf{R}_{\Psi }^{\alpha ,n}\left( \mathbf{1}_{I}\sigma \right)
\right\vert ^{2}+\left\vert \mathbf{R}_{\Psi }^{\alpha ,n}\left( \mathbf{1}%
_{J^{\ast }}\sigma \right) \right\vert ^{2}\right) d\omega \left( x\right) ,
\end{eqnarray*}%
and we now proceed as in (\ref{forward A1}) and (\ref{forward A2}). Then we
dominate the second term $A_{2}^{j}$ by term $B$.

Now we treat term $B$ using the `Carleson shadow method' that was used to
treat term $B$ in (\ref{B1 + B2}) in Subsection \ref{Subsec forward}. We
first assume that $n-1\leq \alpha <n$ so that $\mathrm{P}^{\alpha }\left( J,%
\mathbf{1}_{\text{\textsc{S}}\left( J^{\ast }\right) }\sigma \right) \leq 
\mathcal{P}^{\alpha }\left( J,\mathbf{1}_{\text{\textsc{S}}\left( J^{\ast
}\right) }\sigma \right) $, and then use $\left\Vert \mathsf{P}_{J}^{\omega }%
\mathbf{x}\right\Vert _{L^{2}\left( \omega \right) }^{2}\lesssim \left\vert
J\right\vert ^{\frac{2}{n}}\left\vert J\right\vert _{\omega }$ and apply the 
$\mathcal{A}_{2}^{\alpha }$ condition with holes to obtain the following `$%
\mathcal{A}_{2}^{\alpha }$ reversal' of quasienergy,%
\begin{equation*}
B\lesssim \mathcal{A}_{2}^{\alpha }\int_{I}\left\{ \sum_{J\in \mathcal{C}%
_{2}\left( I\right) }\left( \frac{\left\vert J\right\vert ^{\frac{1}{n}}}{%
\left\vert J\right\vert ^{\frac{1}{n}}+\left\vert y-c_{J}\right\vert }%
\right) ^{n+1-\alpha }\mathbf{1}_{\text{\textsc{S}}\left( J^{\ast }\right)
}\left( y\right) \right\} d\sigma \left( y\right) \equiv \mathcal{A}%
_{2}^{\alpha }\int_{I}F\left( y\right) d\sigma \left( y\right) .
\end{equation*}

At this point we claim as above that $F\left( y\right) \leq C$ with a
constant $C$ independent of the decomposition $\mathcal{C}_{2}\left(
I\right) $. For this we now define $\limfunc{Sh}\left( y;\gamma \right) $ to
be the Carleson shadow of the point $y$ onto the $x_{1}$-axis $L$ with sides
of slope $\frac{10}{\gamma }$, i.e. $\limfunc{Sh}\left( y;\gamma \right) $
is the interval on $L$ with length $\frac{1}{5}\gamma \limfunc{dist}\left(
y,L\right) $ and center equal to the point on $W$ that is closest to $y$. We
then proceed as above to the following variant of a previous estimate, where
we here redefine $\mathcal{J}\equiv 3J\cap L$ for $J\in \mathcal{C}%
_{2}\left( I\right) $:%
\begin{eqnarray*}
&&F\left( y\right) =\sum_{\substack{ J\in \mathcal{C}_{2}\left( I\right)  \\ 
\mathcal{J}\subset C\limfunc{Sh}\left( y;\gamma \right) }}\left( \frac{%
\left\vert J\right\vert ^{\frac{1}{n}}}{\left\vert J\right\vert ^{\frac{1}{n}%
}+\left\vert y-c_{J}\right\vert }\right) ^{n+1-\alpha }\mathbf{1}_{\text{%
\textsc{S}}\left( J^{\ast }\right) }\left( y\right) \\
&\lesssim &\sum_{r=1}^{\infty }\sum_{\substack{ J\in \mathcal{M}_{\mathbf{r}-%
\limfunc{deep}}\left( I_{r}\right)  \\ \emptyset \neq \mathcal{J}\subset C%
\limfunc{Sh}\left( y;\gamma \right) }}\left( \frac{\left\vert J\right\vert ^{%
\frac{1}{n}}}{\left\vert y-c_{J}\right\vert }\right) ^{n-\alpha }\frac{%
\left\vert J\right\vert ^{\frac{1}{n}}}{\limfunc{dist}\left( y,L\right) }%
\mathbf{1}_{\text{\textsc{S}}\left( J^{\ast }\right) }\left( y\right) ,
\end{eqnarray*}%
and then following the argument for (\ref{continue}), we can dominate this
by 
\begin{eqnarray*}
\frac{1}{\limfunc{dist}\left( y,L\right) }\sum_{r=1}^{\infty }\left\{ \sum 
_{\substack{ J\in \mathcal{M}_{\mathbf{r}-\limfunc{deep}}\left( I_{r}\right) 
\\ \emptyset \neq \mathcal{J}\subset C\limfunc{Sh}\left( y;\gamma \right) }}%
\left\vert J\right\vert ^{\frac{1}{n}}\right\} &\lesssim &\frac{1}{\limfunc{%
dist}\left( y,L\right) }\sum_{r=1}^{\infty }\beta \left\vert I_{r}\cap
C^{\prime }\limfunc{Sh}\left( y;\gamma \right) \right\vert \\
&\lesssim &\beta \frac{1}{\limfunc{dist}\left( y,L\right) }\left\vert
C^{\prime }\limfunc{Sh}\left( y;\gamma \right) \right\vert \lesssim \beta
\gamma ,
\end{eqnarray*}%
since the quasicubes $J^{\ast }$ have overlap bounded by $\beta $ (so that
we can essentially treat the shadows as being pairwise disjoint). This
completes the proof that%
\begin{equation*}
B\lesssim \mathcal{A}_{2}^{\alpha }\int_{I}F\left( y\right) d\sigma \left(
y\right) \lesssim \mathcal{A}_{2}^{\alpha }\left\vert I\right\vert _{\sigma
}\ .
\end{equation*}%
Finally, note that the case $0\leq \alpha <n-1$ is handled using the
Cauchy-Schwartz inequality as in (\ref{small alpha}).

\bigskip

\textbf{Remark}: Our decomposition into end and side pieces here uses the
line $L$ as the means of definition, rather than the possibly larger
subspace $W$, in order to exploit one-dimensional reversal of energy. Of
course in the argument for the forward quasienergy condition in Subsection %
\ref{Subsec forward}, the spaces $W$ and $L$ coincide.

\bigskip

Finally, we prove the estimate for the tripled testing condition $\mathfrak{T%
}_{\mathbf{R}_{\Psi }^{\alpha ,n}}^{\Omega \mathcal{Q}^{n},\limfunc{triple},%
\func{dual}}$ by exactly the same method as used before in Subsection \ref%
{Subsec triple}. Indeed, with notation analogous to that in Subsection \ref%
{Subsec triple}, we take absolute values inside the fractional singular
integral,%
\begin{equation*}
\int_{Q}\left\vert \mathbf{R}_{\Psi }^{\alpha ,n}\left( 1_{Q^{\prime
}}\omega \right) \right\vert ^{2}d\sigma \lesssim \int_{Q}\left\{
\int_{Q^{\prime }}\left\vert y-x\right\vert ^{\alpha -n}d\omega \left(
x\right) \right\} ^{2}d\sigma \left( y\right) ,
\end{equation*}%
and then decompose the two perpendicular sets $Q^{\prime }\cap W$ and $Q\cap
S$ in annuli away from their point of intersection $P\equiv Q^{\prime }\cap
Q\cap L$. Then using that $\Omega $ is a $C^{1}$ diffeomorphism and $L$%
-transverse, and that $Q$ and $Q^{\prime }$ are neighbouring $\Omega $%
-quasicubes, the Hardy operator applies just as before in Subsection \ref%
{Subsec triple}.
\end{proof}

It is now an easy matter to obtain from Theorems \ref{main'} and \ref{T1
theorem}\ the following $T1$ theorem that generalizes Theorem \ref{final}.

\begin{theorem}
Let 
\begin{eqnarray*}
S &=&\left\{ \left( x_{1},x^{\prime },0\right) \in \mathbb{R}\times \mathbb{R%
}^{k_{1}}\times \mathbb{R}^{k_{2}}:\left( x_{1},x^{\prime }\right) \in 
\mathbb{R}\times \mathbb{R}^{k_{1}}\right\} , \\
W &=&\left\{ \left( x_{1},0,x^{\prime \prime }\right) \in \mathbb{R}\times 
\mathbb{R}^{k_{1}}\times \mathbb{R}^{k_{2}}:\left( x_{1},x^{\prime \prime
}\right) \in \mathbb{R}\times \mathbb{R}^{k_{2}}\right\} , \\
L &=&S\cap W=\left\{ \left( x_{1},0,0\right) \in \mathbb{R}\times \mathbb{R}%
^{k_{1}}\times \mathbb{R}^{k_{2}}:x_{1}\in \mathbb{R}\right\} ,
\end{eqnarray*}%
be $\left( k_{1}+1\right) $-, $\left( k_{2}+1\right) $- and $1$- dimensional
subspaces respectively of $\mathbb{R}^{n}=\mathbb{R}\times \mathbb{R}%
^{k_{1}}\times \mathbb{R}^{k_{2}}$. Let $\sigma $ and $\omega $ be locally
finite positive Borel measures supported on $S$ and $W$ respectively
(possibly having common point masses in the intersection $L$ of their
supports). Suppose that $\Omega $ is a $C^{1}$ diffeomorphism and $L$%
-transverse. Suppose also that $\mathbf{R}_{\Psi }^{\alpha ,n}$ is a
conformal fractional Riesz transform with $0\leq \alpha <n$, where $\Psi $
is a $C^{1,\delta }$ diffeomorphism given by $\Psi \left( x\right) =x-\left(
0,\psi \left( x_{1}\right) \right) $ where $\psi $ satisfies (\ref{small Lip}%
). Set $\left( \mathbf{R}_{\Psi }^{\alpha ,n}\right) _{\sigma }f=\mathbf{R}%
_{\Psi }^{\alpha ,n}\left( f\sigma \right) $ for any smooth truncation of $%
\mathbf{R}_{\Psi }^{\alpha ,n}$. Then the operator norm $\mathfrak{N}_{%
\mathbf{R}_{\Psi }^{\alpha ,n}}$ of $\left( \mathbf{R}_{\Psi }^{\alpha
,n}\right) _{\sigma }$ as an operator from\thinspace $L^{2}\left( \sigma
\right) $ to $L^{2}\left( \omega \right) $, uniformly in smooth truncations,
satisfies%
\begin{equation*}
\mathfrak{N}_{\mathbf{R}_{\Psi }^{\alpha ,n}}\approx C_{\alpha }\left( \sqrt{%
\mathfrak{A}_{2}^{\alpha }}+\mathfrak{T}_{\mathbf{R}_{\Psi }^{\alpha
,n}}^{\Omega \mathcal{Q}^{n}}+\mathfrak{T}_{\mathbf{R}_{\Psi }^{\alpha
,n}}^{\Omega \mathcal{Q}^{n},\func{dual}}\right) .
\end{equation*}
\end{theorem}

\begin{remark}
The above theorem generalizes Theorem \ref{final} by permitting the support
of the measure $\omega $ to extend into an othogonal subspace in a higher
dimension. There is an analogous theorem that generalizes Theorem \ref%
{general T1} in this way, but we will not pursue this here.
\end{remark}

\end{document}